\documentclass[10pt, a4paper]{amsart}

\usepackage{enumerate, mathtools, amssymb}
\usepackage[backref=page]{hyperref}
\hypersetup{
    colorlinks=true,
    linkcolor=blue,
    filecolor=blue,         
    urlcolor=blue,
    citecolor=blue}
\usepackage{amsmath}
\usepackage{amsthm}
\usepackage{amssymb}
\usepackage[matha,mathx]{mathabx} 
\usepackage[svgnames]{xcolor}
\usepackage{tikz}
\usepackage{tikz-cd}
\usepackage{subcaption}
\usepackage{url}
\usepackage{stmaryrd}
\usepackage{bbm}
\usepackage{appendix}
\usepackage{bm}
\usepackage{longtable}
\usepackage{tikz-qtree}
\usepackage{mathrsfs}
\usepackage[capitalise]{cleveref}
\usepackage{ytableau,graphicx, color}
\usepackage{marvosym}
\usepackage{enumitem}
\input xy
\xyoption{all}
\ytableausetup{boxsize=.37cm,centertableaux}

\newcommand\xqed[1]{%
  \leavevmode\unskip\penalty9999 \hbox{}\nobreak\hfill
  \quad\hbox{#1}}
\newcommand\exqed{\xqed{$\diamondsuit$}}

\usetikzlibrary{arrows, backgrounds, decorations.pathreplacing}

\numberwithin{equation}{section}
\numberwithin{figure}{section}

\definecolor{mygray}{gray}{0.7}

\newtheorem{thmabc}{Theorem}

\newtheorem{thm}{Theorem}[section]
\newtheorem*{thm*}{Theorem}
\newtheorem*{con*}{Conjecture}
\newtheorem{lem}[thm]{Lemma}

\newtheorem{prop}[thm]{Proposition}
\newtheorem{cor}[thm]{Corollary}
\newtheorem{lemma}[thm]{Lemma}
\newtheorem{conj}[thm]{Conjecture}

\crefname{thm}{Theorem}{Theorems}
\crefname{lem}{Lemma}{Lemmas}
\crefname{prop}{Proposition}{Propositions}
\crefname{equation}{Equation}{Equations}
\crefname{thmabc}{Theorem}{Theorems}
\crefname{ex}{Example}{Examples}
\crefname{conj}{Conjecture}{Conjectures}

\newcommand{\tud}{\textup{d}}
\newcommand{\oldPhi}{\mathsf{D}}
\newcommand{\oldPsi}{\mathsf{P}}

\theoremstyle{definition}\newtheorem{defn}[thm]{Definition}
\newtheorem{remark}[thm]{Remark}
\newtheorem{ex}[thm]{Example}

\newtheorem{problem}[thm]{Problem}

\newtheorem*{acknowledgements}{Acknowledgements}

\newcommand{\comp}[1]{#1^{\textup{c}}}

\DeclareMathOperator{\host}{HoSt}

\DeclareMathOperator{\Spec}{Spec}
\DeclareMathOperator{\GSp}{GSp}
\DeclareMathOperator{\den}{den}
\DeclareMathOperator{\num}{num}
\DeclareMathOperator{\GT}{GT}
\DeclareMathOperator{\Tr}{Tr}
\DeclareMathOperator{\Leg}{Leg}
\DeclareMathOperator{\SR}{SR}
\DeclareMathOperator{\pha}{phan}
\DeclareMathOperator{\SSYT}{SSYT}
\DeclareMathOperator{\rSSYT}{rSSYT}

\DeclareMathOperator{\tp}{peak}
\DeclareMathOperator{\bt}{vall}

\DeclareMathOperator{\wt}{wt}

\DeclareMathOperator{\sh}{sh}

\DeclareMathOperator{\HS}{HS}

\DeclareMathOperator{\Id}{Id}

\DeclareMathOperator{\maj}{maj}
\DeclareMathOperator{\Mat}{Mat}

\DeclareMathOperator{\Hilb}{Hilb}

\DeclareMathOperator{\GL}{GL}
\DeclareMathOperator{\SL}{SL}

\DeclareMathOperator{\diag}{diag}

\newcommand{\bfs}{\bm{s}}

\newcommand{\bfx}{\bm{x}}
\newcommand{\bfy}{\bm{y}}

\newcommand{\bfX}{\bm{X}}

\newcommand{\bfZ}{\bm{Z}}

\newcommand{\N}{\mathbb{N}}

\newcommand{\Q}{\mathbb{Q}}

\newcommand{\Z}{\mathbb{Z}}

\renewcommand{\phi}{\varphi}

\newcommand{\pbinom}[3][]{
    \genfrac{[}{]}{0pt}{}{#2}{#3}_{\ifthenelse{\isempty{#1}}{p}{#1}}
}
\renewcommand{\leq}{\leqslant}
\renewcommand{\geq}{\geqslant}

\renewcommand{\epsilon}{\varepsilon}

\newcommand{\Des}{\mathrm{Des}}

\newcommand{\msfQ}{\mathsf{Q}}
\newcommand{\msfS}{\mathsf{S}}
\newcommand{\msfT}{\mathsf{T}}

\newcommand{\lri}{\mathfrak{o}}

\newcommand{\define}[1]{\textbf{\textit{#1}}}

\newcommand{\sqsubsetdot}{\mathrel{\stackrel{\bullet}{\sqsubset}}}
  \def \fin {f^{\textup{in}}}
  \def \fpr {f^{\textup{pr}}}
  \def \bfz {{\bf 0}}
  \def \bfo {{\bf 1}}
\def \ol {\overline}

\def \mcD {\mathcal{D}}

\def \mcH {\mathcal{H}}
\def \mcL {\mathcal{L}}
\def \mcLin {\mathcal{L}^{\textup{in}}}
\def \mcLpr {\mathcal{L}^{\textup{pr}}}

\def \mcP {\mathcal{P}}

\def \Gri {\mathcal{O}}

\def \Z {\mathbb{Z}}
\def \Zp {\mathbb{Z}_p}
\def \Q {\mathbb{Q}}
\def \N {\mathbb{N}}
\def \mfp {\mathfrak{p}}

\newcommand{\mcEin}{\mathcal{E}^{\textup{in}}}
\newcommand{\mcEpr}{\mathcal{E}^{\textup{pr}}}

\newcommand{\affS}{\mathrm{affS}}

\newcommand{\Dif}{\mathrm{inc}}

\newcommand{\HLS}{\mathsf{HLS}}

\newcommand{\igusa}{\mathsf{I}}
\newcommand{\Hecke}{\mathsf{H}}
\newcommand{\Schubdim}[2]{d_{#1}(#2)}
\newcommand{\affSin}{\affS^{\textup{in}}}
\newcommand{\affSpr}{\affS^{\textup{pr}}}
\newcommand{\incr}[2]{\Dif_{#1}(#2)}

\newcommand{\finmod}[2]{C_{#1}(#2)}
\newcommand{\dual}[1]{\widetilde{#1}}

\newcommand{\flagsymbol}{V}
\newcommand{\intflag}{\flagsymbol^{\bullet}}
\newcommand{\projflag}{\flagsymbol_{\bullet}}
\newcommand{\intflagterm}[1]{\flagsymbol^{(#1)}}
\newcommand{\projflagterm}[1]{\flagsymbol_{(#1)}}

\makeatletter
\@namedef{subjclassname@2020}{2020 MSC}
\makeatother

\allowdisplaybreaks

\title[Affine Schubert series]{Hall--Littlewood polynomials,
  affine Schubert series, and lattice enumeration}

\author{Joshua Maglione}
\author{Christopher Voll}

\address{School of Mathematical and Statistical Sciences, University of Galway, Ireland}

\email{joshua.maglione@universityofgalway.ie}

\address{Fakult\"at f\"ur Mathematik, Universit\"at Bielefeld, D-33501
  Bielefeld, Germany}

\email{C.Voll.98@cantab.net}



\begin{document}

\begin{abstract}
  We introduce multivariate rational generating series called
  Hall--Littlewood--Schubert ($\HLS_n$) series. They are defined in terms of
  polynomials related to Hall--Littlewood polynomials and semi\-stan\-dard Young
  tableaux. We show that $\HLS_n$ series provide solutions to a range of
  enumeration problems upon judicious substitutions of their variables. These
  include the problem to enumerate sublattices of a $p$-adic lattice according
  to the elementary divisor types of their intersections with the members of a
  complete flag of reference in the ambient lattice. This is an affine analog of
  the stratification of Grassmannians by Schubert varieties. Other substitutions
  of $\HLS_n$ series yield new formulae for Hecke series and $p$-adic integrals
  associated with symplectic $p$-adic groups, and combinatorially defined quiver
  representation zeta functions. $\HLS_n$ series are $q$-analogs of Hilbert
  series of Stanley--Reisner rings associated with posets arising from parabolic
  quotients of Coxeter groups of type $\mathsf{B}$ with the Bruhat order.
  Special values of coarsened $\HLS_n$ series yield analogs of the classical
  Littlewood identity for the generating functions of Schur polynomials.
\end{abstract}

\date{\today} \maketitle

\thispagestyle{empty}

\tableofcontents{}

\section*{Introduction}

We offer a unifying framework for a wide variety of counting problems
from geometry, number theory, and algebra. To this end we introduce
\emph{Hall--Littlewood--Schubert series} $\HLS_n$; see
\Cref{def:HLS}. These are multivariate rational generating functions
defined as sums over semistandard Young tableaux (or just tableaux in
the sequel), involving polynomials related to Hall--Littlewood
polynomials.  We show that they specialize, under judicious
substitutions of their $2^n$ variables, to generating series solving
various counting problems.

What makes each of these problems amenable to Hall--Littlewood--Schubert series
is that they all factor over natural maps from the set of all finite-index
sublattices of a fixed lattice of finite rank $n$ to the infinite set $\SSYT_n$
of tableaux with entries from~$\{1,\dots, n\}$. In each case, the key to
reducing the respective counting problem to $\HLS_n$ is to compute the fibers of
the relevant map. En route we discover connections with further classical
objects of algebraic combinatorics, such as Dyck words, the Bruhat order, and
Stanley--Reisner rings. Three such instantiations, all related to lattice
enumeration, stand out.

(1) Let $V$ be a module over a compact discrete valuation ring~$\lri$, free of
finite rank~$n$, equipped with a complete flag of isolated submodules
$\{0\}=\intflagterm{0} \subsetneq \intflagterm{1} \subsetneq \intflagterm{2}
\subsetneq \dots \subsetneq \intflagterm{n} = V$. The \emph{affine Schubert
series} $\affSin_{n,\lri}$ introduced in \Cref{def:affS.int} enumerates
sublattices of finite index in $V$ by the elementary divisors of their
intersections with each of the lattices~$\intflagterm{i}$. This may be seen as
an affine analog of the classical concept of Schubert varieties, stratifying
Grassmannians by the intersection dimensions with a fixed complete flag in the
ambient vector space; see~\cite{Fulton/97}.  \Cref{thmabc:HLS.affS.int}
asserts that $\HLS_n$ specializes to $\affSin_{n,\lri}$ under a monomial
substitution of the variables.  \Cref{thmabc:HLS.affS.proj} is a similar
result for the affine Schubert series $\affSpr_{n,\lri}$, enumerating lattices
by the elementary divisors of their projections to, rather than intersections
with, the members of a complete flag of reference.

(2) \emph{Hecke series} play an important role in algebra and number
theory. \Cref{thmabc:HLS.hecke} shows that Hall--Littlewood--Schubert series
$\HLS_n$ specialize to the {Hecke series} associated with groups of symplectic
similitudes over local fields as studied by
Macdonald~\cite[Ch.~V]{Macdonald/95}. This leads to new formulae for and new
results about these classical series. As byproducts we prove, for instance,
conjectures raised in~\cite{PanchishkinVankov/07,Vankov/11}.

(3) \emph{Quiver representation zeta functions} enumerate
subrepresentations of integral quiver representations; see
\cite{LV/23}. Specializations of Hall--Littlewood--Schubert series
yield new and explicit formulae for these zeta functions associated
with combinatorially defined quiver representations over compact
discrete valuation rings; see~\Cref{thmabc:HLS.quiver}. Our work
brings ideas and tools from algebraic combinatorics to bear where
previously algebro-geometric methods dominated.

Our explicit formulae show that the
generating series associated with these lattice enumeration problems
depend on the local rings over which they are defined. This dependence is only mild and through their residue field cardinalities. More precisely, they all
turn out to be rational functions whose coefficients are polynomials
in the residue field cardinalities. This so-called uniformity is
reminiscent of the well-known fact that the numbers of rational points
of Schubert varieties over finite fields are given by integral
polynomials in the cardinalities of these fields.

Additional applications flow from the fact that $\HLS_n$ is a
$Y$-analog of the Hilbert series of the Stanley--Reisner ring of a
natural simplicial complex.  This is the order complex
$\Delta(\msfT_n)$ of the poset
$\msfT_n=2^{[n]}\setminus\{\varnothing\}$ equipped with the
\emph{tableaux order} introduced in \Cref{subsec:poset}. The poset
$\msfT_n$ may be interpreted in terms of the Bruhat order on parabolic
quotients of finite Coxeter groups of type~$\mathsf{B}$~\cite[Thm.~1]{Vince/00}.

We state a general
self-reciprocity result (\Cref{thmabc:HLS.funeq}) for the Hall--Littlewood--Schubert series $\HLS_n$ upon
inversion of their variables. Through the relevant variable
substitutions, self-reciprocity is passed on to the generating series described above, vastly extending the scope of this well-studied
symmetry phenomenon.
Our proof of \Cref{thmabc:HLS.funeq} is facilitated by interpreting
$\HLS_n$ in terms of $\mfp$-adic integrals. Conversely, we give
pleasing formulae for well-studied $\mfp$-adic integrals associated
with symplectic $\mfp$-adic groups in terms of Hall--Littlewood--Schubert series.

\section{Main objects and main results}\label{sec:main}

We defer precise definitions, even of standard objects pertaining to
partitions, tableaux, Young diagrams, lattices, and flags, to
\Cref{sec:not}. Throughout, let $n\in\N$.

\subsection{Hall--Littlewood--Schubert series}\label{subsec:HLS}

Write $\SSYT_n$ for the set of tableaux $T=(T_{ij})$ of degree~$n$,
i.e.\ with labels in~$[n] := \{1,\dots,n\}$. Write
$T=(C_1,\dots,C_\ell)$ to denote the columns of~$T\in\SSYT_n$. For
$i,j\in \N$ we define the \define{leg set} of~$T$:
\begin{align*}
  \Leg_T^+(i,j) &= \begin{cases}
    C_j \cap [T_{ij},\ T_{i (j+1)}] & \text{if } T_{i (j+1)}\notin C_j, \\ 
    \emptyset & \text{otherwise}.
  \end{cases}
\end{align*}
We set $\mathscr{L}_T=\left\{(i,j)\in\N^2 \mid \Leg^+_{T}(i,j)\neq\varnothing\right\}$.

\begin{defn}\label{def:PhiT} 
  The \define{leg polynomial} associated with $T\in\SSYT_n$ is
  \begin{align*}
    \Phi_T(Y) &= \prod_{(i,j)\in\mathscr{L}_T} \left(1 -
    Y^{\#\Leg_T^+(i,j)}\right) \in \Z[Y].
  \end{align*}
\end{defn}

We introduce further $2^n-1$ variables~$\bfX = (X_C)_{\varnothing \neq
  C \subseteq[n]}$. We call a tableau \define{reduced} if its columns
are pairwise distinct and write $\rSSYT_n$ for the finite (!)\ set of
reduced tableaux of degree~$n$.
  
\begin{defn}\label{def:HLS} 
 The \define{Hall--Littlewood--Schubert series} is
 \begin{equation*} 
   \HLS_n\left(Y,\bfX\right) = \sum_{T \in \rSSYT_n}\Phi_T(Y)
   \prod_{C\in T} \frac{X_C}{1-X_C}\in\Z[Y]\left(\bfX\right).
 \end{equation*}
\end{defn}

\begin{remark}
 The leg polynomial $\Phi_T$ coincides with a known polynomial invariant of
 Gelfand--Tsetlin patterns, which are known to be in bijection with tableaux;
 see \Cref{lem:FM}. For a partition $\lambda$, let $P_\lambda(\bfx;t)$ be the
 \emph{Hall--Littlewood polynomial}. In \Cref{equ:FM} we reproduce an
 expression, due to Feighin--Maklin, for $P_\lambda(\bfx;t)$ as a (finite)
 sum indexed by the tableaux $T\in \SSYT_n$ of shape~$\lambda$, involving both the leg
 polynomials $\Phi_T$ and the weights of the tableaux. By recording the $2^n-1$
 possible label sets of columns of tableaux of degree $n$, the
 Hall--Littlewood--Schubert series keeps track of much finer information.

 We note that leg sets index the cells contained in the leg of the
 $(i,j)$-cell for a suitable partition in Macdonald's terminology; see
 \cite[p.~337]{Macdonald/95}.
\end{remark}

We define the denominator polynomial
\begin{align*}
  \mathsf{D}_n(\bm{X}) &= \prod_{\varnothing \neq C \subseteq [n]} (1 - X_C) \in \Z[\bm{X}]. 
\end{align*}
We then define the numerator polynomial
$\mathsf{N}_n(Y,\bm{X})\in \Z[Y, \bm{X}]$ via
\begin{equation}\label{equ:HLS.num.den} 
  \HLS_n(Y,\bm{X}) = \dfrac{\mathsf{N}_n(Y,
    \bm{X})}{\mathsf{D}_n(\bm{X})}.
\end{equation}
  
\begin{ex}[$\HLS_n$ for $n\leq 3$]\label{exa:HLS.small}
  Given subsets $I_1,I_2,\ldots\subset \N$ we write
  $\bm{X}_{I_1|I_2|\dots}= X_{I_1}X_{I_2}\cdots$. We further simplify
  the subscripts by displaying only the sets' elements: for example,
  we write $X_{13}$ instead of~$X_{\{1,3\}}$.  For $n\leq 3$, we find
  \begin{align*}
    \mathsf{N}_1(Y,\bm{X}) = 1, & \quad 
    \mathsf{N}_2(Y,\bm{X}) = 1 - Y \bm{X}_{1|2},
  \end{align*} 
  and 
  \small{
  \begin{align*}
    \lefteqn{\mathsf{N}_3(Y,\bm{X}) = 1 - \bm{X}_{1|23}}\\ & -Y
    \left(\bm{X}_{1|2} + \bm{X}_{1|3} + \bm{X}_{2|3} + \bm{X}_{2|13} +
    \bm{X}_{12|13} + \bm{X}_{12|23} + \bm{X}_{13|23} + \bm{X}_{2|13|23}
      + \bm{X}_{1|2|13|23}\right) \\ &+Y \left(\bm{X}_{1|2|3} +
    \bm{X}_{1|2|13} + \bm{X}_{1|2|23} + \bm{X}_{1|3|23} +
    \bm{X}_{1|12|23} + \bm{X}_{1|13|23} + \bm{X}_{12|13|23}\right) \\
    &
    +Y^2 \left(\bm{X}_{1|2|3} + \bm{X}_{2|3|13} + \bm{X}_{1|3|13} + \bm{X}_{2|3|12|13}
    + \bm{X}_{3|12|13} + \bm{X}_{3|12|23} + \bm{X}_{12|23|13} \right) \\ & -Y^2 \left(\bm{X}_{3|12} + \bm{X}_{1|3|12} +\bm{X}_{1|2|3|12} + \bm{X}_{1|2|3|13} + \bm{X}_{1|3|12|23}+ \bm{X}_{1|12|13|23} \right. \\
    &\qquad \left.+ \bm{X}_{2|12|13|23} + \bm{X}_{3|12|13|23}\right)\\ & - Y^3 \left(\bm{X}_{2|3|12|13} -
    \bm{X}_{1|2|3|12|13|23}\right). 
  \end{align*}}

  \vspace{-1.75em}
  \exqed
\end{ex}

Our first main result establishes a general self-reciprocity property
for~$\HLS_n$.

\begin{thmabc}\label{thmabc:HLS.funeq} 
  We have
  \[
    \HLS_n(Y^{-1},\bfX^{-1}) = (-1)^{n} Y^{-\binom{n}{2}}X_{[n]} \cdot
    \HLS_n(Y,\bfX).
  \]
\end{thmabc}

(Self-)Reciprocity results as the one established in
\Cref{thmabc:HLS.funeq} are ubiquitous, but not universal, phenomena
seen in numerous counting problems in algebra, geometry and
combinatorics; see, for instance, \cite{BS/18} or~\cite{Voll/10}. As a
corollary we obtain reciprocity results for instantiations
of~$\HLS_n$. One such result is \Cref{cor:funeq.hecke}, which
establishes a functional equation for Fourier transforms of the Hecke
series associated with symplectic groups; see~\Cref{subsubsec:hecke}
for details. We prove \Cref{thmabc:HLS.funeq} in~\Cref{subsec:fun.eq}.

We now present the principal applications of
Hall--Littlewood--Schubert series to $\mfp$-adic lattice enumeration
problems as well as some of their combinatorial and topological
properties.

\subsection{Affine Schubert series}\label{subsubsec:affS}
  
Enumerating full lattices in $\Z^n$ by their index is a classical
problem with a well-known solution. The monograph
\cite{LubotzkySegal/03} lists no fewer than five proofs of the
following identity:
\begin{equation}\label{equ:zeta}
  \zeta_{\Z^n}(s) := \sum_{\Lambda\leq \Z^n}|\Z^n:\Lambda|^{-s} =
  \prod_{i=0}^{n-1}\zeta(s-i),
\end{equation}
where the sum runs over all lattices of finite index, $\zeta(s) =
\sum_{n=1}^\infty n^{-s}$ is the Riemann zeta function and $s$ is a complex
variable.

One way to prove \eqref{equ:zeta} is to enumerate matrices in Hermite
normal form; see \cite[Sec.~1]{BushnellReiner/80}. Its simplicity
notwithstanding, this approach has two drawbacks: it is
basis-dependent and is oblivious of an important set of intrinsic
invariants, namely the \emph{elementary divisors} of $\Lambda$ with
respect to the ambient lattice~$\Z^n$.

Enumeration of lattices by their elementary divisors is achieved
through suitable specializations of Igusa functions. The
\define{Igusa function of degree $n$} is the rational function in
variables $Z_1,\dots, Z_n$
\begin{align}\label{def:Igusa}
  \igusa_n(Y;Z_1,\dots,Z_n) &= \sum_{I
  \subseteq[n]}\binom{n}{I}_{Y}\prod_{i\in I}\frac{Z_i}{1-Z_i} \in
  \Z[Y](Z_1,\dots,Z_n).
\end{align}
Here, $\binom{n}{I}_Y\in\Z[Y]$ is the $Y$-multinomial coefficient.  The zeta
function in~\eqref{equ:zeta} satisfies the following Euler product decomposition
(see \cite[Ex.~2.20]{Voll/11}):
\begin{equation*}\label{equ:abelian}
  \zeta_{\Z^n}(s) = \prod_{p \textup{ prime}} \igusa_n\left(p^{-1}; \left(p^{i(n-i-s)}\right)_{i\in[n]}\right).
\end{equation*}

Hall--Littlewood--Schubert series may be seen as substantial generalizations of Igusa
functions. Indeed, one of their principal applications is to the
enumeration of lattices $\Lambda\leq\Z^n$ by the elementary divisors
of their intersections with all the members of a fixed
complete isolated flag of~$\Z^n$. As in the case of $\zeta_{\Z^n}(s)$,
it suffices to solve this problem locally for all primes~$p$, or
equivalently for lattices in~$\Zp^n$, where $\Zp$ is the ring of
$p$-adic integers. More generally, we consider lattices over a
compact discrete valuation ring (cDVR) $\lri$ of arbitrary
characteristic.

In this local setup, the relevant elementary divisors are encoded by
$n$ partitions, one for each intersection. More precisely, let
$\intflag = \left(\intflagterm{i}\right)_{i\in[n]}$ be a complete
isolated flag of~$\lri^n$; see~\eqref{def:flag}. For a lattice
$\Lambda \leq \lri^n$, denote the type of $\Lambda\cap\intflagterm{i}$ in
$\intflagterm{i}$ by the partition~$\lambda^{(i)}(\Lambda)$ of at most $i$
parts. It determines and is determined by its vector of increments
$\Dif\left(\lambda^{\bullet}(\Lambda)\right) =
\left(\Dif(\lambda^{(i)}(\Lambda)\right)_{i\in[n]}\in\N_0^{\binom{n+1}{2}}$;
see~\eqref{def:incr}. We introduce $\binom{n+1}{2}$ variables~$\bfZ =
(Z_{ij})_{1\leq j \leq i \leq n}$ and set
$\bfZ^{\Dif\left(\lambda^{\bullet}(\Lambda)\right)} =
\prod_{i=1}^n\bfZ_i^{\Dif(\lambda^{(i)}(\Lambda))}$.

\begin{defn}\label{def:affS.int}
 The \define{affine Schubert series of intersection type} is
  \begin{equation}\label{def:affS}
    \affSin_{n,\lri}(\bfZ) = \sum_{\Lambda\leq \lri^n} \bfZ^{\Dif\left(\lambda^{\bullet}(\Lambda)\right)} \in
    \Z\llbracket\bfZ\rrbracket,
  \end{equation}
  where the sum runs over all finite-index sublattices $\Lambda$
  of~$\lri^n$.
\end{defn}

\begin{remark}
 The term \emph{affine Schubert series} is a nod to the fact that the defining
 sum~\eqref{def:affS} may (up to a factor) be interpreted as the generating
 function of a natural-valued weight function on the vertices of the affine
 Bruhat--Tits building associated with the group~$\SL_n(K)$, where $K$ is the
 field of fractions of the cDVR~$\lri$. Indeed, homothety classes of lattices in
 $K^n$ form a natural model for the vertex set of the simplicial complex
 underlying this building. For an early exploitation of this perspective in the
 enumeration of lattices; see~\cite{Voll/04}. To what extent affine Schubert
 series are invariants of affine Schubert varieties remains an interesting open
 question.
\end{remark}

\Cref{thmabc:HLS.affS.int} shows that the affine Schubert series
$\affSin_{n,\lri}$ is a specialization of the
Hall--Littlewood--Schubert series~$\HLS_n$. Given $C\subseteq [n]$, we
set
\begin{equation}\label{def:schubert.dim}
  \Schubdim{n}{C} = \left(\sum_{i\in[n]\setminus C}i\right) -
  \binom{n-\#C+1}{2}.
\end{equation}
This is the dimension of the Schubert variety associated with $C$;
see~\cite[p.~1071]{KL/72}. We denote by $C(k)$ the $k$th smallest
member of $C$. Set $C(\#C+1) = n + 1$ and
\begin{equation} \label{def:ZnC}
  \bfZ_{n,C}=\prod_{k=1}^{\#C}~\prod_{\varepsilon =
  0}^{C(k+1)-C(k)-1}Z_{(C(k)+\varepsilon)k}.
\end{equation}
Note that the (total) degree of $\bfZ_{n,C}$ is $n+1-C(1)$.
\begin{thmabc}\label{thmabc:HLS.affS.int}
  For all cDVR $\lri$ with residue field cardinality~$q$ we have
  \[ 
    \affSin_{n,\lri}(\bfZ) = \HLS_n\left(q^{-1},\left(q^{\Schubdim{n}{C}}
    \bfZ_{n,C}\right)_C\right).
  \]
\end{thmabc} 

In particular, $\affSin_{n,\lri}(\bfZ)$ is a rational function in
$\bfZ$ whose coefficients are polynomials in~$q$.  We list these
functions for $n\leq 3$ in~\Cref{exa:affS.int}. Key to the proof of
\Cref{thmabc:HLS.affS.int}, which we complete
in~\Cref{subsec:HLS.affSin.proof}, is to enumerate lattices $\Lambda$
in $\lri^n$ by associated \emph{intersection tableaux} that, by
design, encode the information stored by the
partitions~$\Dif(\lambda^{(i)}(\Lambda))$ for~$i\in[n]$.

In \Cref{def:affS.proj} we define $\affSpr_{n,\lri}$, a function dual to
$\affSin_{n,\lri}$, recording the elementary divisor types of the projections
onto a flag of reference. The duality between the two affine Schubert series is
discussed in~\Cref{sec:jigsaw-complement}. By defining a combinatorial operation
on pairs of partitions whose skew diagram is a horizontal strip, we show that
$\affSpr_{n,\lri}(\bm{Z})$ and $\affSin_{n,\lri}(\bm{Z})$ are closely related;
see also~\Cref{thm:fnT}. \Cref{exa:affS.int,,exa:affS.proj} illustrate this
proximity for~$n\leq 3$.

\begin{thmabc}\label{thmabc:HLS.affS.proj} 
  For all cDVR $\lri$ with residue field cardinality~$q$ we have
  \begin{align*}
    \affSpr_{n,\lri}(\bfZ) &= 
    \HLS_n\left(q^{-1},\left(q^{\Schubdim{n}{[n] \setminus C}}\bfZ_{n,C}\right)_C\right).
  \end{align*}
\end{thmabc}

Combining \Cref{thmabc:HLS.funeq} with \Cref{thmabc:HLS.affS.int,thmabc:HLS.affS.proj} yields that the affine Schubert
series also satisfies the following self-reciprocity property:

\begin{cor}\label{cor:funeq.affS} We have
  \begin{align*}
    \left.\affSin_{n,\lri}(\bm{Z}^{-1})\right|_{q\to q^{-1}} &=
    (-1)^{n}q^{\binom{n}{2}}\left(\prod_{i=1}^nZ_{ii}\right) \cdot
    \affSin_{n,\lri}(\bfZ),\\ \left.\affSpr_{n,\lri}(\bm{Z}^{-1})\right|_{q\to
      q^{-1}} &=
    (-1)^{n}q^{\binom{n}{2}}\left(\prod_{i=1}^nZ_{ii}\right) \cdot
    \affSpr_{n,\lri}(\bfZ).
  \end{align*}
\end{cor}

\subsection{Hermite--Smith series}\label{subsubsec:HS} 

A further substitution of $\HLS_n$ pertains to the generating series
enumerating lattices in $\lri^n$ according to their elementary divisor
types and Hermite composition simultaneously. For the former, let
$\lambda(\Lambda)$ be the partition encoding the elementary divisor
type of a lattice $\Lambda\leq \lri^n$ (see
\Cref{subsec:lat.flags}). For the latter, recall that $\Lambda$ may be
represented by a matrix $M\in\Mat_n(\lri)$, whose rows record
coordinates of generators of $\Lambda$ with respect to some ordered
$\lri$-basis of $\lri^n$.  The coset $\GL_n(\lri)M$ comprises all such
matrices. Let
\[ 
  \delta(\Lambda)=(\delta_1(\Lambda), \dots, \delta_n(\Lambda))\in\N_0^n
\]
be the vector of valuations of the diagonal entries of any
upper-triangular matrix in $\GL_n(\lri)M$. The vector
$\delta(\Lambda)$ is in fact an invariant of $\Lambda$ and the flag
$\intflag$ whose $i$th member is generated by the first $i$
elements of the ordered basis. We thus call $\delta(\Lambda)$ the
\define{Hermite composition} of $\Lambda$ relative to~$\intflag$.

\begin{defn}\label{def:HS}
  For variables $\bfx = (x_1,\dots,x_n)$ and $\bfy=(y_1,\dots,y_n)$
  the \define{Hermite--Smith series} is
  $$
    \HS_{n,\lri}(\bfx,\bfy) = \sum_{\Lambda \leq \lri^n}
    \bfx^{\Dif(\lambda(\Lambda))} \bfy^{\delta(\Lambda)} \in \Z\llbracket \bfx,\bfy\rrbracket.
  $$
\end{defn}

The Hermite--Smith series was first defined in~\cite[Sec.~1.3]{AMV2/24} because
of its connection to the symplectic series, which we discuss in
\Cref{subsubsec:hecke}. For $C\subseteq [n]$, let $\bm{x}_C = \prod_{i\in C}x_i$
and $\bm{y}_C = \prod_{i\in C}y_i$. Hermite--Smith series are instantiations of
Hall--Littlewood--Schubert series:

\begin{thmabc}\label{thmabc:HLS.HS} 
  For $C\subseteq [n]$, set $C^*=\{n-i+1 \mid i\in C\}$. We have
  \begin{align*}
    \HS_{n,\lri}(\bfx,\bfy) &=
    \HLS_n\left(q^{-1}, \left(q^{\Schubdim{n}{C}}x_{\#C}\bm{y}_{C^*}\right)_C\right).
  \end{align*}
\end{thmabc}

In the proof of \Cref{thmabc:HLS.HS} we show that $\HS_{n,\lri}$ factors over
$\affSin_{n,\lri}$; see~\eqref{eqn:aff-sub}.
\Cref{thmabc:HLS.funeq,thmabc:HLS.HS} yield the functional equation result
of~\cite[Thm.~C]{AMV2/24}.

\Cref{thmabc:HLS.affS.int} shows that the Igusa function
\eqref{def:Igusa} is an instantiation of~$\HLS_n$.

\begin{cor}\label{cor:igusa.hs.hls} We have
 \[
    \igusa_n\left(Y, \left(Z_i\right)_{i\in[n]}\right) =
    \HLS_n\left(Y,\left(Y^{\Schubdim{n}{[n]\setminus
        C}}Z_{\#C}\right)_C\right).
  \]
In particular, the zeta function $\zeta_{\Z^n,p}(s)$ is equal to
{\small
  \begin{align*}
    \igusa_n\left(p^{-1},\left(p^{i(n-i)-is}\right)_{i\in[n]}\right) = \HS_{n,\lri}\left(\left(p^{-is}\right)_{i\in[n]},\bfo\right) = \HLS_n\left(p^{-1},\left(p^{\Schubdim{n}{C}-s\#C}\right)_C\right).
  \end{align*}
}
\end{cor}

\begin{remark}
  In \Cref{obs:weak.order} we observe that the \define{weak order zeta
    function} $I^{\textup{wo}}_n((X_C)_C)$ (see~\eqref{def:Iwo}) is
  also a specialization of~$\HLS_n$. Together with the Igusa
  functions~$\igusa_n$, they are extremal members of the family of
  \emph{generalized Igusa functions}, introduced and studied in
  \cite{CSV/24}. It would be of great interest to find a framework
  unifying generalized Igusa functions and Hall--Littlewood--Schubert series.
\end{remark}

\subsection{Symplectic Hecke series}\label{subsubsec:hecke}

The Hecke series $\tau(Z)$ and its Fourier transforms
$\hat{\tau}(\bfs,Z)$ associated with the groups of symplectic
similitudes $\textup{GSp}_{2n}(F)$ over a local field $F$ are the
focus of \cite[Sec.~V.5]{Macdonald/95}, where $Z$ and $\bfs =
(s_0,\dots, s_n)$ are variables. In \cite[V.5~(5.3)]{Macdonald/95}
Macdonald gives a formula for $\hat{\tau}(\bfs,Z)$ as a sum of $2^n$
rational functions in $Z$ and $q^{-s_0},\dots,q^{-s_n}$, where $q$ is
the residue field cardinality of the ring of integers $\lri$
of~$K$. For a variable $X$, Macdonald exhibits a function
\begin{equation}\label{equ:hecke}
  \Hecke_{n,\lri}(\bm{x}, X) = \dfrac{\Hecke_n^{\mathrm{num}}(q^{-1},
    \bm{x}, X)}{\prod_{I\subseteq [n]}(1 - \bm{x}_IX)}\in\Q(\bfx,X),
\end{equation}
where $\Hecke_n^{\mathrm{num}}(Y, \bm{x}, X)$ is a polynomial of
degree $2^n-2$ in $X$, that satisfies
\begin{equation}\label{def:hecke}
  \hat{\tau}(s_0,\dots, s_n,Z) = \Hecke_{n,\lri}(q^{-s_1},\dots, q^{-s_n}, q^{N-s_0}Z)
\end{equation}
for $N= \frac{1}{4}n(n+1)$. We extend the terminology
(\emph{symplectic}) \emph{Hecke series} to the rational functions
$\Hecke_{n,\lri}$. We show that they are substitutions of~$\HLS_n$.

\begin{thmabc}\label{thmabc:HLS.hecke} For all cDVR $\lri$ with residue field cardinality $q$ we have 
  \[
    \Hecke_{n,\lri}(\bm{x},X)(1 - X) =
    \HLS_n\left(q^{-1},\left(\bm{x}_CX\right)_{C}\right).
  \]
  In particular,
  \begin{align*}
    \Hecke_n^{\mathrm{num}}(Y, \bm{x}, X) &= \sum_{T\in \rSSYT_n} \Phi_T(Y) 
    \prod_{C\in T}\bm{x}_CX \prod_{\varnothing \neq I \not\in T}(1 - \bm{x}_IX) \in \Z[Y,\bm{x}, X].
  \end{align*}
\end{thmabc}

We list the numerator polynomials $\Hecke_n^{\textup{num}}(Y,\bfx,X)$
for $n\leq 3$ in~\Cref{exa:hecke}. Note that, in addition to providing
an alternative to Macdonald's expression, this formula explicates a
numerator of the rational function $\hat{\tau}(\bfs,X)$. It also
reveals additional properties of the~$\Hecke_{n,\lri}$: in
\Cref{prop:X=1.hecke} we record a simple multiplicative formula for
the special value of $\Hecke_{n,\lri}(\bfx,X)$ at $X=1$. Furthermore,
\Cref{thmabc:HLS.funeq,,thmabc:HLS.hecke} imply that the Hecke series
also satisfies a self-reciprocity property.

\begin{cor}\label{cor:funeq.hecke} 
  For all cDVR $\lri$ with residue cardinality $q$ we have
  \begin{align*}
    \left.\Hecke_{n,\lri}\left(\bm{x}^{-1}, X^{-1}\right)\right|_{q\to q^{-1}} &=
    (-1)^{n+1}q^{\binom{n}{2}}x_1\cdots x_nX^2 \cdot
    \Hecke_{n,\lri}(\bm{x}, X).
  \end{align*}
\end{cor}

One can show that the numerator polynomials $\mathsf{N}_{n}(Y,\bfX)$
in~\eqref{equ:HLS.num.den} have no linear term in $\bfX$, that is, the
coefficient, as an element of $\Z[Y]$, of $X_I$ is~$0$ for all non-empty
$I\subseteq [n]$. By \Cref{thmabc:HLS.hecke,,cor:funeq.hecke}, the coefficients
of $X$ and $X^{2^n-3}$ in $\Hecke_n^{\mathrm{num}}$ are both $0$, thereby
proving a conjecture of Panchishkin and Vankov concerning the Hecke series
$\tau(Z)$~\cite[Rem.~1.3]{PanchishkinVankov/07}. Moreover \Cref{cor:funeq.hecke}
proves a conjecture of Vankov~\cite[Rem.~4]{Vankov/11} concerning the
palindromicity of $\Hecke_n^{\mathrm{num}}$.

\subsection{Submodule and quiver representation zeta functions}
\label{subsubsec:quiver}
Submodule zeta functions generalize the zeta function $\zeta_{\Z^n}(s)$
introduced in \Cref{subsubsec:affS}. Whereas the latter enumerates all
sublattices in $\Z^n$, the former enumerate submodules of finite index that
are invariant under an integral matrix algebra.

Before we set out our contributions to this class of zeta functions, we
briefly sample a few of the milestones in the development of this class of
Dirichlet series. A classical prototype is Dedekind's zeta function associated
to a number field, enumerating ideals in the number field's ring of
integers. Solomon was interested in submodule zeta functions in the context of
integral representation theory; see \cite{Solomon/77}. Grunewald, Segal, and
Smith studied global and local submodule zeta functions associated with
nilpotent Lie rings in \cite{GSS/88}, pioneering tools from model theory and
$\mfp$-adic integration. An algebro-geometric approach was taken in du Sautoy
and Grunewald's seminal paper~\cite{duSG/00}. Rossmann turned a toroidal
vantage point into theoretical~\cite{Rossmann/15} and practical~\cite{Zeta}
advances. The second author studied submodule zeta functions associated with
nilpotent matrix algebras of class $2$ via affine Bruhat--Tits
buildings~\cite{Voll/04} and established self-reciprocity results akin to
\Cref{thmabc:HLS.funeq} in~\cite{Voll/10}.

Algebraic geometry, notably $\mfp$-adic integration, has been the prevalent
source of methodology in the development of the theory of submodule zeta
functions in the recent decade~\cite{Rossmann/18}. We argue that
Hall--Littlewood--Schubert series $\HLS_n$ are a powerful new tool in the
study of these and related Dirichlet generating series. We substantiate this
with \Cref{thmabc:HLS.quiver}, which we believe to be one of many instances of
this phenomenon.

As explained in \cite[Sec.~1.3.3]{LV/23}, submodule zeta functions are exactly
the zeta functions of integral quiver representations. To explain the latter,
recall that a \define{quiver} $\msfQ$ is a finite directed graph with vertex
set $Q_0$ and arrow set $Q_1$. For $\alpha \in Q_1$, write $h(\alpha)\in Q_0$
and $t(\alpha)$ for the respective head and tail of $\alpha$: if
$\alpha : i \to j$, then $h(\alpha) = i$ and $t(\alpha)=j$. Let $R$ be a
commutative ring. An \define{$R$-representation} of $\mathsf{Q}$ is a
collection $U=\left(U_\iota\right)_{\iota\in Q_0}$ of $R$-modules $U_i$,
together with an $R$-module homomorphisms
$f_{\alpha}:U_{t(\alpha)} \rightarrow U_{h(\alpha)}$ for each~$\alpha\in
Q_1$. An $R$-representation $U'$, with modules $U_i'$ and homomorphisms
$f_{\alpha}'$, is a \define{subrepresentation} of $U$ if $U_j'\leq U_j$ with
inclusion $\iota_j: U_j' \hookrightarrow U_j$ for all $j\in Q_0$ and
$f_{\alpha}\iota_j = \iota_k f_{\alpha}'$ for all arrows $\alpha : j\to k$. In
this case, we write $U'\leq U$. The \define{index} of $U'$ in $U$ is the
product of the indices $|U_i : U_i'|$ for each $i\in Q_0$.

The \emph{representation zeta function} $\zeta_{U}(\bfs)$ associated with a
fixed $R$-representation $U$ of a quiver $\mathsf{Q}$ was first introduced in
\cite{LV/23} for the case when $R$ is a global or local ring of integers and the
$U_i$ free, finite-rank $R$-modules; see \cite[(1.1)]{LV/23}. Let $\bm{s} =
(s_i)_{i\in Q_0}$ be complex variables. The representation zeta function is
defined as
\begin{equation}\label{def:rep.zeta}
  \zeta_{U}(\bfs) = \sum_{U'\leq U}\prod_{i\in Q_0}|U_i:U_i'|^{-s_i},
\end{equation}   
where the sum runs over finite index subrepresentations of
$U$. Certain substitutions of the rational functions $\HLS_n$ yield
concrete formulae for the (local) representation zeta functions of
various quiver representations. We exemplify this with certain
representations of dual star quivers.

For $n\in\N$, the \emph{dual star quiver} $\msfS_n^*$ is the quiver
with vertex set $[n]$ and arrows $\alpha_i : i \to n$ for all $i\in
[n-1]$. See \Cref{fig:Q3.star} for $n=4$. We define a representation
$V_n(\lri)$ of $\msfS_n^*$, as follows: let $V_i = \lri^i$ for all
$i\in [n]$, and let $f_{\alpha_i} : \lri^i \to \lri^n$ be an embedding
whose images form a complete isolated flag in $V_n=\lri^n$.

\begin{figure}[h]
  \centering
  \begin{tikzpicture}[
    ->,>=stealth'
  ]
	\node[inner sep=3, outer sep=0] (a) at (0,0) {$4$}; 
  \node[inner sep=3, outer sep=0] (b) at (90:1) {$1$}; 
  \node[inner sep=3, outer sep=0] (c) at (210:1) {$2$}; 
  \node[inner sep=3, outer sep=0] (d) at (330:1) {$3$};
	\path (b) edge (a);
  \path (c) edge (a);
  \path (d) edge (a);
	\end{tikzpicture}
  \caption{The dual star quiver $\msfS_4^*$}
  \label{fig:Q3.star}
\end{figure}

\begin{thmabc}\label{thmabc:HLS.quiver}
  For $C\subseteq [n]$, set $C_0=C\cup \{0\}$ and let
  $v_C=(\max(C_0\cap [i]_0))_{i=1}^n\in\N_0^n$. For the
  $\lri$-representation $V_n(\lri)$ of $\msfS^*_n$ as above, we have
  \begin{equation*}
    \zeta_{V_n(\lri)}(\bfs) = \HLS_n\left(q^{-1},\left(q^{\Schubdim{n}{C}-v_C\cdot \bm{s}}\right)_C\right)\prod_{i=1}^{n-1}\zeta_{\lri^i}(s_i).
  \end{equation*}
\end{thmabc}

We prove \Cref{thmabc:HLS.quiver} in \Cref{subsec:proof.quiver}.

\subsection{Tableaux and Bruhat orders}\label{subsec:tab.bru}

Hall--Littlewood--Schubert series are defined as finite sums over reduced
tableaux. Identifying this index set with the set of chains in a poset
opens further combinatorial and topological vantage points.

In \Cref{sec:poset}, we define a partial order $\sqsubseteq$ on the set
$\msfT_n$ of non-empty subsets of $[n]$ and explore the topological properties
of its associated order complex. In this poset structure, we compare non-empty
subsets $A$ and $B$ of $[n]$, written $A \sqsubseteq B$, if $A$ and $B$ arise as
labels of adjacent columns in some tableau $T\in \SSYT_n$. This refines the
usual containment relation $\supseteq$ on $[n]$: if $A \supseteq B$, then $A
\sqsubseteq B$. This partial order has been studied in different contexts and is
also known as the \emph{Gale order}; see~\cite{Gale/68}.

The order complex $\Delta(\msfT_n)$ associated with the poset~$\msfT_n$ is, as
an abstract simplicial complex, isomorphic to the set of $\rSSYT_n$ of reduced
tableaux with labels in $[n]$. We denote by $|\Delta(\msfT_n)|$ a geometric
realization of~$\Delta(\msfT_n)$.

\begin{thm}\label{thm:CM}
  The simplicial complex $|\Delta(\msfT_n)|$ is Cohen--Macaulay over~$\Z$ and
  homeomorphic to an $\left(\binom{n+1}{2}-1\right)$-ball. The number of maximal
  flags in $\Delta(\msfT_n)$ is
  \[ 
    \dfrac{\binom{n+1}{2}!\cdot\prod_{a=1}^{n-1}(a!)}{\prod_{b=1}^n
    ((2b-1)!)}.
  \] 
\end{thm}

We prove \Cref{thm:CM} in \Cref{subsubsec:CM}. At the heart of the proof is a
poset isomorphism between $\mathsf{T}_n$ and a poset arising from the parabolic
quotient of the hyperoctahedral group of degree $n$ by its maximal symmetric
group. The partial order on that set is given by the Bruhat order.

\subsection{Main ideas and structure of the paper}

\Cref{sec:not} sets up some essential notation. It contains a table of
notation and may serve as a reference throughout.

\Cref{sec:lat.tab.new} is central to the paper's methodology. We
associate with a full lattice $\Lambda$ in $V\cong \lri^n$ two types
of tableaux, viz.\ the intersection tableaux $T^{\bullet}(\Lambda)$
and the projection tableaux $T_{\bullet}(\Lambda)$, both
in~$\SSYT_n$. These tableaux encode how the lattice $\Lambda$ is built
up successively from the intersections and the projections relative to
the members of a flag in $V$ by cyclic extensions.

In \Cref{thm:fnT} we enumerate the fibers of these two maps in terms
of two combinatorial invariants of a tableau $T\in\SSYT_n$. The first is the
leg polynomial $\Phi_T(Y)$ introduced in \Cref{def:PhiT}, and the second is the
sum $D_n(T)$ of the dimensions of the Schubert varieties indexed by
the columns of~$T$; see~\Cref{def:D}. Since Hall--Littlewood--Schubert series are
defined as sums over tableaux that record their column structure
weighted by their leg polynomials, this paves the way to the proofs of
\Cref{thmabc:HLS.affS.int,thmabc:HLS.affS.proj} in
\Cref{subsec:HLS.affSin.proof,subsec:HLS.affSpr.proof} and of
\Cref{thmabc:HLS.HS} in \Cref{subsec:HS.affS.proof}. The proof of
\Cref{thmabc:HLS.quiver} is given in \Cref{subsec:proof.quiver}.

In \Cref{sec:tab.dyck} we give an interpretation of the leg
polynomials $\Phi_T(Y)$ in terms of Dyck words. It is logically
independent from any of the paper's main theorems and may be of
independent combinatorial interest.

Hall--Littlewood--Schubert series are $Y$-analogs of Stanley--Reisner
rings of a simplicial complex, namely the order complex
$\Delta(\msfT_n)$ of the poset $\msfT_n$ we define in
\Cref{subsec:poset}. The observation that $\Delta(\msfT_n)$ is
isomorphic to the poset~$\rSSYT_n$, the finite indexing set of the sum
defining~$\HLS_n$, is crucial for the rest of this section;
see~\Cref{lem:tab.flags}. In \Cref{prop:Bruhat-iso} we establish an
isomorphism between $\msfT_n\cup\{\varnothing\}$ and a parabolic quotient
of the hyperoctahedral group $B_n$ of degree~$n$. This is an important
waypoint towards the proof of~\Cref{thm:CM} in~\Cref{subsubsec:CM}. En
route we leverage the well-known bijection between tableaux and
Gelfand--Tsetlin patterns to obtain quantitative statements about
maximal reduced tableaux.

In \Cref{sec:hecke} we develop the connections between Hall--Littlewood--Schubert
series and symplectic Hecke series. Apart from our proof of
\Cref{thmabc:HLS.hecke}, we record in \Cref{subsec:hecke.at.1} a
simple multiplicative formula for the special value of the Hecke
series at~$X=1$ (see \Cref{prop:X=1.hecke}), generalizing Schur's
classical Littlewood identities.

Special values of Hall--Littlewood--Schubert series are the theme
of~\Cref{sec:HLS.coarse}. We focus on univariate series obtained by
setting all the $X_C$ to $X$ and $Y$ to one of $0$, $1$, or~$-1$.  In
the case $Y=0$ the Cohen--Macaulay property of certain (Stanley--Reisner)
rings implies the non-negativity of the relevant series' numerators;
see~\Cref{cor:depth}. In the other cases $Y=1$ or $Y=-1$, we formulate
non-negativity conjectures that seem to transcend the remit of
Stanley--Reisner rings of simplicial complexes; see \Cref{conj:depth,conj:HLS.Y=-1.coarse}.

In \Cref{sec:HLS.int} we explore different interpretations of
Hall--Littlewood--Schubert series as $\mfp$-adic integrals. In \Cref{subsec:sym.int}
we show that classical integrals over the integral $\mfp$-adic points
of groups of symplectic similitudes are instances of Hall--Littlewood--Schubert
series. This yields a simplified proof of a combinatorial identity for
these integrals in terms of Igusa functions, previously proven
in~\cite{BGS/22}. We use a different expression of $\HLS_n$ as a
$\mfp$-adic integral to prove \Cref{thmabc:HLS.funeq}
in~\Cref{subsec:HLS.funeq.proof}.

\section{Notation}\label{sec:not}

We recall some standard definitions and notation from the theories of integer
partitions, tableaux, and lattices.
We write $\N= \{1,2,\dots\}$, and for $I\subseteq \N$, set $I_0 = I
\cup \{0\}$.  For $a,b\in \N$, let $[a,b] = \{a, a+1,\dots, b\}$ and
$[a] = [1,a]$. For $C\subseteq [n]$ and $k\in \N$, let $C(k)$ be the
$k$th smallest member of~$C$. We also set $C(\#C+1) = n + 1$.  For a
variable $Y$, we set $\binom{n}{\varnothing}_Y = 1$, and for
$I\subseteq [n]$ with $k=\max(I)$, set
\[ 
  \binom{n}{I}_Y = \binom{n}{k}_Y\binom{k}{I\setminus\{k\}}_Y = \dfrac{(1 - Y^n)(1 - Y^{n-1}) \cdots (1 - Y^{n - k + 1})}{(1 - Y) (1 - Y^2)\cdots (1 - Y^k)} \cdot\binom{k}{I\setminus\{k\}}_Y. 
\] 
For $I\subseteq [n]$, set $\maj(I) = \sum_{i\in I}i$. For $u = (u_1,\dots, u_m)\in \Z^m$, set $|u| = \sum_{i=1}^mu_i\in\Z$. 

\subsection{Partitions}
\label{subsec:part.not}

A \define{partition} $\lambda$ is a weakly decreasing sequence
$(\lambda_i)_{i\in \N}$ such that each $\lambda_i\in \N_0$ and all but finitely
many $\lambda_i$ are zero. The \define{parts} of $\lambda$ are the positive
$\lambda_i$ for~$i\in [n]$. We denote by $\mathcal{P}_n$ the set of all
partitions with at most $n$ parts. We write $\lambda=(\lambda_1,\dots,
\lambda_n)$ to mean $\lambda = (\lambda_1, \dots, \lambda_n, 0, \dots)$.

A partition $\lambda\in\mathcal{P}_n$ determines and is determined by
its \define{Young diagram}, which is a collection of $|\lambda|$
cells, arranged in at most $n$ left-justified rows with $\lambda_i$
cells in the $i$th row, starting at the top and going down. The
\define{shape} of a Young diagram with $n$ rows is the partition
$\lambda$ such that $\lambda_i$ is the number of cells in row $i$. The
\define{conjugate} partition $\lambda'$ corresponds to the Young
diagram associated with $\lambda$, reflected along the main diagonal:
$\lambda_i'$ is the number of cells in the $i$th column.

Given $\lambda, \mu\in\mathcal{P}_n$, we write $\mu \subseteq \lambda$
if $\mu_i\leq\lambda_i$ for each $i\in [n]$. In this case, we may
superimpose the two Young diagrams by aligning the top left corners,
with the Young diagram of shape $\mu$ inside of the Young diagram of
shape $\lambda$. The \define{skew diagram} $\lambda- \mu$ is the
diagram obtained from the Young diagram with shape $\lambda$ by
removing all of the cells in $\mu$. A skew diagram is a
\define{horizontal strip} if there is at most one cell in each of its
columns. We write $\host_{n}$ for the set of pairs of partitions
$(\lambda,\mu)\in\mcP_n\times \mcP_{n-1}$ such that $\mu\subseteq
\lambda$ and $\lambda-\mu$ is a horizontal strip.

For $\lambda\in \mathcal{P}_n$, set
\begin{equation}\label{def:incr}
  \Dif(\lambda) = (\lambda_1 - \lambda_2,\ \dots,\ \lambda_{n-1} -
  \lambda_n,\ \lambda_n) = \left(\Dif_i(\lambda)\right)_{i\in[n]}\in
  \N_0^n.
\end{equation}
Given a sequence of partitions $\lambda^{\bullet} =
(\lambda^{(1)},\dots, \lambda^{(n)})$, we write
\[ 
  \Dif(\lambda^{\bullet}) = (\Dif(\lambda^{(1)}),\ \dots,\ \Dif(\lambda^{(n)})).
\]

\subsection{Tableaux}\label{subsec:tab}

A (\define{semistandard Young}) \define{tableau} $T$ is a Young diagram where
each cell is filled in with a natural number such that the values are
non-decreasing across each row and increasing down each column. For $i,j\in \N$,
the $(i,j)$-\define{cell} of $T$ is---if it exists---the cell in the $i$th row
and $j$th column; its entry is written~$T_{ij}$. The \define{shape} of $T$ is
that of its underlying Young diagram, written $\sh(T)$. We write
$T=(C_1,C_2,\dots)$ for subsets $C_j\subseteq[n]$ to denote the columns of~$T$,
where 
\[
  C_j=\{T_{ij}\mid i\in[n]\} = \{C_j(k)\mid k\in[\#C_j]\}\subseteq [n] .
\]
We write $C\in T$ to express that $C$ is a column of~$T$. The \define{weight} of
$T$ is the vector $\wt(T) = (\omega_1,\dots, \omega_n)\in \N_0^n$ if $T$ has
exactly $\omega_i$ cells with entry~$i$.  A tableau is \define{reduced} if it
contains no repeated columns. We write $\SSYT_n$ for the set of tableaux with
entries in~$[n]$ and $\rSSYT_n\subseteq \SSYT_n$ for the set of reduced
tableaux.

Tableaux are in bijection with flags of partitions where all the skew
diagrams associated with successive pairs of partitions are horizontal
strips. For $T\in\SSYT_n$ and $k\in [n]_0$, let $T^{(k)}$ be the
tableau obtained from $T$ by removing all cells with labels greater
than~$k$. The shape of $T^{(k)}$ is the partition $\lambda^{(k)}(T) =
(\lambda_{1}^{(k)},\dots, \lambda_{k}^{(k)})\in\mathcal{P}_k$. This
yields a flag of partitions, namely
\begin{align}\label{def:flop.tab}
  \lambda^\bullet(T): \quad () = \lambda^{(0)}(T) \subseteq \lambda^{(1)}(T)
  \subseteq \cdots \subseteq \lambda^{(n-1)}(T) \subseteq
  \lambda^{(n)}(T) = \lambda.
\end{align}
By the tableau condition, differences of successive pairs of these partitions
are horizontal strips. Conversely, given a flag of partitions
\[ 
  () = \lambda^{(0)}\subseteq \lambda^{(1)} \subseteq \cdots \subseteq
\lambda^{(n-1)} \subseteq \lambda^{(n)},
\] 
where each successive skew diagram $\lambda^{(i)} - \lambda^{(i-1)}$
is a horizontal strip, we obtain a tableau $T$ of shape
$\lambda^{(n)}$ by labelling cells in the $i$th horizontal strip by
$i$ for each~$i\in[n]$. These two constructions are mutually inverse.

\subsection{Lattices and isolated (co-)flags} \label{subsec:lat.flags}

Let $\lri$ be a cDVR of arbitrary characteristic and $K$ its field of
fractions. Let $\mfp\subset \lri$ be the unique maximal ideal and
$\pi\in \mfp$ a uniformizing element. We assume that the residue field
$\lri/\mfp$ has characteristic $p$ and cardinality~$q$, sometimes
written $\mathbb{F}_q$.

Fix throughout a free $\lri$-module $V$ of finite rank~$n$. An
\define{$\lri$-lattice} $\Lambda\leq V$ is a $\lri$-submodule of $V$ of full
rank $n$.  We write $\mcL(V)$ for the set of all such lattices. For each $i\in
[n-1]_0$, let $U_i$ be a free $\lri$-module. A \define{complete isolated flag}
is 
\[ 
  \begin{tikzcd}
  U_0 \arrow[r, hook, "\iota_1"] & U_1 \arrow[r, hook, "\iota_2"] & \cdots \arrow[r, hook, "\iota_n"] & U_n,
  \end{tikzcd}
\]
where each $U_i$ is an $\lri$-module with rank $i$ and the cokernel of
each $\iota_j$ is torsion-free. Dually, a \define{complete isolated
  coflag} is
\[ 
  \begin{tikzcd}
  U_0 & \arrow[l, two heads, "\varpi_0", swap] U_1 & \arrow[l, two heads, "\varpi_1", swap] \cdots & \arrow[l, two heads, "\varpi_{n-1}", swap] U_n,
  \end{tikzcd}
\]
where each $U_i$ is an $\lri$-module with rank $i$ and the coimage of each
$\varpi_j$ is torsion-free. Fix once and for all a complete isolated flag and a
complete isolated coflag of $V$:
\begin{equation}\label{def:flag}
  \begin{tikzcd}
    \intflag: & 0 = \intflagterm{0} \arrow[r, hook, "\iota_1"] & \intflagterm{1} \arrow[r, hook, "\iota_2"] & \cdots \arrow[r, hook, "\iota_n"] & \intflagterm{n} = V, \\[-1.5em]
    \projflag: & 0 = \projflagterm{0} & \arrow[l, two heads, "\varpi_0", swap] \projflagterm{1} & \arrow[l, two heads, "\varpi_1", swap] \cdots & \arrow[l, two heads, "\varpi_{n-1}", swap] \projflagterm{n} = V.
  \end{tikzcd}
\end{equation}
We assume, without loss of generality, that each $\intflagterm{i}$ is a submodule of
$V$ and each $\projflagterm{i}$ is a quotient of $V$.

For a partition $\lambda$, we define the finite $\lri$-module
$\finmod{\lambda}{\lri} = \bigoplus_{i\in\N} \lri / \mfp^{\lambda_i}.$
The (\define{elementary divisor}) \define{type} of a lattice
$\Lambda\in\mcL(V)$, written $\lambda(\Lambda)$, is the partition
$\lambda$ with $V/\Lambda \cong \finmod{\lambda}{\lri}$. Let $w\in
\finmod{\lambda}{\lri}$ and $k\in \N_0$. Write $v_{\mfp}(w) \geq k$ if
$w\in \mfp^k\finmod{\lambda}{\lri}$. For an $\lri$-module $M$, define
\begin{align*}
  \mathrm{Ann}_M(\mfp^k) &= \left\{m\in M ~\middle|~ \mfp^k m =0 \right\}.
\end{align*}
If $M=\finmod{\lambda}{\lri}$, then 
\begin{align*}
  \lambda_i' &= \dim_{\mathbb{F}_q}\left(\mfp^{i-1}M / \mfp^{i}M\right) = \dim_{\mathbb{F}_q}\left(\mathrm{Ann}_M(\mfp^i)/\mathrm{Ann}_M(\mfp^{i-1})\right).
\end{align*}
The proof for the following lemma is essentially due to
\cite[II.4~(4.12)]{Macdonald/95}.

\begin{lem}\label{lem:horizontal-strips}
  Let $\lambda,\mu\in\mathcal{P}_n$. Let $M$ be a finite $\lri$-module
  of type $\lambda$ and $N\leq M$ of type~$\mu$. If $M/N$ is cyclic,
  then $\mu\subseteq \lambda$ and $\lambda - \mu$ is a horizontal
  strip.
\end{lem}

\begin{proof}
  Since $N$ is a submodule, it follows that $\mu\subseteq \lambda$. Fix $k\in
  \N$. Note that $\mathrm{Ann}_N(\mfp^k) = N\cap \mathrm{Ann}_M(\mfp^k)$ and
  $\mfp \mathrm{Ann}_M(\mfp^k) \subseteq \mathrm{Ann}_M(\mfp^{k-1})$. Since
  $\mathrm{Ann}_N(\mfp^{k-1}) = \mathrm{Ann}_N(\mfp^k) \cap
  \mathrm{Ann}_M(\mfp^{k-1})$, we have
  \begin{align*}
    \dim_{\mathbb{F}_q}\left((\mathrm{Ann}_N(\mfp^k) + \mathrm{Ann}_M(\mfp^{k-1}))/\mathrm{Ann}_M(\mfp^{k-1}) \right) = \mu_k'.
  \end{align*}
  As
  $\dim_{\mathbb{F}_q}\left(\mathrm{Ann}_M(\mfp^k)/\mathrm{Ann}_M(\mfp^{k-1})\right)
  = \lambda_k'$, we have 
  \begin{align}\label{eqn:F_q-vs}
    \dim_{\mathbb{F}_q}\left(\mathrm{Ann}_M(\mfp^{k})/(\mathrm{Ann}_N(\mfp^k) + \mathrm{Ann}_M(\mfp^{k-1})) \right) = \lambda_k' - \mu_k'.
  \end{align}
  Since $\mathrm{Ann}_M(\mfp^{k})/\mathrm{Ann}_N(\mfp^k) \cong (N +
  \mathrm{Ann}_M(\mfp^{k}))/N$ and since $M/N$ is cyclic, the
  $\mathbb{F}_q$-vector space in \eqref{eqn:F_q-vs} is also cyclic. Hence,
  $\lambda_k' - \mu_k' \leq 1$ for all $k\in\N$.
\end{proof}

\subsection{Further notation}
We record further notation in the following table.

\begingroup \renewcommand\arraystretch{1.18}
\begin{longtable}{l|l|l}
  Symbol &
  Description & Reference \\ \hline \hline 
  $\mathcal{P}_n$ & partitions with at most $n$ parts &
  \Cref{subsec:part.not} \\ $\host_n$ & partitions $\mu\subseteq
  \lambda$ yielding horizontal strips &
  \Cref{subsec:part.not}\\ $\mcL(V)$ & set of full lattices in $V$ &
  \Cref{subsec:lat.flags}\\ $\delta(\Lambda)$ & Hermite composition
   & \Cref{subsubsec:HS}\\ $\SSYT_n$ & tableaux
  with labels in $[n]$ & \Cref{subsec:tab}\\ $\rSSYT_n$ & tableaux
  without repeated columns & \Cref{subsec:tab}\\ $\Phi_T(Y)$ & Leg
  polynomial &\Cref{def:PhiT}\\ $\HLS_n(Y,\bfx)$ &
  Hall--Littlewood--Schubert series & \Cref{def:HLS}\\ $\affSin_{n,\lri}(\bfZ)$ &
  affine Schubert series of intersection type &
  \Cref{def:affS.int}\\ $\affSpr_{n,\lri}(\bfZ)$ & affine Schubert
  series of projection type &
  \Cref{def:affS.proj}\\ $\Hecke_{n,\lri}(\bfx,X)$ & Fourier transform
  of Hecke series & \eqref{equ:hecke}\\ $\lambda^\bullet(T)$ & flag of
  partitions of tableau $T$ &
  \eqref{def:flop.tab}\\ $\lambda^\bullet(\Lambda)$ & flag of
  partitions of intersection types for $ \Lambda$ &
  \eqref{def:flop.lat.int}\\ $T^{\bullet}(\Lambda)$ & intersection
  tableau for $\Lambda$ &
  \Cref{subsec:int.data}\\ $\lambda_\bullet(\Lambda)$ & flag of
  partitions of projection types for $\Lambda$ &
  \eqref{def:flop.lat.proj}\\ $T_{\bullet}(\Lambda)$ & projection
  tableau for $\Lambda$ & \Cref{subsec:proj.data}\\ $\intflag$ &
  complete isolated flag of submodules of $V$ &
  \eqref{def:flag}\\ $\projflag$ & complete isolated coflag of
  quotients of $V$ & \eqref{def:flag}\\ $\mcLin_{T}(V)$ & lattices
  with intersection data given by $T$ &
  \eqref{def:LT.int}\\ $\mcLpr_{T}(V)$ & lattices with projection data
  given by $T$ & \eqref{def:LT.proj}\\ $\mcD$ & finite Dyck words &
  \Cref{sec:tab.dyck}\\ $\Schubdim{n}{C}$ & dimension of Schubert
  variety for $C\subseteq[n]$& \eqref{def:schubert.dim}\\ $D_k(T)$,
  $\comp{D}_k(T)$ & (dual) $k$th Schubert dimension &
  \Cref{def:D}\\ $\msfT_n$ & poset on
  $2^{[n]}\setminus\{\varnothing\}$ with tableau order &
  \Cref{subsec:poset}\\ $\Delta(\msfT_n)$ & order complex of $\msfT_n$
  & \Cref{subsec:poset}\\ $\SR_n$ & Stanley--Reisner ring associated
  with $\msfT_n$ & \Cref{subsec:HLS.Y=0}\\ $B_n$ & hyperoctahedral
  group of degree $n$ & \Cref{subsec:bruhat}\\ $\GT_n$ &
  Gelfand--Tsetlin patterns of degree $n$& \Cref{subsec:top-props}
\end{longtable}
\endgroup

\section{Cyclic extensions of lattices}
\label{sec:lat.tab.new}

In this section we prepare the groundwork for our application of
Hall--Littlewood--Schubert series to affine Schubert series of
intersection and projection type.  In both cases, we construct a
surjective map from the set $\mcL(V)$ of finite-index lattices of $V$
to the set $\SSYT_{n}$ of tableaux; see
\Cref{subsec:int.data,,subsec:proj.data}. In \Cref{sec:enum.latt.tab}
we provide formulae for the cardinalities of the fibers of these
maps. \Cref{subsec:cor,subsec:ext.latt} contain preliminary notation
and results on partitions and lattice extensions.

\subsection{Corners, gaps, and jigsaws}\label{subsec:cor}

Before we enumerate lattices by either intersection or projection
data, we establish a few combinatorial results concerning pairs of
partitions. We illustrate them throughout.

The $(i,j)$-cell of the Young diagram of shape $\lambda\in\mcP_n$ is a
\define{corner} if there is no $(i+1,j)$-cell and no $(i,j+1)$-cell in
the diagram. Corners have the form $(\lambda_a',a)$ for some $a\in \N$
and, equivalently, $(b,\lambda_b)$ for some~$b\in\N$.

Let $\sigma$ be a horizontal strip
and~$(\lambda,\mu)\in\host_{n}$; see~\Cref{subsec:part.not}. We write
\[ 
  C_{\lambda,\mu} = \left\{ (\mu_a', a) ~\middle|~
  a\in\N,\ \mu_a'=\lambda_a',\ \mu_{a+1}' < \lambda_{a+1}' \right\}
  \subseteq [n-1] \times [\mu_1].
  \]
 for the set of the corners of $\mu$ within columns not containing a
 cell of $\sigma$ and immediately preceding a column containing a cell
 from $\sigma$. The cells in $C_{\lambda,\mu}$ are, in other words,
 those immediately to the (south-)west of maximal contiguous substrips
 of~$\sigma$. Write $\pi_1$ for the projection onto the first
 coordinate of elements in $C_{\lambda,\mu}$ and $\pi_2$ for the
 projection onto the second coordinate.  Let us define the sets
\begin{align*}
  I_{\lambda,\mu} &= \pi_1(C_{\lambda,\mu}),  &
  J_{\lambda,\mu} &= \pi_2(C_{\lambda,\mu}). 
\end{align*}
 Since there is at most one corner in each row and column, $\#
 C_{\lambda,\mu} = \#I_{\lambda,\mu} = \#J_{\lambda,\mu}$.

We define two additional partitions $\nu,\gamma\in\mathcal{P}_{n-1}$,
whose parts are given by
\begin{align*}
  \nu_i &= \sum_{j>i} (\lambda_j - \mu_j), & \gamma_i &= \mu_i - \nu_i = \mu_i - \sum_{j>i} (\lambda_j - \mu_j).
\end{align*}
We note that $\nu_i$ is the number of cells of $\lambda$ in row $i$
lying above (but not on) a cell of the horizontal strip~$\sigma$,
whereas $\gamma_i$ is the number of such cells lying neither on nor
above a cell of~$\sigma$. We call $\nu$ the \define{valuation
  partition} and $\gamma$ the \define{gap partition} associated
with~$(\lambda,\mu)$. The number of cells in $\lambda$ lying neither
on nor above a cell of $\sigma$ is
denoted \begin{align}\label{eqn:gap} \mathrm{gap}(\lambda, \mu) &=
  |\gamma|.
\end{align}
Hence, $\mathrm{gap}(\lambda, \mu)$ counts the number of cells in columns of
$\lambda$ in the gaps between the contiguous strips of $\sigma$.

\begin{ex}\label{ex:many-partitions}
  For $n=6$, $\lambda = (9,8,7,6,2,1)$, and $\mu = (9,7,7,3,2)$, we
  illustrate the valuation and gap partitions in the Young diagram for
  $\lambda$ in~\Cref{fig:kitchen-sink}.  We color the cells of
  $\sigma:=\lambda-\mu$ in blue, the cells of $C_{\lambda,\mu} =
  \left\{ (3,7), (4,3) \right\}$ in green, and the cells above a
  cell of $\sigma$ in yellow. The number of yellow cells in row $i$
  is $\gamma_i$, and the number of white or green cells in row $i$ is
  $\nu_i$. We note that $\mathrm{gap}(\lambda,\mu)=13$. \exqed

  \begin{figure}[h]
    \centering
    \begin{tikzpicture}
      \pgfmathsetmacro{\w}{0.35}
      \foreach \i [count = \k] in {9,7,6,3,2}{
        \foreach \j in {1,...,\i}{
          \draw (\j*\w - \w, -\k*\w) -- ++(-\w,0) -- ++(0,\w) -- ++(\w,0) -- cycle;
        }
      }
      \node at (16*\w, -3*\w) {{\small $\begin{aligned}
        \lambda &= (9,8,7,6,2,1) \in\mathcal{P}_6, \\
        \mu &= (9,7,7,3,2)\in\mathcal{P}_5, \\
        \sigma &= (0, 1, 0, 3, 0, 1)\in\N_0^6, \\
        \nu &= (5, 4, 4, 1, 1)\in\mathcal{P}_5, \\
        \gamma &= (4, 3, 3, 2, 1)\in\mathcal{P}_5.
      \end{aligned}$}};
      \draw[fill=Gold] (0*\w, -1*\w) -- ++(-\w,0) -- ++(0,\w) -- ++(\w,0) -- cycle;
      \draw[fill=Gold] (0*\w, -2*\w) -- ++(-\w,0) -- ++(0,\w) -- ++(\w,0) -- cycle;
      \draw[fill=Gold] (0*\w, -3*\w) -- ++(-\w,0) -- ++(0,\w) -- ++(\w,0) -- cycle;
      \draw[fill=Gold] (0*\w, -4*\w) -- ++(-\w,0) -- ++(0,\w) -- ++(\w,0) -- cycle;
      \draw[fill=Gold] (0*\w, -5*\w) -- ++(-\w,0) -- ++(0,\w) -- ++(\w,0) -- cycle;
      \draw[fill=Gold] (3*\w, -1*\w) -- ++(-\w,0) -- ++(0,\w) -- ++(\w,0) -- cycle;
      \draw[fill=Gold] (4*\w, -1*\w) -- ++(-\w,0) -- ++(0,\w) -- ++(\w,0) -- cycle;
      \draw[fill=Gold] (5*\w, -1*\w) -- ++(-\w,0) -- ++(0,\w) -- ++(\w,0) -- cycle;
      \draw[fill=Gold] (3*\w, -2*\w) -- ++(-\w,0) -- ++(0,\w) -- ++(\w,0) -- cycle;
      \draw[fill=Gold] (4*\w, -2*\w) -- ++(-\w,0) -- ++(0,\w) -- ++(\w,0) -- cycle;
      \draw[fill=Gold] (5*\w, -2*\w) -- ++(-\w,0) -- ++(0,\w) -- ++(\w,0) -- cycle;
      \draw[fill=Gold] (3*\w, -3*\w) -- ++(-\w,0) -- ++(0,\w) -- ++(\w,0) -- cycle;
      \draw[fill=Gold] (4*\w, -3*\w) -- ++(-\w,0) -- ++(0,\w) -- ++(\w,0) -- cycle;
      \draw[fill=Gold] (5*\w, -3*\w) -- ++(-\w,0) -- ++(0,\w) -- ++(\w,0) -- cycle;
      \draw[fill=Gold] (7*\w, -1*\w) -- ++(-\w,0) -- ++(0,\w) -- ++(\w,0) -- cycle;
      \draw[fill=RoyalBlue] (5*\w, -4*\w) -- ++(-\w,0) -- ++(0,\w) -- ++(\w,0) -- cycle;
      \draw[fill=RoyalBlue] (4*\w, -4*\w) -- ++(-\w,0) -- ++(0,\w) -- ++(\w,0) -- cycle;
      \draw[fill=RoyalBlue] (3*\w, -4*\w) -- ++(-\w,0) -- ++(0,\w) -- ++(\w,0) -- cycle;
      \draw[fill=RoyalBlue] (7*\w, -2*\w) -- ++(-\w,0) -- ++(0,\w) -- ++(\w,0) -- cycle;
      \draw[fill=RoyalBlue] (0*\w, -6*\w) -- ++(-\w,0) -- ++(0,\w) -- ++(\w,0) -- cycle;
      \draw[fill=LimeGreen] (2*\w, -4*\w) -- ++(-\w,0) -- ++(0,\w) -- ++(\w,0) -- cycle;
      \draw[fill=LimeGreen] (6*\w, -3*\w) -- ++(-\w,0) -- ++(0,\w) -- ++(\w,0) -- cycle;
    \end{tikzpicture}
    \caption{Valuation and gap partitions for $(\lambda,
      \mu)\in\host_6$}
    \label{fig:kitchen-sink}
  \end{figure} 
\end{ex}

\begin{lem}\label{lem:nu-eta-max}
  Let $(\lambda,\mu)\in\host_{n}$ with associated valuation and gap partitions
  $\nu,\gamma$, and set~$m=\max(I_{\lambda,\mu})$. Suppose
  \begin{align*}
    a &= \min\left(\{k\in \N \mid \nu_k < |\sigma|\} \cup \{\infty\}\right), & 
    b &= \min\{k\in \N\mid \gamma_k=0\}.
  \end{align*}
  Then $a \geq b$ if and only if $C_{\lambda,\mu}=\varnothing$. If $a<b$, then
  $\nu_{b-1} = \nu_m$.
\end{lem}

\begin{proof}
  The case when $a=\infty$ is clear, so assume $b\leq a < \infty$. Then
  $\nu_{a-1} = |\sigma|$ if and only if $\lambda_k=\mu_k$ for all $k\in [a-1]$.
  Hence, $[a-1]\cap I_{\lambda,\mu}=\varnothing$. Since $\gamma_a=0$ and
  $\nu_{a-1} = |\sigma|$, we have $\lambda_a = |\sigma|$. Thus, $[\lambda_a]\cap
  J_{\lambda,\mu} =\varnothing$. Therefore, $C_{\lambda,\mu}=\varnothing$. 
  
  Conversely, suppose $C_{\lambda,\mu}=\varnothing$. If $|\sigma|=0$, then we
  are done, so suppose $|\sigma|\in \N$. Then there exists $k\in \N$ such that
  $\lambda_i'\neq \mu_i'$ for all $i\in [k]$ and $\lambda_j' = \mu_j'$ for all
  $j>k$. In other words, there exists $r\in \N$ such that $\lambda_i=\mu_i$
  for all $i\in [r-1]$ and $\lambda_r=|\sigma|$. Thus $\nu_{r-1}=|\sigma|$, so
  $\mu_r - \nu_r = \mu_r - (|\sigma|-\lambda_r+\mu_r) = 0$. Hence,
  $\gamma_r = 0$ and thus $a\geq b$.

  For the final claim, suppose $a<b$. Since $\gamma_b=0$ implies $\mu_b=\nu_b$,
  we have $m\leq b-1$, so assume $m<b-1$. Since $\gamma_{b-1}>0$, we have
  $0<\gamma_{b-1}-\gamma_b = \mu_{b-1} - \lambda_{b}$. Hence,
  $\mu_{b-1}>\lambda_b\geq \mu_b$, so that $(b-1, \mu_{b-1})$ is a corner of
  $\mu$. For $r=\mu_{b-1}$, we have $\lambda_r'=\mu_r'$. For $s=\mu_m$, we also
  have $\lambda_s'=\mu_s'$. By maximality of $m$, we have $\lambda_i'=\mu_i'$
  for all $i\in [r,s]$. Since $(m,\mu_m)$ and $(b-1, \mu_b)$ are corners, this
  implies $\lambda_j = \mu_j$ for all $j\in [m+1, b-1]$. Hence, 
  \begin{align*} 
    0 &= \sum_{j=m+1}^{b-1}(\lambda_j-\mu_j) = \nu_{m} - \nu_{b-1}. \qedhere
  \end{align*}
\end{proof}

Given $(\lambda,\mu)\in \host_n$, we define
$(\dual{\lambda},\dual{\mu})\in\mathcal{P}_n \times \mathcal{P}_{n-1}$ by
\begin{equation}\label{eqn:dual-partition}
  \begin{split}
    \dual{\lambda} &= (\lambda_1 - \lambda_n, \lambda_1 - \lambda_{n-1},\dots, \lambda_1 - \lambda_1), \\ 
    \dual{\mu} &= (\lambda_1 - \mu_{n-1}, \lambda_1 - \mu_{n-2}, \dots, \lambda_1 - \mu_1).
  \end{split}
\end{equation}
Let $\rho\in \mathcal{P}_n$ with $\rho_1=\rho_n=\lambda_1$. The
\define{jigsaw operation} in~\eqref{eqn:dual-partition} for $\lambda$
can be visualized by reflecting the skew diagram $\rho - \lambda$
vertically and horizontally and for $\mu$ with the same reflections
applied to the skew diagram $\rho - (\lambda_1, \mu_1,\dots,
\mu_{n-1})$. 

We write
$\dual{\lambda}'$ to mean $(\dual{\lambda})'$ and similarly
$\dual{\mu}'=(\dual{\mu})'$. For all $i\in [\lambda_1]$,
\begin{align}\label{eqn:conjugate-duals}
  \lambda_i' + \dual{\lambda}_{\lambda_1-i+1}' &= n, & \mu_i' + \dual{\mu}_{\lambda_1-i+1}' &= n-1.
\end{align}
This implies that $\dual{\lambda} - \dual{\mu}$ is a horizontal strip,
so $(\dual{\lambda},\dual{\mu})\in\host_n$. Moreover,
\begin{align}\label{eqn:dual-hs}
  \lambda_i' - \mu_i' &= 0 & &\Longleftrightarrow & \dual{\lambda}_{\lambda_1-i+1}' - \dual{\mu}_{\lambda_1-i+1}' = 1.
\end{align}
Hence $\mathrm{gap}(\dual{\lambda},\dual{\mu})$ is the number of cells in $\rho$ below a cell of $\sigma$, or equivalently
\begin{align}\label{eqn:dual-gap}
  \mathrm{gap}(\dual{\lambda},\dual{\mu}) &= n|\sigma| - |\nu|,
\end{align}
where $\nu$ is the valuation partition associated with $\lambda$ and $\mu$.

\begin{lem}\label{lem:dual-corners}
  The map $J_{\lambda,\mu}\to J_{\dual{\lambda},\dual{\mu}}$ given by $a\mapsto
  \lambda_1-a$ is a bijection, and 
  \[ 
    \incr{a}{\mu'} = \incr{\lambda_1-a}{\dual{\mu}'}.
  \]
\end{lem}

\begin{proof}
  Suppose $(\mu_a',a)\in C_{\lambda,\mu}$, so $\lambda_a' - \mu_a' =0$ and
  $\lambda_{a+1}' - \mu_{a+1}' = 1$. By \eqref{eqn:dual-hs}, 
  \begin{align*}
    \dual{\lambda}_{\lambda_1-a+1}' - \dual{\mu}_{\lambda_1-a+1}' &= 1, & \dual{\lambda}_{\lambda_1-a}' - \dual{\mu}_{\lambda_1-a}' &=0 .
  \end{align*}
  So $(\dual{\mu}_{\lambda_1-a}', \lambda_1-a)\in C_{\dual{\lambda},\dual{\mu}}$. Thus,
  $C_{\lambda,\mu}$ and $C_{\dual{\lambda},\dual{\mu}}$ are in bijection. By
  \eqref{eqn:conjugate-duals}, 
  \begin{align*}
    \incr{a}{\mu'} &= \mu_a' - \mu_{a+1}' = \dual{\mu}_{\lambda_1-a}' - \dual{\mu}_{\lambda_1-a+1}' = \incr{\lambda_1-a}{\dual{\mu}'}. \qedhere
  \end{align*}
\end{proof}

\begin{ex}\label{ex:dual-operation}
  We illustrate the jigsaw operation defined in
  \eqref{eqn:dual-partition} with the partitions from
  \Cref{ex:many-partitions}: $\lambda=(9,8,7,6,2,1)$ and
  $\mu=(9,7,7,3,2)$ yield $\dual{\lambda}=(8,7,3,2,1,0)$ and
  $\dual{\mu}=(7,6,2,2,0)$. \Cref{fig:duals} shows the pairs
  $(\lambda,\mu), (\tilde{\lambda},\tilde{\mu})\in\host_6$. There we
  color the cells of $\lambda-\mu$ in blue, $C_{\lambda,\mu}$ in
  green, $\dual{\lambda}-\dual{\mu}$ in purple, and
  $C_{\dual{\lambda},\dual{\mu}}$ in orange. \exqed
\end{ex}

\begin{figure}[h]
  \centering
  \begin{subfigure}{0.3\textwidth}
    \centering
    \begin{tikzpicture}
      \pgfmathsetmacro{\w}{0.35}
      \foreach \i [count = \k] in {9,9,9,9,9,9}{
        \foreach \j in {1,...,\i}{
          \draw (\j*\w - \w, -\k*\w) -- ++(-\w,0) -- ++(0,\w) -- ++(\w,0) -- cycle;
        }
      }
      \draw[fill=RoyalBlue] (5*\w, -4*\w) -- ++(-\w,0) -- ++(0,\w) -- ++(\w,0) -- cycle;
      \draw[fill=RoyalBlue] (4*\w, -4*\w) -- ++(-\w,0) -- ++(0,\w) -- ++(\w,0) -- cycle;
      \draw[fill=RoyalBlue] (3*\w, -4*\w) -- ++(-\w,0) -- ++(0,\w) -- ++(\w,0) -- cycle;
      \draw[fill=RoyalBlue] (7*\w, -2*\w) -- ++(-\w,0) -- ++(0,\w) -- ++(\w,0) -- cycle;
      \draw[fill=RoyalBlue] (0*\w, -6*\w) -- ++(-\w,0) -- ++(0,\w) -- ++(\w,0) -- cycle;
      \draw[fill=MediumOrchid] (1*\w, -6*\w) -- ++(-\w,0) -- ++(0,\w) -- ++(\w,0) -- cycle;
      \draw[fill=MediumOrchid] (2*\w, -5*\w) -- ++(-\w,0) -- ++(0,\w) -- ++(\w,0) -- cycle;
      \draw[fill=MediumOrchid] (6*\w, -4*\w) -- ++(-\w,0) -- ++(0,\w) -- ++(\w,0) -- cycle;
      \draw[fill=MediumOrchid] (8*\w, -2*\w) -- ++(-\w,0) -- ++(0,\w) -- ++(\w,0) -- cycle;
      \draw[fill=LimeGreen] (2*\w, -4*\w) -- ++(-\w,0) -- ++(0,\w) -- ++(\w,0) -- cycle;
      \draw[fill=LimeGreen] (6*\w, -3*\w) -- ++(-\w,0) -- ++(0,\w) -- ++(\w,0) -- cycle;
      \draw[fill=DarkOrange] (3*\w, -5*\w) -- ++(-\w,0) -- ++(0,\w) -- ++(\w,0) -- cycle;
      \draw[fill=DarkOrange] (7*\w, -3*\w) -- ++(-\w,0) -- ++(0,\w) -- ++(\w,0) -- cycle;
      \draw[ultra thick] (-1*\w, 0*\w) -- (8*\w, 0*\w) -- (8*\w, -6*\w) -- (-1*\w, -6*\w) -- cycle;
      \draw[ultra thick] (0*\w, -6*\w) -- ++(0*\w, 1*\w) -- ++(1*\w, 0*\w) -- ++(0*\w, 1*\w) -- ++(4*\w, 0*\w) -- ++(0*\w, 1*\w) -- ++(1*\w, 0*\w) -- ++(0*\w, 1*\w) -- ++(1*\w, 0*\w) -- ++(0*\w, 1*\w) -- ++(1*\w, 0*\w);
    \end{tikzpicture}
    \caption{$(\lambda,\mu)$ and $(\dual{\lambda},\dual{\mu})$}
  \end{subfigure}~%
  \begin{subfigure}{0.3\textwidth}
    \centering
    \begin{tikzpicture}
      \pgfmathsetmacro{\w}{0.35}
      \foreach \i [count = \k] in {9,7,6,3,2}{
        \foreach \j in {1,...,\i}{
          \draw (\j*\w - \w, -\k*\w) -- ++(-\w,0) -- ++(0,\w) -- ++(\w,0) -- cycle;
        }
      }
      \draw[fill=RoyalBlue] (5*\w, -4*\w) -- ++(-\w,0) -- ++(0,\w) -- ++(\w,0) -- cycle;
      \draw[fill=RoyalBlue] (4*\w, -4*\w) -- ++(-\w,0) -- ++(0,\w) -- ++(\w,0) -- cycle;
      \draw[fill=RoyalBlue] (3*\w, -4*\w) -- ++(-\w,0) -- ++(0,\w) -- ++(\w,0) -- cycle;
      \draw[fill=RoyalBlue] (7*\w, -2*\w) -- ++(-\w,0) -- ++(0,\w) -- ++(\w,0) -- cycle;
      \draw[fill=RoyalBlue] (0*\w, -6*\w) -- ++(-\w,0) -- ++(0,\w) -- ++(\w,0) -- cycle;
      \draw[fill=LimeGreen] (2*\w, -4*\w) -- ++(-\w,0) -- ++(0,\w) -- ++(\w,0) -- cycle;
      \draw[fill=LimeGreen] (6*\w, -3*\w) -- ++(-\w,0) -- ++(0,\w) -- ++(\w,0) -- cycle;
    \end{tikzpicture}
    \caption{$\phantom{\dual{\lambda}}(\lambda,\mu)\phantom{\dual{\lambda}}$}
  \end{subfigure}~%
  \begin{subfigure}{0.3\textwidth}
    \centering
    \begin{tikzpicture}
      \pgfmathsetmacro{\w}{0.35}
      \draw[fill=White, draw=White] (0*\w, -6*\w) -- ++(-\w,0) -- ++(0,\w) -- ++(\w,0) -- cycle;
      \foreach \i [count = \k] in {8,7,3,2,1}{
        \foreach \j in {1,...,\i}{
          \draw (\j*\w - \w, -\k*\w) -- ++(-\w,0) -- ++(0,\w) -- ++(\w,0) -- cycle;
        }
      }
      \draw[fill=MediumOrchid] (7*\w, -1*\w) -- ++(-\w,0) -- ++(0,\w) -- ++(\w,0) -- cycle;
      \draw[fill=MediumOrchid] (6*\w, -2*\w) -- ++(-\w,0) -- ++(0,\w) -- ++(\w,0) -- cycle;
      \draw[fill=MediumOrchid] (0*\w, -5*\w) -- ++(-\w,0) -- ++(0,\w) -- ++(\w,0) -- cycle;
      \draw[fill=MediumOrchid] (2*\w, -3*\w) -- ++(-\w,0) -- ++(0,\w) -- ++(\w,0) -- cycle;
      \draw[fill=DarkOrange] (1*\w, -4*\w) -- ++(-\w,0) -- ++(0,\w) -- ++(\w,0) -- cycle;
      \draw[fill=DarkOrange] (5*\w, -2*\w) -- ++(-\w,0) -- ++(0,\w) -- ++(\w,0) -- cycle;
    \end{tikzpicture}
    \caption{$(\dual{\lambda},\dual{\mu})$}
  \end{subfigure}
  \caption{An illustration of the jigsaw
    operation~\eqref{eqn:dual-partition}}
  \label{fig:duals}
\end{figure} 

\subsection{Extending elements for lattices}\label{subsec:ext.latt}
In this section we study ways to extend lattices of type $\mu\in\mathcal{P}_{n-1}$
to lattices of type $\lambda\in\mathcal{P}_n$ by adjoining suitable elements. For a
partition $\lambda$, we set $D_{\lambda} = \left\{ i\in \N \mid
\lambda_i > \lambda_{i + 1}\right\} = \{\lambda_j' \mid j\in \N\}$.
For $i\in D_{\lambda}$, we set $e_i = \incr{\lambda_i}{\lambda'}\in
\N$. Then
\[ 
  \finmod{\lambda}{\lri} = \bigoplus_{i\in D_{\lambda}} \finmod{\lambda_i}{\lri}^{e_i}. 
\] 
For each $i\in D_{\lambda}$, we define the projection
\begin{equation}\label{eqn:vartheta}
  \vartheta_{\lambda,i}: \finmod{\lambda}{\lri}\to
\finmod{\lambda_i}{\lri}^{e_i}, \quad 
  (w_1, w_2, \dots) \mapsto (w_{i-e_i+1}, w_{i-e_i+2}, \dots, w_i).
\end{equation}
Given a partition $\beta$, we generalize the projections
$\vartheta_{\lambda,i}$ to projections
\[ 
  \vartheta_{\lambda,i}^{\beta} : \finmod{\beta}{\lri} \to \bigoplus_{j = 1}^{e_i} \finmod{\beta_{i-j+1}}{\lri}
\]  
given by the same formula in~\eqref{eqn:vartheta}. Thus
$\vartheta_{\lambda,i}^{\lambda} = \vartheta_{\lambda,i}$. We have two specific
use cases for $\beta$: for intersection data, $\beta=\mu$, and for projection
data, $\beta=(|\lambda-\mu|^{(n-1)})$.

\begin{defn}\label{def:extending} 
  Let $(\lambda,\mu)\in\host_{n}$ with valuation partition
  $\nu\in\mathcal{P}_{n-1}$. An element $w\in \finmod{\beta}{\lri}$ is
  \define{$(\lambda,\mu)$-extending} if for all $i\in D_{\mu}$ and all
  $j\in I_{\lambda, \mu}$,
  \begin{align*}
    v_{\mfp}\left(\vartheta_{\mu,i}^{\beta}(w)\right) &\geq \nu_i, & v_{\mfp}\left(\vartheta_{\mu,j}^{\beta}(w)\right) &\leq \nu_j.
  \end{align*}
\end{defn}

We call the inequalities in \Cref{def:extending} ``Condition (1)'' and
``Condition (2)''.

\begin{lem}\label{lem:valuation-extending}
  Let $(\lambda,\mu)\in\host_{n}$ with $\mu_{n-1}\neq 0$. Let
  $\beta\in \mathcal{P}_{n-1}$ with $\beta_{n-1}\geq \lambda_n$. If
  $w\in \finmod{\beta}{\lri}$ is $(\lambda,\mu)$-extending, then
  $v_{\mfp}(w)\geq \lambda_n$ and either $\lambda_n=\mu_{n-1}$,
  $\mu=(\lambda_1,\dots, \lambda_{n-1})$, or $v_{\mfp}(w)= \lambda_n$.
\end{lem}

\begin{proof}
  Let $\nu,\gamma\in \mathcal{P}_{n-1}$ be the valuation and gap
  partitions associated with~$(\lambda,\mu)$. As $\mu_{n-1}\neq 0$, we
  have $n-1\in D_{\mu}$. Since $\min\{\nu_1,\dots, \nu_{n-1}\} =
  \nu_{n-1} = \lambda_n$, by Condition (1) of \Cref{def:extending} we
  have that $v_{\mfp}(w)\geq \lambda_n$.
  
  Assume that $\lambda_n\neq \mu_{n-1}$ and $\mu\neq (\lambda_1,\dots,
  \lambda_{n-1})$, so we will show that $v_{\mfp}(w)= \lambda_n$.  Note that
  $\lambda_n\neq \mu_{n-1}$ implies that 
  \[ 
    \gamma_{n-1} = \mu_{n-1} - \lambda_n >0, 
  \]  
  and $\mu\neq (\lambda_1,\dots, \lambda_{n-1})$ implies that 
  \begin{align*}
    \nu_{n-1} &= \lambda_n < |\lambda| - |\mu|.
  \end{align*}
  By \Cref{lem:nu-eta-max} with $a\leq n-1$ and $b=n$, we have
  $C_{\lambda,\mu}\neq\varnothing$. Moreover for $m= \max(I_{\lambda,\mu})$,
  \Cref{lem:nu-eta-max} also implies that $\nu_m = \nu_{n-1} = \lambda_n$. Since
  $\beta_n\geq\lambda_n$, we have $v_{\mfp}(\vartheta_{\mu,n-1}(w))=\lambda_n$
  by Condition (2) of \Cref{def:extending}. Hence, $v_{\mfp}(w)= \lambda_n$.
\end{proof}

\begin{lem}\label{prop:descension}
  Let $(\lambda,\mu)\in\host_{n}$. Let $\rho = (\lambda_1,\dots,
  \lambda_{n-1})$ and $\tau = (\mu_1,\dots, \mu_{n-2})$. Let
  $\beta\in\mathcal{P}_{n-1}$ and $\alpha=(\beta_1,\dots,
  \beta_{n-2})$. With $\varpi : \finmod{\beta}{\lri}\to
  \finmod{\alpha}{\lri}$ given by $(w_1,\dots, w_{n-1})\mapsto (w_1, \dots,
  w_{n-2})$, if $w\in \finmod{\beta}{\lri}$ is
  $(\lambda,\mu)$-extending, then $\varpi(w)$ is
  $(\rho,\tau)$-extending.
\end{lem}

\begin{proof}
  If either $\mu_{n-1} = 0$ or $\mu_{n-2} > \mu_{n-1} > 0$, then
  $D_{\tau}\subseteq D_{\mu}$. And in these cases, Condition (1) of
  \Cref{def:extending} holds, with $(\lambda,\mu)$ replaced by $(\rho,\tau)$.
  Assume, therefore, that $\mu_{n-2} = \mu_{n-1} > 0$, so $n-2\in D_{\tau}$ and
  $D_{\mu}=\left(D_{\tau}\setminus\{n-2\}\right) \cup \{n-1\}$. Set $e =
  \incr{\mu_{n-1}}{\mu'} \geq 2$. For $\varpi_{n-1} : \bigoplus_{j=1}^e
  \finmod{\beta_{n-j}}{\lri} \to \bigoplus_{j=2}^{e}
  \finmod{\alpha_{n-j}}{\lri}$ given by $(w_1,\dots, w_e)\mapsto (w_1,\dots,
  w_{e-1})$, the following diagram commutes.
  \begin{center}
    \begin{tikzcd}
      \finmod{\beta}{\lri} \arrow[r, "\vartheta_{\mu,n-1}^{\beta}"] \arrow[d, "\varpi"] & \bigoplus_{j=1}^e \finmod{\beta_{n-j}}{\lri} \arrow[d, "\varpi_{n-1}"] \\
      \finmod{\alpha}{\lri} \arrow[r, "\vartheta_{\tau,n-2}^{\alpha}"] & \bigoplus_{j=2}^{e} \finmod{\alpha_{n-j}}{\lri} 
    \end{tikzcd}
  \end{center}
  Hence, $v_{\mfp}(\vartheta_{\mu,n-1}^{\beta}(w)) \geq \lambda_n$ implies
  \begin{align*}
    v_{\mfp}((\varpi_{n-1}\vartheta_{\mu,n-1}^{\beta})(w)) &= v_{\mfp}((\vartheta_{\tau,n-2}^{\alpha}\varpi)(w))
  \geq \lambda_n.
  \end{align*}
  Since $\tau_{n-2} - \rho_{n-1} = \mu_{n-2}-\lambda_{n-1}\geq
  \mu_{n-1} - \lambda_n$, Condition (1) of \Cref{def:extending}, again
  with $(\lambda,\mu)$ replaced by $(\rho,\tau)$, follows in this
  case. Because $C_{\rho,\tau}\subseteq C_{\lambda,\mu}$, Condition
  (2) holds as required.
\end{proof}

\subsection{Lattices and their intersection tableaux}
\label{subsec:int.data}

Let $\Lambda\in\mathcal{L}(V)$. Recall from \Cref{subsubsec:affS} that
$\lambda^{(i)}(\Lambda)\in\mathcal{P}_i$ is the type of
$\intflagterm{i}/(\intflagterm{i}\cap\Lambda)$ for each $i\in [n]$. This yields
the \define{flag of partitions of intersection types} associated with
$\Lambda$
\begin{equation}\label{def:flop.lat.int}
  \lambda^{\bullet}(\Lambda): \quad () = \lambda^{(0)}(\Lambda)
  \subseteq \lambda^{(1)}(\Lambda) \subseteq \cdots \subseteq
  \lambda^{(n)}(\Lambda) = \lambda(\Lambda).
\end{equation}

\begin{lem}\label{prop:lattice-part-flag}
  For all $i\in [n]$, we have $\left(\lambda^{(i)}(\Lambda),
  \lambda^{(i-1)}(\Lambda)\right)\in\host_i$.
\end{lem}

\begin{proof}
  As $\intflagterm{i-1}\leq \intflagterm{i}$ with respective ranks $i-1$ and $i$, the
  quotient $\lri$-module $\intflagterm{i}/((\intflagterm{i}\cap \Lambda) + \intflagterm{i-1})$
  is cyclic. Since $\intflagterm{i} + (\intflagterm{i-1} + \Lambda) = \intflagterm{i} +
  \Lambda$ and $\intflagterm{i} \cap (\intflagterm{i-1} + \Lambda) = (\intflagterm{i} \cap
  \Lambda) + \intflagterm{i-1}$, we have
  \[ 
    \intflagterm{i}/((\intflagterm{i} \cap \Lambda) + \intflagterm{i-1}) \cong (\intflagterm{i} +
    \Lambda)/(\intflagterm{i-1} + \Lambda) .
  \] 
  Since $\lambda^{(j)}(\Lambda)$ is the type of
  $(\intflagterm{j} + \Lambda)/\Lambda$ and
  $(\intflagterm{i} + \Lambda)/(\intflagterm{i-1} + \Lambda)$ is cyclic, the skew
  diagram $\lambda^{(i)}(\Lambda) - \lambda^{(i-1)}(\Lambda)$ is a horizontal
  strip by \Cref{lem:horizontal-strips}.
\end{proof}

By \Cref{prop:lattice-part-flag} the flag $\lambda^{\bullet}(\Lambda)$ defines
a tableau as explained in~\Cref{subsec:tab}. We denote this tableau by $T^{\bullet}(\Lambda) \in \SSYT_n$ and call it the \define{intersection
tableau of $\Lambda$}. Given
$T\in\SSYT_n$ we set
\begin{align}\label{def:LT.int}
\mcLin_T(V) &= \left\{\Lambda \in \mcL(V) \mid T^{\bullet}(\Lambda) = T \right\}.
\end{align}
In \Cref{thm:fnT} we determine the cardinality $
f_{n,T}^{\mathrm{in}}(\lri) = \#\mcLin_T(V)$.  Thus
\begin{equation}\label{equ:affS.int.rewrite}
    \affSin_{n,\lri}(\bfZ) = \sum_{T\in\SSYT_n}
    f_{n,T}^{\mathrm{in}}(\lri)\cdot\bm{Z}^{\Dif(\lambda^\bullet(T))}.
\end{equation}

\begin{ex}\label{ex:int3}
  We consider an example in~$\SSYT_3$. Let
  \begin{center}
    \begin{tikzpicture}
      \node at (-1.4, 0.06) {$T = $};
      \node at (0,0) {
        \begin{tikzpicture}
          \pgfmathsetmacro{\w}{0.35}
          \foreach \i [count = \k] in {5,3}{
            \foreach \j in {1,...,\i}{
              \draw (\j*\w - \w, -\k*\w) -- ++(-\w,0) -- ++(0,\w) -- ++(\w,0) -- cycle;
            }
          }
          \node at (-0.5*\w, -0.5*\w) {{\footnotesize $1$}};
          \node at (0.5*\w, -0.5*\w) {{\footnotesize $1$}};
          \node at (1.5*\w, -0.5*\w) {{\footnotesize $1$}};
          \node at (2.5*\w, -0.5*\w) {{\footnotesize $2$}};
          \node at (3.5*\w, -0.5*\w) {{\footnotesize $3$}};
          \node at (-0.5*\w, -1.5*\w) {{\footnotesize $2$}};
          \node at (0.5*\w, -1.5*\w) {{\footnotesize $2$}};
          \node at (1.5*\w, -1.5*\w) {{\footnotesize $3$}};
        \end{tikzpicture}
      };
      \node at (6,0) {$\lambda^{\bullet}(T) ~:\quad ()\subseteq (3)\subseteq (4,2) \subseteq (5,3,0)$.};
    \end{tikzpicture}
  \end{center}
  The lattices in $\mcLin_T(\lri^3)$ correspond to
  $\GL_3(\lri)$-cosets of matrices of the form
  \[ 
    \left\{ \begin{pmatrix}
      \pi^2 & a_{12} & a_{13} \\
          & \pi^3 & a_{23} \\
          &     & \pi^3
    \end{pmatrix} \in \Mat_3(\lri) ~\middle|~ \begin{array}{c}
      v_{\mfp}(a_{12}) \geq 1,\ v_{\mfp}(a_{23}) = 2,\\ 
      v_{\mfp}(a_{13}) = 0, \\
      v_{\mfp}(a_{12} a_{23} - \pi^3 a_{13}) = 3
    \end{array}\right\} .
  \] 
  Thus, $f_{3,T}^{\mathrm{in}}(\lri) = (q-1)(q^2-q)(q^3-q^2)$, so the
  term in $\affSin_{3,\lri}(\bfZ)$ associated with $T$ is
  \[ 
    \pushQED{\qed}
    f_{3,T}^{\mathrm{in}}(\lri)\cdot\bm{Z}^{\Dif(\lambda^\bullet(T))} = q^6(1 - q^{-1})^3 \cdot Z_{11}^3Z_{21}^2Z_{22}^2Z_{31}^2Z_{32}^3Z_{33}^0. 
  \] 
In \Cref{sec:tab.dyck} we describe a combinatorial way to obtain the factor
$(1-q^{-1})^3$ in
$f_{3,T}^{\mathrm{in}}(\lri)\cdot\bm{Z}^{\Dif(\lambda^\bullet(T))}$ in terms of
an invariant of a Dyck word associated with the tableau~$T$; see
\Cref{ex:int3.rev}.\exqed
\end{ex}

\subsection{Cyclic extensions of lattices by intersection data}\label{subsec:in.tab}
For a partition $\lambda\in \mathcal{P}_{n}$ and a lattice
$\Lambda_0\in\mcL(\intflagterm{n-1})$, we set
\[ 
  \mcEin_{\lambda}(V, \Lambda_0) = \left\{ \Lambda \in\mcL(V) ~\middle|~
  \Lambda \cap \intflagterm{n-1} = \Lambda_0,\ \lambda(\Lambda) =\lambda \right\} .
\] 
The lattices in $\mcEin_{\lambda}(V, \Lambda_0)$ are rank-$n$ extensions of
the rank-$(n-1)$ lattice $\Lambda_0$, all of which have type~$\lambda$. The
formula for the cardinality of $\mcEin_{\lambda}(V, \Lambda_0)$ established
in~\Cref{thm:latt.ext.int} plays an important role in the determination of
$f_{n,T}^{\mathrm{in}}(\lri)$, the cardinality of $\mcLin_T(V)$ for
$T\in\SSYT_n$, in~\Cref{thm:fnT}.

\begin{prop}\label{characterize-extensions}
  Set $\mu = \lambda(\Lambda_0)$. If $\mcEin_{\lambda}(V,
  \Lambda_0)$ is non-empty, then $(\lambda,\mu)\in \host_n$.
\end{prop}

\begin{proof}
  Suppose $\Lambda
  \in\mcEin_{\lambda}(V, \Lambda_0)$. We have $(\Lambda +
  \intflagterm{n-1})/\Lambda \cong \intflagterm{n-1}/(\Lambda \cap \intflagterm{n-1})
  \cong \finmod{\mu}{\lri}$ and $V/\Lambda \cong \finmod{\lambda}{\lri}$. Since
  $V/(\Lambda + \intflagterm{n-1})$ is cyclic, by
  \Cref{lem:horizontal-strips} we conclude that $(\lambda, \mu) \in \host_n$.
\end{proof}

The following proposition is the key to enumerating the lattices in
$\mcEin_{\lambda}(V, \Lambda_0)$, and it establishes the connection
between $(\lambda,\mu)$-extending elements of $\finmod{\mu}{\lri}$ (see \Cref{def:extending}) and
lattices $\Lambda\in\mathcal{L}(V)$ of type~$\lambda$.

\begin{prop}\label{claim:u-type} 
  Let $(\lambda,\mu)\in\host_{n}$ and let $\Lambda_0\in
  \mcL(\intflagterm{n-1})$ with $\mu=\lambda(\Lambda_0)$. Let
  $v\in V\setminus (\mfp V + \Lambda_0)$ and $u\in \intflagterm{n-1}$, and
  set $\Lambda = \lri (\pi^{|\lambda|-|\mu|}v + u) + \Lambda_0$. The
  following are equivalent.
  \begin{enumerate}[label=$(\roman*)$]
    \item\label{proppart:type} $\lambda (\Lambda) = \lambda$.
    \item\label{proppart:ext} The coset $u + \Lambda_0$ in $\intflagterm{n-1}/\Lambda_0$ is
    $(\lambda,\mu)$-extending.
  \end{enumerate}
\end{prop}

\begin{proof} 
  We show that \ref{proppart:ext} implies \ref{proppart:type}. Write the $n$th term of
  $\lambda(\Lambda)$ as $\lambda_n(\Lambda)$. By \Cref{prop:descension}, it
  suffices to show that $\lambda_n(\Lambda)=\lambda_n$. If $\mu_{n-1} = 0$,
  then $\lambda_n=0$ since $\lambda-\mu$ is a horizontal strip. Hence,
  $V/\Lambda$ has rank less than $n$. Therefore, $\lambda_n(\Lambda)=0$ as
  needed, so we assume that $\mu_{n-1} > 0$.

  Let $\sigma = \lambda-\mu$, and observe that $\mu_{n-1}\geq \lambda_n$. Note
  that
  \begin{align}\label{eqn:nth} 
    \lambda_n(\Lambda) &= \min\left\{k\in \N_0~\middle|~ \exists v_0\in V\setminus \mfp V,\ \pi^kv_0\in \Lambda\right\}.
  \end{align} 
  Let $w = u + \Lambda_0\in \intflagterm{n-1}/\Lambda_0\cong \finmod{\mu}{\lri}$. Since
  $w$ is $(\lambda,\mu)$-extending, by \Cref{lem:valuation-extending}
  $v_{\mfp}(w)\geq \lambda_n$, and therefore by \eqref{eqn:nth},
  \[ 
    \lambda_n(\Lambda)\geq \lambda_n.
  \] 
  Also by \Cref{lem:valuation-extending}, we have three cases. First suppose
  $v_{\mfp}(w)=\lambda_n$. Thus there exists a $u'\in V\setminus \mfp V$ such that
  $\pi^{\lambda_n}u' + \Lambda_0 = w$. Set $u''=u' + \pi^{|\sigma|-\lambda_n}v\in
  V\setminus \mfp V$, so $\pi^{\lambda_n}u''\in \Lambda$. Hence,
  $\lambda_n(\Lambda)\leq n$ in this case. Now consider the second case where
  $\lambda_n=\mu_{n-1}$. This implies that $\gamma_{n-1} = \mu_{n-1} - \lambda_n =
  0$. Since $w$ is $(\lambda,\mu)$-extending, $\vartheta_{\mu,n-1}(w) = 0$. Hence,
  there exists such a $u'\in V\setminus \mfp V$ such that $\pi^{\mu_{n-1}}u' +
  \Lambda_0 = w$. As in the first case, this implies that $\lambda_n(\Lambda)
  \leq \lambda_n = \mu_{n-1}$. Now consider the final case where
  $\mu=(\lambda_1,\dots, \lambda_{n-1})$, so $\nu_{n-1} = |\sigma| = \lambda_n$.
  By \Cref{eqn:nth}, $\lambda_n(\Lambda)\leq |\sigma| = \lambda_n$, so
  $\lambda_n(\Lambda)=\lambda_n$ in all three cases as required.

  Now we show that \ref{proppart:type} implies \ref{proppart:ext}. Suppose $w$ is not
  $(\lambda,\mu)$-extending. Then there is a largest $i\in D_\mu$ such that
  $v_{\mfp}(\vartheta_{\mu,i}(w))<\nu_i$ or a largest $j\in I_{\lambda,\mu}$
  such that $v_{\mfp}(\vartheta_{\mu,j}(w)) > \nu_j$. In both cases, by applying
  the same inductive proof as above there is a $k\in [n]$ such that the value of
  $\lambda_k(\Lambda)$ is strictly smaller or larger than $\lambda_k$. (The
  value $k$ depends on the value of the problematic $v_{\mfp}(\vartheta_{\mu,
  i}(w))$ and its relation to the other $\nu_j$.) Hence, $\lambda(\Lambda)\neq
  \lambda$.
\end{proof}

\begin{thm}\label{thm:latt.ext.int}
  Let $(\lambda,\mu)\in\host_n$ and let $\Lambda_0\in
  \mcL(\intflagterm{n-1})$ of type $\mu$. Then
  \[ 
    \#\mcEin_{\lambda}(V, \Lambda_0) = q^{\mathrm{gap}(\lambda, \mu)}
    \prod_{a\in J_{\lambda,\mu}} \left(1-q^{-\incr{a}{\mu'}}\right) .
  \]
\end{thm}

\begin{proof}
  Let $\sigma = \lambda-\mu$ be the horizontal strip. Let $v\in V\setminus (\mfp
  V + \intflagterm{n-1})$, so $V = \lri v + \intflagterm{n-1}$. For $u\in V\setminus
  \intflagterm{n-1}$, set 
  \[ 
    \Lambda(u) = \lri u + \Lambda_0.
  \] 
  For all $u\in V$ there exist $\alpha\in \lri\setminus \mfp$, $f\in \N_0$, and
  $u''\in \intflagterm{n-1}$ such that 
  \begin{align}\label{eqn:u-form}
    u + \Lambda_0 &= \alpha \pi^{f} v + u'' + \Lambda_0.
  \end{align}
  By \Cref{prop:lattice-part-flag}, if $\lambda(V / \Lambda(u)) = \lambda$, then
  $f=|\sigma|$ in \eqref{eqn:u-form}. Therefore it suffices to count the number
  of $\Lambda_0$-cosets, $u+\Lambda_0$, in $V$ such that
  $\lambda(V/\Lambda(u))=\lambda$ as different choices of $v\in V\setminus (\mfp
  V + \intflagterm{n-1})$ are inconsequential. Thus, by \Cref{claim:u-type},
  $\#\mcEin_{\lambda}(V, \Lambda_0)$ is equal to the number of
  $(\lambda,\mu)$-extending elements of $\finmod{\mu}{\lri}$. For $e, m\in \N$
  and $f\in \N_0$ with $f\leq m$, we have $\# \mfp^f\finmod{m}{\lri}^e =
  q^{e(m-f)}$. For $i\in \N$, set $e_i=\incr{\mu_i}{\mu'}$. Therefore by
  \Cref{claim:u-type} and \eqref{eqn:gap}, we have
  \begin{align*}
    \#\mcEin_{\lambda}(V, \Lambda_0) = \prod_{i\in D_{\mu}}q^{e_i(\mu_i - \nu_i)} \prod_{i\in I_{\lambda, \mu}} (1 - q^{-e_i}) &= \prod_{i\in D_{\mu}}q^{e_i\gamma_i} \prod_{a\in J_{\lambda, \mu}} (1 - q^{-\incr{a}{\mu'}}) \\
    &= q^{\mathrm{gap}(\lambda, \mu)} \prod_{a\in J_{\lambda, \mu}} (1 - q^{-\incr{a}{\mu'}}) . \qedhere \end{align*}
\end{proof}

As consequence of \Cref{thm:latt.ext.int} the cardinality of
$\mcEin_{\lambda}(V, \Lambda_0)$ depends only on $q$, $\lambda$, and
$\lambda(V/\Lambda_0)$. To reflect this, we define
\begin{align}\label{eqn:ext-in}
  \mathrm{ext}^{\mathrm{in}}_{\lambda, \mu}(\lri) &= \#\mcEin_{\lambda}(V, \Lambda_0)
\end{align}
whenever $(\lambda,\mu)\in\host_n$ and $\mu$ is the type of
$\Lambda_0$.

\begin{ex}\label{exa:y.tab}
  To illustrate \Cref{thm:latt.ext.int} we compute
  $\mathrm{ext}^{\mathrm{in}}_{\lambda, \mu}(\lri)$ for a pair
  $(\lambda,\mu)\in\host_9$ by counting the matrices in $\Mat_9(\lri)$
  whose rows generate the lattices in question. \Cref{fig:big-tab}
  displays, beside $\lambda$ and $\mu$, the horizontal strip $\sigma
  = \lambda-\mu$, whose cells we color in blue, as well as the
  associated valuation partition $\nu$ and the gap
  partition~$\gamma$. Cells in $C_{\lambda,\mu}$ we color in green.

  \begin{figure}[h]
    \centering
    \begin{tikzpicture}
      \pgfmathsetmacro{\w}{0.35}
      \foreach \i [count = \k] in {9,8,6,6,6,4,3,2}{
              \foreach \j in {1,...,\i}{
                      \draw (\j*\w - \w, -\k*\w) -- ++(-\w,0) -- ++(0,\w) -- ++(\w,0) -- cycle;
              }
      }
      \draw[fill=RoyalBlue] (0*\w, -9*\w) -- ++(-\w,0) -- ++(0,\w) -- ++(\w,0) -- cycle;
      \draw[fill=RoyalBlue] (5*\w, -6*\w) -- ++(-\w,0) -- ++(0,\w) -- ++(\w,0) -- cycle;
      \draw[fill=RoyalBlue] (4*\w, -6*\w) -- ++(-\w,0) -- ++(0,\w) -- ++(\w,0) -- cycle;
      \draw[fill=RoyalBlue] (6*\w, -3*\w) -- ++(-\w,0) -- ++(0,\w) -- ++(\w,0) -- cycle;
      \draw[fill=RoyalBlue] (11*\w, -1*\w) -- ++(-\w,0) -- ++(0,\w) -- ++(\w,0) -- cycle;
      \draw[fill=RoyalBlue] (10*\w, -1*\w) -- ++(-\w,0) -- ++(0,\w) -- ++(\w,0) -- cycle;
      \draw[fill=RoyalBlue] (9*\w, -1*\w) -- ++(-\w,0) -- ++(0,\w) -- ++(\w,0) -- cycle;
      \draw[fill=LimeGreen] (8*\w, -2*\w) -- ++(-\w,0) -- ++(0,\w) -- ++(\w,0) -- cycle;
      \draw[fill=LimeGreen] (3*\w, -7*\w) -- ++(-\w,0) -- ++(0,\w) -- ++(\w,0) -- cycle;

      \node at (21*\w,-4*\w) {$\begin{aligned}
        \lambda &= (12,9,7,6,6,6,4,2,1)\in\mathcal{P}_9, \\
        \mu &= (9,9,6,6,6,4,4,2)\in\mathcal{P}_8, \\
        \sigma &= (3,0,1,0,0,2,0,0,1) \in\N_0^9, \\
        \nu &= (4,4,3,3,3,1,1,1) \in\mathcal{P}_8, \\
        \gamma &= (5,5,3,3,3,3,3,1) \in \mathcal{P}_8.
      \end{aligned}$};
    \end{tikzpicture}
    \caption{Data associated with an example $(\lambda,\mu)\in\host_9$.}
    \label{fig:big-tab}
  \end{figure}

  We count the cosets in $\GL_9(\lri)\backslash \Mat_9(\lri)$ with a
  representative of the form
  \begin{align*}
    M &= \left(\begin{array}{c|cccc}
      \pi^7 & \pi^4 w_2 & \pi^3 w_5 & \pi w_7 & \pi w_8 \\ \hline 
      & \pi^9\Id_2 & & & \\ 
      & & \pi^6\Id_3 & & \\
      & & & \pi^4\Id_2 & \\
      & & & & \pi^2\Id_1
    \end{array}\right) \in \Mat_{9}(\lri),
  \end{align*}
  where $v_{\mfp}(w_2)=0$, $v_{\mfp}(w_5)\in [3]_0$, $v_{\mfp}(w_7)=0$, and
  $v_{\mfp}(w_8)\in [1]_0$. Because we view the rows of such matrices as
  generating the lattices in $\mcEin_{\lambda}(\lri^9, \Lambda_0)$, where
  $\Lambda_0$ is the lattice generated by the lower right $8\times 8$ matrix
  in $M$, we need only count specific cosets for the $w_i$, namely,
  $w_2\in (\lri/\mfp^5)^2$, $w_5\in (\lri/\mfp^3)^3$,
  $w_7\in (\lri/\mfp^3)^2$, and $w_8\in \lri/\mfp$. Their respective
  contributions to the number of such cosets are $q^{10}(1-q^{-2})$, $q^9$,
  $q^6(1-q^{-2})$ and $q$. Hence,
  \begin{align*}
    \mathrm{ext}^{\mathrm{in}}_{\lambda, \mu}(\lri) &= q^{26}(1-q^{-2})^2.
  \end{align*}
  Note that $26 = 5+5+3+3+3+3+3+1 = \mathrm{gap}(\lambda,\mu)$ and
  $J_{\lambda,\mu}=\{4,9\}$ with 
  \begin{align*}
    \mu' &= \left(8, 8, 7, 7, 5, 5, 2, 2, 2 \right) \in \mathcal{P}_9, & \Dif(\mu') &= (0, 1, 0, \underline{2}, 0, 3, 0, 0, \underline{2}) \in \N_0^9.
  \end{align*}

  \vspace{-2em} \exqed
\end{ex}

\subsection{Lattices and their projection tableaux}\label{subsec:proj.data}

We dualize the approach followed in \Cref{subsec:int.data}. Recall
from \Cref{subsec:lat.flags} that $\projflag$ is a complete
isolated coflag with projections $\varpi_j: \projflagterm{j+1}
\twoheadrightarrow \projflagterm{j}$. For each $j\in [n-1]_0$, let $\hat{\varpi}_j = \varpi_j\varpi_{j+1} \cdots
\varpi_{n-1}$, so that $\hat{\varpi}_j : V \twoheadrightarrow \projflagterm{j}$. For $i\in [n]$, let
$\lambda_{(i)}(\Lambda)\in\mathcal{P}_i$ be the type of~$\intflagterm{i} /
\hat{\varpi}_i(\Lambda)$. This yields the \define{flag of partitions of
  projection types} associated with $\Lambda$, similar
to~\eqref{def:flop.lat.int}:
\begin{equation}\label{def:flop.lat.proj}
  \lambda_{\bullet}(\Lambda): \quad () = \lambda_{(0)}(\Lambda)
  \subseteq \lambda_{(1)}(\Lambda) \subseteq \cdots \subseteq
  \lambda_{(n)}(\Lambda) = \lambda(\Lambda).
\end{equation}

\begin{lem}\label{prop:proj-part-flag}
  For all $i\in [n]$, we have $\left(\lambda_{(i)}(\Lambda),
  \lambda_{(i-1)}(\Lambda)\right) \in \host_i$.
\end{lem}

\begin{proof}
  
  Since the coimage of each $\varpi_i$ is torsion-free, the coimage of
  $\hat{\varpi}_j$ is torsion-free. Let $U_{n-j} = \ker(\hat{\varpi}_j) \leq V$,
  which is therefore isolated in $V$, so the full pre-image of
  $\hat{\varpi}_j(\Lambda)$ is $\Lambda + U_{n-j}$. Hence the $U_{\bullet}$ form
  a complete isolated flag of $V$. For all $i\in [n]$,
  \begin{align*}
    (\Lambda + U_{n-i+1}) / (\Lambda + U_{n-i}) &\cong U_{n-i+1}/((U_{n-i+1} \cap \Lambda) + U_{n-i}),
  \end{align*}
  and is therefore cyclic (see the proof of \Cref{prop:lattice-part-flag}) and
  \begin{align*}
    \hat{\varpi}_i(\Lambda + U_{n-i+1}) / \hat{\varpi}_i(\Lambda + U_{n-i}) &\cong
    (\Lambda + U_{n-i+1}) / (\Lambda + U_{n-i}). 
  \end{align*} 
  By \Cref{lem:horizontal-strips}, the statement follows.
\end{proof}

By \Cref{prop:proj-part-flag} the flag $\lambda_{\bullet}(\Lambda)$
determines a tableau as explained
in~\Cref{subsec:tab}. We denote this tableau by $T_{\bullet}(\Lambda)$ and call it the
\define{projection tableaux of $\Lambda$}. Given $T\in\SSYT_n$ we set
\begin{align}\label{def:LT.proj}
  \mcLpr_T(V) &= \left\{\Lambda \in \mcL(V) \mid
  T_{\bullet}(\Lambda) = T \right\},
\end{align}
in analogy with $\mcLin_T(V)$ from~\eqref{def:LT.int}. In \Cref{thm:fnT} we
determine the cardinality $f_{n, T}^{\mathrm{pr}}(\lri) = \#
\mcLpr_T(V)$. The following is analogous to \Cref{def:affS.int};
see~\eqref{equ:affS.int.rewrite}.

\begin{defn}\label{def:affS.proj}
 The \define{affine Schubert series of projection type} is
  \[ 
    \affSpr_{n,\lri}(\bfZ) = \sum_{T \in
      \SSYT_n}f_{n,T}^{\mathrm{pr}}(\lri)\cdot\bfZ^{\Dif(\lambda^\bullet(T))}
    \in \Z\llbracket \bfZ\rrbracket.
  \] 
\end{defn}

We list these rational (!)\ functions for $n\leq 3$
in~\Cref{exa:affS.proj}.

\begin{ex}\label{ex:proj3}
  We revisit the tableau from \Cref{ex:int3}, namely,
  \begin{center}
    \begin{tikzpicture}
      \node at (-1.4, 0.06) {$T = $};
      \node at (0,0) {
        \begin{tikzpicture}
          \pgfmathsetmacro{\w}{0.35}
          \foreach \i [count = \k] in {5,3}{
            \foreach \j in {1,...,\i}{
              \draw (\j*\w - \w, -\k*\w) -- ++(-\w,0) -- ++(0,\w) -- ++(\w,0) -- cycle;
            }
          }
          \node at (-0.5*\w, -0.5*\w) {{\footnotesize $1$}};
          \node at (0.5*\w, -0.5*\w) {{\footnotesize $1$}};
          \node at (1.5*\w, -0.5*\w) {{\footnotesize $1$}};
          \node at (2.5*\w, -0.5*\w) {{\footnotesize $2$}};
          \node at (3.5*\w, -0.5*\w) {{\footnotesize $3$}};
          \node at (-0.5*\w, -1.5*\w) {{\footnotesize $2$}};
          \node at (0.5*\w, -1.5*\w) {{\footnotesize $2$}};
          \node at (1.5*\w, -1.5*\w) {{\footnotesize $3$}};
        \end{tikzpicture}
      };
      \node at (6,0) {$\lambda^{\bullet}(T) ~:\quad ()\subseteq (3)\subseteq (4,2) \subseteq (5,3,0)$.};
    \end{tikzpicture}
  \end{center}
  The lattices in $\mcLpr_T(\lri^3)$ correspond to
  $\GL_3(\lri)$-cosets of matrices of the form
  \[ 
    \left\{ \begin{pmatrix}
      \pi^3 & b_{12} & b_{13} \\
          & \pi^3 & b_{23} \\
          &     & \pi^2
    \end{pmatrix} \in \Mat_3(\lri) ~\middle|~ \begin{array}{c}
      v_{\mfp}(b_{12}) = 2,\ v_{\mfp}(b_{23}) \geq 1,\\ 
      v_{\mfp}(b_{13}) = 0, \\
      v_{\mfp}(b_{12} b_{23} - \pi^3 b_{13}) = 3
    \end{array}\right\} .
  \] 
  Thus, $f_{3,T}^{\mathrm{pr}}(\lri) = (q-1)^2(q^2-q)$, so the term
  associated with $T$ is
    
  \hfill $
    \displaystyle f_{3,T}^{\mathrm{pr}}(\lri)\cdot\bm{Z}^{\Dif(\lambda^\bullet(T))} = q^4(1 - q^{-1})^3 \cdot Z_{11}^3Z_{21}^2Z_{22}^2Z_{31}^2Z_{32}^3Z_{33}^0. 
  $ \exqed
\end{ex}

\subsection{Cyclic extensions of lattices by projection data}
\label{subsec:pr.tab}

We prove \Cref{thm:fnT} by induction on~$n$. For the
induction step we enumerate, in \Cref{thm:latt.ext.proj}, cyclic
extensions of lattices. For $\lambda\in\mathcal{P}_{n}$ and a lattice
$\Lambda_0\in\mcL(\projflagterm{n-1})$, we set
\[ 
  \mcEpr_{\lambda}(V,\Lambda_0) = \left\{ \Lambda \in\mcL(V) ~\middle|~ \varpi_{n-1}(\Lambda) = \Lambda_0,\ \lambda(\Lambda) = \lambda \right\} .
\]
In complete analogy with \Cref{characterize-extensions}, one proves
that the set $\mcEpr_{\lambda}(V,\Lambda_0)$ is empty unless $(\lambda, \mu)\in \host_n$, where $\mu$ is the type of~$\Lambda_0$. 

The following theorem shows that the formulae for
$\#\mcEpr_{\lambda}(V, \Lambda_0)$ and $\#\mcEin_{\lambda}(V,
\Lambda_0)$ are related by the jigsaw
operation~\eqref{eqn:dual-partition}; see \Cref{thm:latt.ext.int}.

\begin{thm}\label{thm:latt.ext.proj}
  Let $(\lambda,\mu)\in\host_n$ and let $\Lambda_0\in
  \mcL(\projflagterm{n-1})$ of type~$\mu$. Then
  \[ 
    \#\mcEpr_{\lambda}(V, \Lambda_0) = q^{\mathrm{gap}(\dual{\lambda},\dual{\mu})}\prod_{a\in J_{\dual{\lambda},\dual{\mu}}}
    (1-q^{-\incr{a}{\dual{\mu}'}}).
  \]
\end{thm}

\begin{proof}
  Set $\sigma = \lambda-\mu$ and $d=|\sigma|$, and let $v\in
  \ker(\varpi_{n-1}) \setminus \mfp V$. By \Cref{prop:proj-part-flag},
  for each lattice $\Lambda\in \mcEpr_{\lambda}(V, \Lambda_0)$, we
  have $\Lambda \cap \lri v = \mfp^{d}v$.  Since $V$ is free, there
  exists $W\leq V$ such that $V = W \oplus \lri v$. For each $i\in
  [n-1]$, define $u_i + \mfp^{d}v\in \lri v/\mfp^{d}v$ such that $u_i
  - v_i \in W + \mfp^{d}v$. Write $u = (u_1+\mfp^{d}v, \dots, u_{n-1}
  + \mfp^{d}v)\in (\lri v/\mfp^{d}v)^{n-1} \cong
  \finmod{d}{\lri}^{n-1}$. Thus, each $\Lambda\in \mcEpr_{\lambda}(V,
  \Lambda_0)$ gives rise to such an element $u$ by applying $V \mapsto
  V/(W + \mfp^d v)$ to the generators of $\Lambda$.  By an argument
  that is analogous to \Cref{claim:u-type}, we claim that these
  elements $u$ both characterize such lattices and are in bijection
  with the $(\lambda,\mu)$-extending elements of
  $\finmod{d}{\lri}^{n-1}$.
  
  It suffices to count the number of $(\lambda,\mu)$-extending
  elements in $\finmod{d}{\lri}^{n-1}$. Let $\nu\in\mathcal{P}_{n-1}$
  be the valuation partition. For $i\in \N$, set $e_i=\incr{\mu_i}{\mu'}$. Then by \eqref{eqn:dual-gap} and
  \Cref{lem:dual-corners},
  \begin{align*}
    \#\mcEpr_{\lambda}(V, \Lambda_0) = \prod_{i\in D_{\mu}}q^{e_i(d - \nu_i)} \prod_{i\in I_{\lambda, \mu}} (1 - q^{-e_i}) &= q^{\mathrm{gap}(\dual{\lambda},\dual{\mu})} \prod_{a\in J_{\dual{\lambda},\dual{\mu}}}
    (1-q^{-\incr{a}{\dual{\mu}'}}) . \qedhere
  \end{align*}
\end{proof}

In analogy with \eqref{eqn:ext-in} we define $\mathrm{ext}^{\mathrm{pr}}_{\lambda, \mu}(\lri) =
\#\mcEpr_{\lambda}(V, \Lambda_0)$
whenever $(\lambda,\mu)\in\host_n$ and $\mu$ is the type of
$\Lambda_0$.

\begin{ex}\label{exa:y.tab-proj}
  To illustrate \Cref{thm:latt.ext.proj}, we compute
  $\mathrm{ext}^{\mathrm{pr}}_{\lambda, \mu}(\lri)$ for the same
  $(\lambda,\mu)\in\host_9$ as in
  \Cref{exa:y.tab}. The relevant data is illustrated in \Cref{fig:big-tab}. We count the cosets in $\GL_9(\lri)\backslash \Mat_9(\lri)$ with a
  representative of the form
  \begin{align*}
    M &= \left(\begin{array}{cccc|c}
      \pi^9\Id_2 & & & & \pi^4 w_2 \\ 
      & \pi^6\Id_3 & & & \pi^3 w_5 \\
      & & \pi^4\Id_2 & & \pi w_7 \\
      & & & \pi^2\Id_1 & \pi w_8 \\ \hline 
      & & & & \pi^7 
    \end{array}\right) \in \Mat_{9}(\lri),
  \end{align*}
  where $v_{\mfp}(w_2)=0$, $v_{\mfp}(w_5)\in [3]_0$,
  $v_{\mfp}(w_7)=0$, and $v_{\mfp}(w_8)\in [1]_0$. Again we need only
  count specific cosets for the $w_i$, namely, $w_2\in
  (\lri/\mfp^3)^2$, $w_5\in (\lri/\mfp^4)^3$, $w_7\in
  (\lri/\mfp^6)^2$, and $w_8\in \lri/\mfp^6$. Their respective
  contributions to the number of such cosets are $q^{6}(1-q^{-2})$,
  $q^{12}$, $q^{12}(1-q^{-2})$ and $q^6$. Hence,
  \begin{align*}
    \mathrm{ext}^{\mathrm{pr}}_{\lambda, \mu}(\lri) &= q^{36}(1-q^{-2})^2.
  \end{align*}
  Note that $\dual{\lambda}=(11,10,8,6,6,6,5,3,0)$ and
  $\dual{\mu}=(10,8,8,6,6,6,3,3)$, so the gap partition is $(6,6,6,4,4,4,3,3)$,
  whose entry sum is indeed $36$. \exqed
\end{ex}

\section{Enumerating lattices by tableaux}
\label{sec:enum.latt.tab}

We build off of \Cref{sec:lat.tab.new} to enumerate lattices by their
tableaux data. Here, we see the leg polynomial of \Cref{def:PhiT} come
into play, and we show that it is equal to
$f_{n,T}^{\mathrm{in}}(\lri)$ and $f_{n,T}^{\mathrm{pr}}(\lri)$, up to
monomial factors in $q$. We conclude this section with the proofs of 
\Cref{thmabc:HLS.quiver,,thmabc:HLS.affS.int,,thmabc:HLS.HS,,thmabc:HLS.affS.proj}.

\subsection{Jigsaws and complements}\label{sec:jigsaw-complement}

We extend the jigsaw operation defined in~\Cref{subsec:cor}
in~\eqref{eqn:dual-partition} to tableaux. Recall from
\Cref{subsec:tab} that tableaux are equivalent to flags of partitions
whose consecutive skew diagrams are horizontal strips.

Suppose $T\in \SSYT_n$ with $\lambda^{(i)} = \lambda^{(i)}(T)$ for all $i\in
[n]$. Define 
\begin{align}\label{eqn:dual-flag}
  \dual{\lambda^{(i)}} &= \left(\lambda^{(n)}_1 - \lambda_i^{(i)}, \lambda^{(n)}_1 - \lambda_{i-1}^{(i)}, \dots, \lambda^{(n)}_1 - \lambda_1^{(i)}\right) \in\mathcal{P}_i.
\end{align}
As in \Cref{subsec:cor}, one shows that
$\left(\dual{\lambda^{(i)}},\dual{\lambda^{(i-1)}}\right)\in\host_i$ is a horizontal strip for all
$i\in [n]$. Write $\dual{T}$ for the tableau obtained from the flag
$\dual{\lambda^{\bullet}}$ defined by~\eqref{eqn:dual-flag}.

With $T=(C_1,\dots, C_{\ell})\in \SSYT_n$, we define the \define{complement
tableau} to be 
\begin{align*}
  \comp{T} &= \left([n]\setminus C_{\ell},\ [n]\setminus C_{\ell-1},\ \dots,\ [n] \setminus C_1\right),
\end{align*}
where we truncate entries equal to $\varnothing$ arising from columns
$C_i = [n]$.

\begin{prop}\label{prop:complements}
  For all $T\in\SSYT_n$ we have $\dual{T} = \comp{T}$.
\end{prop}

\begin{proof}
  Let $\lambda^{(k)}=\lambda^{(k)}(T)=\sh(T^{(k)})$. Then $\dual{\lambda^{(k)}}$
  is obtained from the skew diagram $\rho -
  (\lambda_1^{(n)},\dots,\lambda_1^{(n)},\lambda^{(k)}_1,\dots,
  \lambda^{(k)}_k)$ by applying a horizontal and vertical reflection, where
  $\rho=((\lambda_1^{(n)})^n)\in\mathcal{P}_n$. Therefore, for $k\in[n-1]$ and
  all $i\in [\lambda^{(k)}_1]$, 
  \begin{align*}
    (\lambda_i^{(k+1)})' +  (\dual{\lambda^{(k+1)}})_i' &= k+1, & (\lambda_i^{(k)})' +  (\dual{\lambda^{(k)}})_i' &= k.
  \end{align*}
  Hence, the conclusion holds.
\end{proof}

Recall that $\Schubdim{n}{C}$ is the dimension of the Schubert variety
associated with $C\subseteq[n]$; see~\eqref{def:schubert.dim}.

\begin{defn}\label{def:D}
  For $T\in \SSYT_n$ and $k\in [n]$, the \define{$k$th Schubert
    dimension} and the \define{$k$th dual Schubert
    dimension} of $T$ are 
\begin{align*}
  D_k(T) &= \sum_{C\in T^{(k)}} \Schubdim{k}{C}, & \comp{D}_k(T) &= \sum_{C\in T^{(k)}} \Schubdim{k}{[k]\setminus C} = D_k(\comp{T}).
\end{align*}
\end{defn}

We now relate the gap statistic from \Cref{subsec:cor} with the sum of the
dimensions of the Schubert varieties associated with the columns of $T$.

\begin{lem}\label{lem:Schubert-gaps}
  For $T\in \SSYT_n$ and $k\in [n-1]$, we have
  \begin{align*}
    D_{k+1}(T) &= D_k(T) + \mathrm{gap}\left(\lambda^{(k+1)}(T), \lambda^{(k)}(T)\right), \\ \comp{D}_{k+1}(T) &= \comp{D}_k(T) + \mathrm{gap}\left(\dual{\lambda^{(k+1)}}(T), \dual{\lambda^{(k)}}(T)\right) . 
  \end{align*}
\end{lem}

\begin{proof}
  By \Cref{prop:complements}, it suffices to prove the first equality. Suppose
  $T$ has $\ell$ columns, and write $T^{(k)} = (C_1,\dots, C_{\ell})$ and
  $T^{(k+1)} = (C_1',\dots, C_{\ell}')$. For each $j\in [\ell]$, 
  \begin{align*}
    \Schubdim{k+1}{C_j} - \Schubdim{k}{C_j'} &= \maj([k+1]\setminus C_j) - \maj([k]\setminus C_j') \\
    &\quad + \binom{n - \#C_j'}{2} - \binom{n - \# C_j + 1}{2} . 
  \end{align*}
  We have two cases to consider. First suppose $C_j=C_j'$, so $k+1$ is not
  contained in $C_j$, and $\Schubdim{k+1}{C_j} - \Schubdim{k}{C_j'} = \#C_j$. Now assume that
  $C_j = C_j' \cup \{k+1\}$. Then $\Schubdim{k+1}{C_j} - \Schubdim{k}{C_j'} = 0$. Therefore,
  $D_{k+1}(T) - D_{k}(T)$ is equal to the number of cells in columns without
  cells labeled $k+1$ whose entries are less than $k+1$. Hence,
  \begin{align*}
    D_{k+1}(T) - D_{k}(T) &= \mathrm{gap}\left(\lambda^{(k+1)}(T), \lambda^{(k)}(T)\right) . \qedhere
  \end{align*}
\end{proof}

We note that for $C\subseteq [n]$ with $k=\#C$, the sum of the
Schubert dimensions $\Schubdim{n}{C} + \Schubdim{n}{[n]\setminus C}$ is equal to
$k(n-k)$, the dimension of the Grassmannian of $k$-dimensional
subspaces in an $n$-dimensional vector space. For $\sh(T)=\lambda$,
\begin{align*}
  D_n(T) + \comp{D}_n(T) &= \sum_{i \geq 1} \lambda_i'(n - \lambda_i').
\end{align*}

We now express the leg polynomial of $T$ in terms of the
leg polynomial of $T^{(n-1)}$.

\begin{lem}\label{lem:Phi-quotient}
  Let $T\in \SSYT_n$. Write $(\lambda, \mu) = (\sh(T),
  \sh(T^{(n-1)}))\in\host_n$.  \begin{align*}
    \dfrac{\Phi_T(Y)}{\Phi_{T^{(n-1)}}(Y)} &= \prod_{a \in J_{\lambda,
        \mu}}(1 - q^{-\incr{a}{\mu'}}) = \prod_{a \in
      J_{\dual{\lambda}, \dual{\mu}}}(1 -
    q^{-\incr{a}{\dual{\mu}'}}).
  \end{align*}
  In particular, $\Phi_{\comp{T}}(Y) = \Phi_{T}(Y)$.
\end{lem}

\begin{proof}
  By setting $\mathscr{L}_T^{(n)} = \{(i,j)\in\N^2 \mid \mathrm{Leg}_{T}^+(i,j)
  \neq \varnothing,\ T_{i(j+1)} = n\}$, we have
  \begin{align*}
    \Phi_T(Y)&= \Phi_{T^{(n-1)}}(Y)\prod_{(i,j) \in \mathscr{L}_T^{(n)}} \left(1 -
    Y^{\#\Leg_T^+(i,j)}\right).
  \end{align*}
  Write $T=(C_1,\dots, C_{\ell})$, and suppose $(i,j)\in\mathscr{L}_T^{(n)}$.
  Then $T_{i(j+1)}=n$ and $n\notin C_j$. Since $\lambda-\mu$ are precisely the
  cells of $T$ labeled $n$, we have $(\mu_j,j)\in C_{\lambda,\mu}$. By
  \Cref{lem:dual-corners}, the result follows.
\end{proof}

\begin{ex}
  We revisit \Cref{ex:dual-operation}, where $n=6$, $\lambda=(9,8,7,6,2,1)$,
  and $\mu=(9,7,7,3,2)$. Let $T\in\SSYT_6$ be given by the flag of
  partitions
  \begin{align*}
  \lambda^{\bullet}(T): \quad   () \subseteq (4) \subseteq (6, 3) \subseteq (7,6,2) \subseteq (7,7,6,2) \subseteq (9,7,7,3,2) \subseteq (9,8,7,6,2,1).
  \end{align*}
  \Cref{fig:complements} illustrates the conclusion of
  \Cref{prop:complements}. One sees that
  \begin{align*}
    D_6(T) &= 25 = 1 + 2 + 0 + 9 + 13, & \comp{D}_6(T) &= 35 = 2 + 5 + 6 + 12 + 10
  \end{align*}
  Furthermore, $\Phi_T(Y)=(1-Y)^4(1-Y^2) = \Phi_{\comp{T}}(Y)$, as
  stated by~\Cref{lem:Phi-quotient}.  \exqed
\end{ex}

\begin{figure}[h]
  \centering
  \begin{subfigure}{0.3\textwidth}
    \centering
    \begin{tikzpicture}
      \pgfmathsetmacro{\w}{0.35}
      \foreach \i [count = \k] in {9,9,9,9,9,9}{
        \foreach \j in {1,...,\i}{
          \draw (\j*\w - \w, -\k*\w) -- ++(-\w,0) -- ++(0,\w) -- ++(\w,0) -- cycle;
        }
      }
      \node at (-0.5*\w, -0.5*\w) {{\footnotesize $1$}};
      \node at (0.5*\w, -0.5*\w) {{\footnotesize $1$}};
      \node at (1.5*\w, -0.5*\w) {{\footnotesize $1$}};
      \node at (2.5*\w, -0.5*\w) {{\footnotesize $1$}};
      \node at (3.5*\w, -0.5*\w) {{\footnotesize $2$}};
      \node at (4.5*\w, -0.5*\w) {{\footnotesize $2$}};
      \node at (5.5*\w, -0.5*\w) {{\footnotesize $3$}};
      \node at (6.5*\w, -0.5*\w) {{\footnotesize $5$}};
      \node at (7.5*\w, -0.5*\w) {{\footnotesize $5$}};
      \node at (-0.5*\w, -1.5*\w) {{\footnotesize $2$}};
      \node at (0.5*\w, -1.5*\w) {{\footnotesize $2$}};
      \node at (1.5*\w, -1.5*\w) {{\footnotesize $2$}};
      \node at (2.5*\w, -1.5*\w) {{\footnotesize $3$}};
      \node at (3.5*\w, -1.5*\w) {{\footnotesize $3$}};
      \node at (4.5*\w, -1.5*\w) {{\footnotesize $3$}};
      \node at (5.5*\w, -1.5*\w) {{\footnotesize $4$}};
      \node at (6.5*\w, -1.5*\w) {{\footnotesize $6$}};
      \node at (7.5*\w, -1.5*\w) {{\footnotesize $6$}};
      \node at (-0.5*\w, -2.5*\w) {{\footnotesize $3$}};
      \node at (0.5*\w, -2.5*\w) {{\footnotesize $3$}};
      \node at (1.5*\w, -2.5*\w) {{\footnotesize $4$}};
      \node at (2.5*\w, -2.5*\w) {{\footnotesize $4$}};
      \node at (3.5*\w, -2.5*\w) {{\footnotesize $4$}};
      \node at (4.5*\w, -2.5*\w) {{\footnotesize $4$}};
      \node at (5.5*\w, -2.5*\w) {{\footnotesize $5$}};
      \node at (6.5*\w, -2.5*\w) {{\footnotesize $4$}};
      \node at (7.5*\w, -2.5*\w) {{\footnotesize $4$}};
      \node at (-0.5*\w, -3.5*\w) {{\footnotesize $4$}};
      \node at (0.5*\w, -3.5*\w) {{\footnotesize $4$}};
      \node at (1.5*\w, -3.5*\w) {{\footnotesize $5$}};
      \node at (2.5*\w, -3.5*\w) {{\footnotesize $6$}};
      \node at (3.5*\w, -3.5*\w) {{\footnotesize $6$}};
      \node at (4.5*\w, -3.5*\w) {{\footnotesize $6$}};
      \node at (5.5*\w, -3.5*\w) {{\footnotesize $6$}};
      \node at (6.5*\w, -3.5*\w) {{\footnotesize $3$}};
      \node at (7.5*\w, -3.5*\w) {{\footnotesize $3$}};
      \node at (-0.5*\w, -4.5*\w) {{\footnotesize $5$}};
      \node at (0.5*\w, -4.5*\w) {{\footnotesize $5$}};
      \node at (1.5*\w, -4.5*\w) {{\footnotesize $6$}};
      \node at (2.5*\w, -4.5*\w) {{\footnotesize $5$}};
      \node at (3.5*\w, -4.5*\w) {{\footnotesize $5$}};
      \node at (4.5*\w, -4.5*\w) {{\footnotesize $5$}};
      \node at (5.5*\w, -4.5*\w) {{\footnotesize $2$}};
      \node at (6.5*\w, -4.5*\w) {{\footnotesize $2$}};
      \node at (7.5*\w, -4.5*\w) {{\footnotesize $2$}};
      \node at (-0.5*\w, -5.5*\w) {{\footnotesize $6$}};
      \node at (0.5*\w, -5.5*\w) {{\footnotesize $6$}};
      \node at (1.5*\w, -5.5*\w) {{\footnotesize $3$}};
      \node at (2.5*\w, -5.5*\w) {{\footnotesize $2$}};
      \node at (3.5*\w, -5.5*\w) {{\footnotesize $1$}};
      \node at (4.5*\w, -5.5*\w) {{\footnotesize $1$}};
      \node at (5.5*\w, -5.5*\w) {{\footnotesize $1$}};
      \node at (6.5*\w, -5.5*\w) {{\footnotesize $1$}};
      \node at (7.5*\w, -5.5*\w) {{\footnotesize $1$}};
      \draw[ultra thick] (-1*\w, 0*\w) -- (8*\w, 0*\w) -- (8*\w, -6*\w) -- (-1*\w, -6*\w) -- cycle;
      \draw[ultra thick] (0*\w, -6*\w) -- ++(0*\w, 1*\w) -- ++(1*\w, 0*\w) -- ++(0*\w, 1*\w) -- ++(4*\w, 0*\w) -- ++(0*\w, 1*\w) -- ++(1*\w, 0*\w) -- ++(0*\w, 1*\w) -- ++(1*\w, 0*\w) -- ++(0*\w, 1*\w) -- ++(1*\w, 0*\w);
    \end{tikzpicture}
    \caption{$T$ and $\comp{T}$}
  \end{subfigure}~
  \begin{subfigure}{0.3\textwidth}
    \centering
    \begin{tikzpicture}
      \pgfmathsetmacro{\w}{0.35}
      \foreach \i [count = \k] in {9,8,7,6,2,1}{
        \foreach \j in {1,...,\i}{
          \draw (\j*\w - \w, -\k*\w) -- ++(-\w,0) -- ++(0,\w) -- ++(\w,0) -- cycle;
        }
      }
      \node at (-0.5*\w, -0.5*\w) {{\footnotesize $1$}};
      \node at (0.5*\w, -0.5*\w) {{\footnotesize $1$}};
      \node at (1.5*\w, -0.5*\w) {{\footnotesize $1$}};
      \node at (2.5*\w, -0.5*\w) {{\footnotesize $1$}};
      \node at (3.5*\w, -0.5*\w) {{\footnotesize $2$}};
      \node at (4.5*\w, -0.5*\w) {{\footnotesize $2$}};
      \node at (5.5*\w, -0.5*\w) {{\footnotesize $3$}};
      \node at (6.5*\w, -0.5*\w) {{\footnotesize $5$}};
      \node at (7.5*\w, -0.5*\w) {{\footnotesize $5$}};
      \node at (-0.5*\w, -1.5*\w) {{\footnotesize $2$}};
      \node at (0.5*\w, -1.5*\w) {{\footnotesize $2$}};
      \node at (1.5*\w, -1.5*\w) {{\footnotesize $2$}};
      \node at (2.5*\w, -1.5*\w) {{\footnotesize $3$}};
      \node at (3.5*\w, -1.5*\w) {{\footnotesize $3$}};
      \node at (4.5*\w, -1.5*\w) {{\footnotesize $3$}};
      \node at (5.5*\w, -1.5*\w) {{\footnotesize $4$}};
      \node at (6.5*\w, -1.5*\w) {{\footnotesize $6$}};
      \node at (-0.5*\w, -2.5*\w) {{\footnotesize $3$}};
      \node at (0.5*\w, -2.5*\w) {{\footnotesize $3$}};
      \node at (1.5*\w, -2.5*\w) {{\footnotesize $4$}};
      \node at (2.5*\w, -2.5*\w) {{\footnotesize $4$}};
      \node at (3.5*\w, -2.5*\w) {{\footnotesize $4$}};
      \node at (4.5*\w, -2.5*\w) {{\footnotesize $4$}};
      \node at (5.5*\w, -2.5*\w) {{\footnotesize $5$}};
      \node at (-0.5*\w, -3.5*\w) {{\footnotesize $4$}};
      \node at (0.5*\w, -3.5*\w) {{\footnotesize $4$}};
      \node at (1.5*\w, -3.5*\w) {{\footnotesize $5$}};
      \node at (2.5*\w, -3.5*\w) {{\footnotesize $6$}};
      \node at (3.5*\w, -3.5*\w) {{\footnotesize $6$}};
      \node at (4.5*\w, -3.5*\w) {{\footnotesize $6$}};
      \node at (-0.5*\w, -4.5*\w) {{\footnotesize $5$}};
      \node at (0.5*\w, -4.5*\w) {{\footnotesize $5$}};
      \node at (-0.5*\w, -5.5*\w) {{\footnotesize $6$}};
    \end{tikzpicture}
    \caption{$T$}
  \end{subfigure}~%
  \begin{subfigure}{0.3\textwidth}
    \centering
    \begin{tikzpicture}
      \pgfmathsetmacro{\w}{0.35}
      \draw[fill=White, draw=White] (0*\w, -6*\w) -- ++(-\w,0) -- ++(0,\w) -- ++(\w,0) -- cycle;
      \foreach \i [count = \k] in {8,7,3,2,1}{
        \foreach \j in {1,...,\i}{
          \draw (\j*\w - \w, -\k*\w) -- ++(-\w,0) -- ++(0,\w) -- ++(\w,0) -- cycle;
        }
      }
      \node at (-0.5*\w, -0.5*\w) {{\footnotesize $1$}};
      \node at (0.5*\w, -0.5*\w) {{\footnotesize $1$}};
      \node at (1.5*\w, -0.5*\w) {{\footnotesize $1$}};
      \node at (2.5*\w, -0.5*\w) {{\footnotesize $1$}};
      \node at (3.5*\w, -0.5*\w) {{\footnotesize $1$}};
      \node at (4.5*\w, -0.5*\w) {{\footnotesize $2$}};
      \node at (5.5*\w, -0.5*\w) {{\footnotesize $3$}};
      \node at (6.5*\w, -0.5*\w) {{\footnotesize $6$}};
      \node at (-0.5*\w, -1.5*\w) {{\footnotesize $2$}};
      \node at (0.5*\w, -1.5*\w) {{\footnotesize $2$}};
      \node at (1.5*\w, -1.5*\w) {{\footnotesize $2$}};
      \node at (2.5*\w, -1.5*\w) {{\footnotesize $5$}};
      \node at (3.5*\w, -1.5*\w) {{\footnotesize $5$}};
      \node at (4.5*\w, -1.5*\w) {{\footnotesize $5$}};
      \node at (5.5*\w, -1.5*\w) {{\footnotesize $6$}};
      \node at (-0.5*\w, -2.5*\w) {{\footnotesize $3$}};
      \node at (0.5*\w, -2.5*\w) {{\footnotesize $3$}};
      \node at (1.5*\w, -2.5*\w) {{\footnotesize $6$}};
      \node at (-0.5*\w, -3.5*\w) {{\footnotesize $4$}};
      \node at (0.5*\w, -3.5*\w) {{\footnotesize $4$}};
      \node at (-0.5*\w, -4.5*\w) {{\footnotesize $6$}};
    \end{tikzpicture}
    \caption{$\comp{T}$}
  \end{subfigure}
  \caption{An illustration of the jigsaw operation on tableaux}
  \label{fig:complements}
\end{figure}

\subsection{Enumerating lattices by intersection and projection tableaux}

We apply the results of \Cref{sec:jigsaw-complement} to compute both
$f_{n, T}^{\mathrm{in}}(\lri)$ and $f_{n, T}^{\mathrm{pr}}(\lri)$ in
one motion.

\begin{thm}\label{thm:fnT} 
  Let $T\in\SSYT_n$. Then
  \begin{align*}
    f_{n, T}^{\mathrm{in}}(\lri) &= q^{D_n(T)}\Phi_{T}(q^{-1}),\\  f_{n, T}^{\mathrm{pr}}(\lri) &= q^{\comp{D}_n(T)}\Phi_{T}(q^{-1}) = f_{n, \comp{T}}^{\mathrm{in}}(\lri).
  \end{align*}
\end{thm}

\begin{proof}
  We proceed by induction on~$n$, the case $n=1$ being trivial since
  $f_{1, T}^{\mathrm{in}}(\lri)= f_{1, T}^{\mathrm{pr}}(\lri) = 1$ for
  $T\in\SSYT_1$. Assume that $n>1$ and that the statement holds
  for~$n-1$. Let $T\in \SSYT_n$, and write $(\lambda,\mu) = (\sh(T),
  \sh(T^{(n-1)}))$. Then we have
  \begin{align*}
    f_{n, T}^{\mathrm{in}}(\lri) &= f_{n, T^{(n-1)}}^{\mathrm{in}}(\lri) \mathrm{ext}^{\mathrm{in}}_{\lambda,\mu}(\lri) & & (\text{\Cref{eqn:ext-in}}) \\ 
    &= f_{n, T^{(n-1)}}^{\mathrm{in}}(\lri) q^{\mathrm{gap}(\lambda,\mu)} \prod_{a\in J_{\lambda,\mu}} \left(1 - q^{-\incr{a}{\mu'}}\right) & & (\text{\Cref{thm:latt.ext.int}}) \\ 
    &= q^{\mathrm{gap}(\lambda,\mu) + D_{n-1}(T)}\Phi_{T^{(n-1)}}(q^{-1}) \prod_{a\in J_{\lambda,\mu}} \left(1 - q^{-\incr{a}{\mu'}}\right) & & (\text{Induction}) \\
    &= q^{D_{n}(T)}\Phi_{T^{(n-1)}}(q^{-1}) \prod_{a\in J_{\lambda,\mu}} \left(1 - q^{-\incr{a}{\mu'}}\right) & & (\text{\Cref{lem:Schubert-gaps}}) \\ 
    &= q^{D_{n}(T)}\Phi_{T}(q^{-1}). & & (\text{\Cref{lem:Phi-quotient}}) 
  \end{align*}
  Since $\comp{D}_k(T) = D_k(\comp{T})$, the statement follows by
  \Cref{lem:Phi-quotient}.
\end{proof}

\subsection{Proofs of~\Cref{thmabc:HLS.affS.int,,thmabc:HLS.affS.proj}} \label{subsec:HLS.affSpr.proof}\label{subsec:HLS.affSin.proof}

We first show that, for $T\in \SSYT_n$,
\begin{align}\label{eqn:Z_nC}
  \bfZ^{\Dif(\lambda^{\bullet}(T))} &= \prod_{C\in T} \bfZ_{n, C}.
\end{align}
Suppose $T = (C_1, C_2)\in\SSYT_n$ is a two-column tableau. Then
$\lambda^{\bullet}(T) = \lambda^{\bullet}(C_1) +
\lambda^{\bullet}(C_2)$, where the $C_i$ are treated as one-column
tableaux for $i\in \{1,2\}$. Thus it suffices to show that
\eqref{eqn:Z_nC} holds for one-column tableaux $T=(C)$ for $C\subseteq
      [n]$.

Let $k\in [n]$. Then $\lambda^{(k)}(T) = (1^{(r_k)})\in \mathcal{P}_{k}$, where
$r_k = \#(C \cap [k])$. Let $e_i\in\N_0^k$ be the vector with $1$ in the $i$th
entry and $0$ elsewhere. Then $\Dif(\lambda^{(k)}(T)) = e_{r_k}$ and $C(r_k)\leq
k$. For $a = C(r_k)$ and $b = C(r_k+1)$, setting $b=n+1$ if $r_k = \# C$, we
have 
\[ 
  \Dif(\lambda^{(a)}(T)) = \Dif(\lambda^{(a + 1)}(T)) = \cdots = \Dif(\lambda^{(b-1)}(T)) = e_{r_k}.
\]
Therefore, \eqref{eqn:Z_nC} follows from the fact that
$ \bfZ^{\Dif(\lambda^{\bullet}(T))} = \bfZ_{n, C}$; see
\eqref{def:ZnC}. Putting everything together, we conclude that
\begin{align*}
  \affSin_{n,\lri}(\bfZ) &=\sum_{T\in\SSYT_n} \fin_{n,T}(\lri)
  \bfZ^{\Dif(\lambda^\bullet(T))} & & (\text{\Cref{equ:affS.int.rewrite}}) \\ 
  &=\sum_{T\in \SSYT_n}\Phi_T(q^{-1})q^{D_n(T)} \bfZ^{\Dif(\lambda^\bullet(T))} & & (\text{\Cref{thm:fnT}}) \\ 
  &=\sum_{T\in \rSSYT_n}\Phi_T(q^{-1})\sum_{(m_C)\in \N^T} \prod_{C\in
  T}\left(q^{\Schubdim{n}{C}}\bfZ_{n,C}\right)^{m_C} & & (\text{\Cref{eqn:Z_nC}}) \\ 
  &= \sum_{T
  \in \rSSYT_n}\Phi_T(q^{-1})\prod_{C\in
  T}\frac{q^{\Schubdim{n}{C}}\bfZ_{n,C}}{1-q^{\Schubdim{n}{C}}\bfZ_{n,C}}\\ &=
  \HLS_n\left(q^{-1},\left(q^{\Schubdim{n}{C}}
  \bfZ_{n,C}\right)_C\right).
\end{align*}
This completes the proof of \Cref{thmabc:HLS.affS.int}. 

Apply \Cref{thm:fnT} to get an analogous proof for
\Cref{thmabc:HLS.affS.proj}.\qed

\subsection{Proof of \Cref{thmabc:HLS.HS}}\label{subsec:HS.affS.proof}

We define a ring a homomorphism
\begin{align*}\Upsilon_n : \Z\llbracket \bm{Z} \rrbracket &\to
\Z\llbracket \bm{x},\bm{y}^{\pm 1} \rrbracket\\
  Z_{ij} &\mapsto \begin{cases}
    y_{n-i}^{-j}y_{n-i+1}^{j} & \text{if } i\neq n, \\
    x_jy_1^j & \text{if } i = n.
  \end{cases}
\end{align*}
We first prove an intermediate equation: 
\begin{align}\label{eqn:aff-sub}
  \HS_{n,\lri}(\bfx,\bfy) &= \affSin_{n,\lri}\left(\left(\Upsilon_n(Z_{ij})\right)_{1 \leq j \leq i \leq n} \right). 
\end{align}
Fix a lattice $\Lambda\in\mathcal{L}(\lri^n)$ and, for $i\in [n]$,
let $\lambda^{(i)} = \lambda^{(i)}(\Lambda)$ and $\delta_i
=\delta_i(\Lambda)$. By \Cref{prop:lattice-part-flag} we have, for
all $i\in [n]$,
\begin{equation*}
  \delta_{n-i+1}  = \left|\lambda^{(i)} -
  \lambda^{(i-1)} \right| = \sum_{j \geq 1} j\left(\incr{j}{\lambda^{(i)}} -
  \incr{j}{\lambda^{(i-1)}}\right).
\end{equation*}
Convening that $y_0=1$, this yields
\begin{equation*}
  \bfy^{\delta} = \prod_{i=1}^n
  y_{i}^{\delta_{n-i+1}} = \prod_{i=1}^n
  y_{i}^{\sum_{j=1}^i j\left(\incr{j}{\lambda^{(i)}} -
    \incr{j}{\lambda^{(i-1)}}\right)}=\prod_{1\leq j \leq i
    \leq n}
  \left(\frac{y_{n-i+1}^j}{y_{n-i}^j}\right)^{\incr{j}{\lambda^{(i)}}}.
\end{equation*}
Thus, \eqref{eqn:aff-sub} holds by \Cref{def:affS} and since 
\[ 
  \bfx^{\Dif(\lambda(\Lambda))} \bfy^{\delta(\Lambda)} = \left(
  \prod_{1\leq j \leq i \leq n}
  \left(\frac{y_{n-i+1}^j}{y_{n-i}^j}\right)^{\incr{j}{\lambda^{(i)}}}\right)
  \prod_{1\leq j\leq n} x_j^{\incr{j}{\lambda^{(n)}}}.
\]

The second step is to show that for all $C\subseteq [n]$, we have 
\begin{align}\label{eqn:ZnC-xy}
  x_{\#C} \bm{y}_{C^*} &= \Upsilon_n(\bm{Z}_{n,C}).
\end{align}
Set $m=\#C$. First assume that $n\in C$. In this case, 
\begin{align*}
  \Upsilon_n(\bm{Z}_{n,C}) &= x_my_1^m\prod_{k=1}^{m-1} ~\prod_{\varepsilon=0}^{C(k+1)-C(k)-1} y_{n-C(k)-\varepsilon}^{-k} y_{n-C(k)-\varepsilon+1}^k \\
  &= x_my_1^m\prod_{k=1}^{m-1} y_{n-C(k)+1}^{k} y_{n-C(k+1)+1}^{-k} = x_{m} \bm{y}_{C^*} .
\end{align*}
The case where $C\subseteq [n-1]$ is similar. Thus, \eqref{eqn:ZnC-xy}
holds. Applying \Cref{thmabc:HLS.affS.int} and \eqref{eqn:ZnC-xy} to
\eqref{eqn:aff-sub} completes the proof. \qed

\subsection{Proof of \Cref{thmabc:HLS.quiver}}\label{subsec:proof.quiver}

Write $V = V_n(\lri)=\left(V_i\right)_{i=1}^n$, where $V_i =
\lri^i$. The zeta function $\zeta_V(\bm{s})$ is a sum over all finite
index subrepresentations $V'$ of~$V$. By assumption $0<\alpha_1(\lri)
< \cdots < \alpha_{n-1}(\lri^{n-1})<\lri^n$ is a complete isolated
flag, so we set $\intflagterm{i} = \alpha_i(\lri^i)$.

Fix a sublattice $\Lambda\leq V_n$, and set $V_n'=\Lambda$. For $i\in
[n-1]$, a sublattice $V_i'\leq \lri^i$ is compatible with $V_n'$ if
and only if $\alpha_i(V_i')\subseteq \intflagterm{i}\cap \Lambda$. Since
the $\alpha_i$ are embeddings, $\Lambda$ determines a canonical
subrepresentation $V^{\Lambda} = \left(
\alpha_i^{-1}(\intflagterm{i}\cap\Lambda)\right)_{i=1}^n$. 
Moreover, every family $(V_i')_{i=1}^{n-1}$ of sublattices of
$\intflagterm{i}\cap \Lambda$ determines a subrepresentation $V'\leq
V^{\Lambda}$ with $V_n'=\Lambda$. Since $V_i'\cong \lri^i$ for all
$i\in [n-1]$,
\begin{align*}
  \sum_{V_i'\leq V_i^{\Lambda}} \left|V_i^{\Lambda} : V_i'\right|^{-s_i} &= \zeta_{\lri^i}(s_i).
\end{align*}
Recall that $v_C=(\max(C_0\cap [i]_0))_{i=1}^n\in\N^n$ for $C\subseteq [n]$.
Thus, by \Cref{thmabc:HLS.affS.int},
\begin{align*}
  \sum_{\Lambda\leq \lri^n} \left|\lri^n : \Lambda \right|^{-s_n} \prod_{i=1}^{n-1}\left|\intflagterm{i}: \intflagterm{i}\cap \Lambda \right|^{-s_i} &= \sum_{T\in\SSYT_n}f_{n,T}^{\mathrm{in}}(\lri) \prod_{i=1}^n q^{-\left|\lambda^{(i)}(T)\right|s_i} \\
  &= \affSin_{n,\lri}\left((q^{-js_i})_{1\leq j\leq i \leq n}\right) \\ 
  &= \HLS_n\left(q^{-1}, \left(q^{\Schubdim{n}{C} - v_C\cdot \bm{s}}\right)_C\right). 
\end{align*}
Putting everything together, we have
\begin{align*}
  \zeta_V(\bm{s}) &= \sum_{\Lambda\leq \lri^n} ~\prod_{i=1}^n\left|\lri^i : V^{\Lambda}_i \right|^{-s_i} \sum_{\substack{V'\leq V^{\Lambda} \\ V_n'=\Lambda}} ~\prod_{i=1}^{n-1}\left|V^{\Lambda}_i : V_i'\right| \\
  &= \HLS_n\left(q^{-1}, \left(q^{\Schubdim{n}{C} - v_C\cdot \bm{s}}\right)_C\right)\prod_{i=1}^{n-1}\zeta_{\lri^i}(s_i),
\end{align*}
which completes the proof of \Cref{thmabc:HLS.quiver}. \qed

\section{Tableaux and Dyck word statistics}
\label{sec:tab.dyck}

In this section we interpret the leg polynomials $\Phi_T(Y)$ from
\Cref{def:PhiT} in terms of Dyck words. Write $\mcD$ for the set of
finite Dyck words, viz.\ words in letters $\bfz$ and $\bfo$, both with
equal multiplicity, with the property that no initial segment contains
more $\bfo$s than~$\bfz$s. We define maps
\begin{alignat*}{3}
  \oldPhi:~&&\SSYT_n &\to \mcD && ~\textup{ in
    \Cref{subsubsec:tab.dyck}, from tableaux to Dyck
    words},\\ \oldPsi:~&&\mcD &\to \Z[Y] && ~\textup{ in
    \Cref{subsec:dyck.poly}, from Dyck words to polynomials}.
 \end{alignat*}
\Cref{prop:Phi=phi} expresses the leg polynomial $\Phi_T$ in terms
of~$\oldPsi(\oldPhi(T))$.

\subsection{From reduced tableaux to Dyck words}\label{subsubsec:tab.dyck}

We first define $\oldPhi$ on 2-column tableaux~$T=(C_1,C_2)$.  Set $\ol{C_1}=
C_1 \setminus (C_1\cap C_2)$ and $\ol{C_2}= C_2 \setminus (C_1 \cap C_2)$.
Clearly $a := \#\ol{C_1} \geq \#\ol{C_2}=:b$.  We obtain a Dyck word
$\oldPhi(T)$ as follows: form a word from the $a+b$ pairwise distinct elements
in $\ol{C_1}\cup\ol{C_2}$ by writing them in natural ascending order. Now
replace every element of $\ol{C_1}$ by a copy of $\bfz$ and every element of
$\ol{C_2}$ by a copy of~$\bfo$ and a further $a-b$ (``phantom'') copies
of~$\bfo$. The tableau condition ensures that $\oldPhi(T)$ is indeed a Dyck word
of length~$2a$. 

To define $\oldPhi$ on a general tableaux $T=(C_1,\dots,C_{\ell})$, for
$\ell\in\N_0$, simply concatenate the Dyck words for the 2-column tableaux
comprising adjacent pairs of columns of~$T$, in natural order: $\oldPhi(T) =
\prod_{i=1}^{\ell-1} \oldPhi((C_i,C_{i+1}))$.

\begin{ex}
  Let $C_1 = \{1,2,3,4,7,9\}$ and $C_2 = \{2,5,6,8\}$, and let
  $T=(C_1,C_2)$.  \Cref{fig:T12} shows, on the left, $T$ together with
  the Dyck word $\oldPhi(T) = \bfz \bfz \bfz \bfo \bfo \bfz \bfo \bfz
  \cdot \bfo\bfo$ of length~$10$. The tableau $T$ also yields the
  first two columns of the tableau $T'$ on the right in
  \Cref{fig:T12}.  In fact, $\oldPhi(T') = \oldPhi(T) \cdot (\bfz \bfz
  \bfo \cdot \bfo) \cdot (\bfz\bfo\bfz\bfz\cdot \bfo\bfo)$. \exqed
\end{ex}

\begin{figure}[h]
  \centering
  \begin{subfigure}{0.45\textwidth}
    \centering
    \begin{tikzpicture}
      \pgfmathsetmacro{\w}{0.35}
      \foreach \i [count = \k] in {2,2,2,2,1,1}{
        \foreach \j in {1,...,\i}{
          \draw (\j*\w - \w, -\k*\w) -- ++(-\w,0) -- ++(0,\w) -- ++(\w,0) -- cycle;
        }
      }
      \node at (-0.5*\w, -0.5*\w) {{\footnotesize $1$}};
      \node at (0.5*\w, -0.5*\w) {{\footnotesize $2$}};
      \node at (-0.5*\w, -1.5*\w) {{\footnotesize $2$}};
      \node at (0.5*\w, -1.5*\w) {{\footnotesize $5$}};
      \node at (-0.5*\w, -2.5*\w) {{\footnotesize $3$}};
      \node at (0.5*\w, -2.5*\w) {{\footnotesize $6$}};
      \node at (-0.5*\w, -3.5*\w) {{\footnotesize $4$}};
      \node at (0.5*\w, -3.5*\w) {{\footnotesize $8$}};
      \node at (-0.5*\w, -4.5*\w) {{\footnotesize $7$}};
      \node at (-0.5*\w, -5.5*\w) {{\footnotesize $9$}};
    \end{tikzpicture}
    \caption{$\oldPhi(T) = \bfz \bfz\bfz \bfo \bfo \bfz \bfo \bfz\bfo\bfo$}
  \end{subfigure}~
  \begin{subfigure}{0.45\textwidth}
    \centering
    \begin{tikzpicture}
      \pgfmathsetmacro{\w}{0.35}
      \foreach \i [count = \k] in {4,3,3,2,1,1}{
        \foreach \j in {1,...,\i}{
          \draw (\j*\w - \w, -\k*\w) -- ++(-\w,0) -- ++(0,\w) -- ++(\w,0) -- cycle;
        }
      }
        \node at (-0.5*\w, -0.5*\w) {{\footnotesize $1$}};
      \node at (0.5*\w, -0.5*\w) {{\footnotesize $2$}};
      \node at (1.5*\w, -0.5*\w) {{\footnotesize $2$}};
      \node at (2.5*\w, -0.5*\w) {{\footnotesize $3$}};
      \node at (-0.5*\w, -1.5*\w) {{\footnotesize $2$}};
      \node at (0.5*\w, -1.5*\w) {{\footnotesize $5$}};
      \node at (1.5*\w, -1.5*\w) {{\footnotesize $7$}};
      \node at (-0.5*\w, -2.5*\w) {{\footnotesize $3$}};
      \node at (0.5*\w, -2.5*\w) {{\footnotesize $6$}};
      \node at (1.5*\w, -2.5*\w) {{\footnotesize $8$}};
      \node at (-0.5*\w, -3.5*\w) {{\footnotesize $4$}};
      \node at (0.5*\w, -3.5*\w) {{\footnotesize $8$}};
      \node at (-0.5*\w, -4.5*\w) {{\footnotesize $7$}};
      \node at (-0.5*\w, -5.5*\w) {{\footnotesize $9$}};
    \end{tikzpicture}
    \caption{$\oldPhi(T') = \bfz \bfz\bfz \bfo \bfo \bfz \bfo \bfz\bfo\bfo\cdot
    \bfz \bfz \bfo \bfo \cdot \bfz\bfo\bfz\bfz \bfo\bfo$}
  \end{subfigure}
  \caption{Two tableaux and their Dyck words}
  \label{fig:T12} 
\end{figure}

\subsection{From Dyck words to polynomials}\label{subsec:dyck.poly}

Let $w\in\mcD$ be a Dyck word. There exist unique $r\in \N_0$ and
$\ell_1,\dots,\ell_r,m_1,\dots,m_r\in \N$ such that $w =
\bfz^{\ell_1}\bfo^{m_1}\dots \bfz^{\ell_r}\bfo^{m_r}$. For $k\in[r]$,
we define the $k$th \emph{valley} and \emph{peak} via
\begin{align*}
  \bt_k &= \sum_{i\leq k}\left(\ell_i-m_i\right), & 
  \tp_k &= m_k + \bt_k = m_k + \sum_{i \leq k}\left(\ell_i-m_i\right).
\end{align*}
For $m\in \N$, we define $\llbracket 0 \rrbracket = 1$, $\llbracket m
\rrbracket = 1 - Y^m$, and $\llbracket m \rrbracket ! = \prod_{j=1}^m
\llbracket j \rrbracket$, and set
\begin{equation}\label{def:Phi}
   \oldPsi: \mcD \rightarrow \Z[Y], \quad  w \longmapsto \prod_{k\in[r]}
  \frac{\llbracket \tp_k\rrbracket!}{\llbracket
    \bt_k\rrbracket!}.
\end{equation}

One may picture the Dyck word $w$ as a mountain range, where $\bfz$ is
a line segment with positive slope and $\bfo$ is one with negative
slope. Thus, $w$ consists of $r$ peaks at altitudes $\tp_k$, separated
by $r-1$ valleys at altitudes $\bt_k$. Note that $\bt_r=0$ by
definition. The factors of the product in \eqref{def:Phi} correspond
to the negative slopes of~$w$, weighted by their altitudes. The
mountain range in \Cref{fig:dyck}, for instance, has three such
segments at height $2$ and one each of heights $1$ and $3$, whence
\[ 
  \oldPsi(\bfz \bfz\bfz \bfo \bfo \bfz \bfo \bfz\bfo\bfo) = (1-Y)(1-Y^2)^3(1-Y^3). 
\] 

\begin{figure}[h]
  \centering
  \begin{tikzpicture}
    \pgfmathsetmacro{\w}{0.5}
    \fill (0,0) circle(1pt);
    \draw (0*\w, 0*\w) -- (1*\w, 1*\w);
    \fill (1*\w, 1*\w) circle(1pt);
    \draw (1*\w, 1*\w) -- (2*\w, 2*\w);
    \fill (2*\w, 2*\w) circle(1pt);
    \draw (2*\w, 2*\w) -- (3*\w, 3*\w);
    \fill (3*\w, 3*\w) circle(1pt);
    \draw (3*\w, 3*\w) -- (4*\w, 2*\w);
    \fill (4*\w, 2*\w) circle(1pt);
    \draw (4*\w, 2*\w) -- (5*\w, 1*\w);
    \fill (5*\w, 1*\w) circle(1pt);
    \draw (5*\w, 1*\w) -- (6*\w, 2*\w);
    \fill (6*\w, 2*\w) circle(1pt);
    \draw (6*\w, 2*\w) -- (7*\w, 1*\w);
    \fill (7*\w, 1*\w) circle(1pt);
    \draw (7*\w, 1*\w) -- (8*\w, 2*\w);
    \fill (8*\w, 2*\w) circle(1pt);
    \draw (8*\w, 2*\w) -- (9*\w, 1*\w);
    \fill (9*\w, 1*\w) circle(1pt);
    \draw (9*\w, 1*\w) -- (10*\w, 0*\w);
    \fill (10*\w, 0*\w) circle(1pt);
  \end{tikzpicture}
  \caption{The Dyck word $\bfz \bfz\bfz \bfo \bfo \bfz \bfo \bfz\bfo\bfo$ as a mountain range}
  \label{fig:dyck}
\end{figure}

\begin{defn}\label{def:phiT}
  For $T= (C_1,\dots,C_\ell)\in\SSYT_n$, define the \define{phantom
    factor}
  \begin{align*}
    \pha_T(Y) &= \prod_{s = 1}^{\ell-1}\llbracket (\#C_s) -
    (\#C_{s+1})\rrbracket!  \in\Z[Y].
  \end{align*}
\end{defn}

 For the tableau $T'$ on the right in \Cref{fig:T12}, for instance, we find
 $\pha_{T'}(Y)=(1-Y)^3(1-Y^2)^2$. Indeed, the three relevant column pairs yield
 the respective factors $\llbracket 2\rrbracket!$, $\llbracket 1\rrbracket!$, and
 $\llbracket 2\rrbracket!$.

\begin{prop}\label{prop:Phi=phi} 
  For all $T\in\SSYT_n$ we have 
  \[ 
    \oldPsi(\oldPhi(T))/\pha_T(Y) = \Phi_T(Y).
  \]
\end{prop}

\begin{proof}
  From \Cref{def:PhiT}, it suffices to show that the statement holds
  for tableaux with exactly two columns, so assume that
  $T=(C_1,C_2)$. Likewise, the leg polynomial of $T$ is clearly
  oblivious of common elements of $C_1$ and $C_2$. In other words, we
  may also assume that $\overline{C_1}=C_1$ and $\overline{C_2}=C_2$
  are disjoint sets of cardinalities $a=\#C_1 \geq \#C_2=b$, say. We
  observe further that the leg set
  \[ 
    \mathscr{L}_{T}= \left\{(i,1) \in\N^2 \mid \Leg^+_T(i,1)\neq
    \varnothing \right\}= \left\{k\in[b]\mid C_2(k) > C_1(k)\right\}
  \]
  is in bijection with the factors defining $\oldPsi(\oldPhi(T))$ bar the final
  $a-b$ factors, which define $\llbracket \tp_r \rrbracket! = \pha_T(Y)$. The
  remaining factors correspond to the negative slopes of the mountain range
  associated with the Dyck word $\oldPhi(T) = \bfz^{\ell_1}\dots\bfo^{m_r}$
  indexed by the copies of the letter $\bfo$ outside the final
  factor~$\bfo^{m_r}$.  Each of them corresponds to a leg whose length
  $\#\Leg_T^+(i,1)$ is exactly the altitude of the corresponding negative slope.
  Hence $\oldPsi(\oldPhi(T)) = \Phi_T(Y)\pha_T(Y)$ as claimed.
\end{proof}

\begin{ex}\label{ex:int3.rev}
  The factor $(1-q^{-1})^3$ of $\fin_{3,T}(\lri)$ in \Cref{ex:int3} and of
  $\fpr_{3,T}(\lri)$ in \Cref{ex:proj3} reflects the fact that
  $\oldPsi(\oldPhi(T)) = \oldPsi(\bfz\bfo\bfz\bfo\bfz\bfo) = (1-Y)^3$. \exqed
\end{ex}

\section{Reduced tableaux and Bruhat orders}\label{sec:poset}

In this section we portray the Hall--Littlewood--Schubert series $\HLS_n(Y,\bfX)$ as a
$Y$-analog of the fine Hilbert series of a Stanley--Reisner ring of
a simplicial complex. To explain this vantage point we define, in
\Cref{subsec:poset}, a poset structure~$\msfT_n$, called the
\emph{tableau order}, on the power set of $[n]$ that models adjacency
of label sets of columns in tableaux and refines the set-containment
relation. In \Cref{subsec:bruhat}, we show that $\msfT_n$ is
isomorphic to a parabolic quotient of the hyperoctahedral group $B_n$
under the Bruhat order. Its order complex~$\Delta(\msfT_n)$ is
isomorphic to~$\rSSYT_n$. Using this and results from
Bj\"orner--Wachs~\cite{BW:Bruhat}, we prove some topological
properties of $\Delta(\msfT_n)$ in~\Cref{subsec:top-props},
culminating in the proof of \Cref{thm:CM} in~\Cref{subsubsec:CM}.

\subsection{Tableau order on $2^{[n]}$}\label{subsec:poset}

We define a partial order on $2^{[n]}$ as follows. Given non-empty
$A,B\subseteq [n]$, we write $A\sqsubseteq B$ if there exists a
2-column tableau whose first column comprises the elements of $A$ and
whose second column comprises the elements of $B$. We write
$A\sqsubseteq \varnothing$ for all $A\subseteq [n]$.  We call
$\sqsubseteq$ the \define{tableau order} on~$2^{[n]}$, and set
$\msfT_n = (2^{[n]}\setminus\{\varnothing\}, \sqsubseteq)$. We write
$I\sqsubset J$ if $I\sqsubseteq J$ and $I\neq J$. We note that, for
$C,D\subseteq [n]$, $C\sqsubseteq D$ if and only if $[n]\setminus D
\sqsubseteq [n] \setminus C$.  \Cref{fig:T-poset} gives the Hasse
diagrams for $\msfT_n$ for $n\in\{2,3,4\}$. 

\begin{figure}[h]
  \captionsetup[subfigure]{labelformat=empty}
  \vspace{-1em}
  \centering 
  \begin{subfigure}[b]{0.3\textwidth}
    \centering
    \begin{tikzpicture}
      \pgfmathsetmacro{\x}{0.65} 
      \pgfmathsetmacro{\y}{0.75}
      \node[label={right=1em:$12$}](0) at (0, 0) [circle, fill=black, inner sep=1.5pt] {}; 
      \node[label={right=1em:$1$}](1) at (0, \y) [circle, fill=black, inner sep=1.5pt] {}; 
      \node[label={right=1em:$2$}](2) at (0, 2*\y) [circle, fill=black, inner sep=1.5pt] {}; 
      \draw (2) -- (1); 
      \draw (1) -- (0);
    \end{tikzpicture}
    \caption{$n=2$}
  \end{subfigure}
  \begin{subfigure}[b]{0.3\textwidth}
    \centering
    \begin{tikzpicture}
      \pgfmathsetmacro{\x}{0.5} 
      \pgfmathsetmacro{\y}{0.5}
      \node[label={right=1em:$123$}](0) at (0, 0) [circle, fill=black, inner sep=1.5pt] {}; 
      \node[label={left=1em:$1$}](1) at (-\x, 3*\y) [circle, fill=black, inner sep=1.5pt] {}; 
      \node[label={right=1em:$2$}](2) at (0, 4*\y) [circle, fill=black, inner sep=1.5pt] {}; 
      \node[label={right=1em:$3$}](3) at (0, 5*\y) [circle, fill=black, inner sep=1.5pt] {}; 
      \node[label={right=1em:$12$}](12) at (0, \y) [circle, fill=black, inner sep=1.5pt] {}; 
      \node[label={right=1em:$13$}](13) at (0, 2*\y) [circle, fill=black, inner sep=1.5pt] {}; 
      \node[label={right=1em:$23$}](23) at (\x, 3*\y) [circle, fill=black, inner sep=1.5pt] {}; 
      \draw (12) -- (13);
      \draw (13) -- (1);
      \draw (13) -- (23);
      \draw (1) -- (2);
      \draw (23) -- (2);
      \draw (2) -- (3);
      \draw (0) -- (12);
    \end{tikzpicture}
    \caption{$n=3$}
  \end{subfigure}
  \begin{subfigure}[b]{0.3\textwidth}
    \centering
    \begin{tikzpicture}
      \pgfmathsetmacro{\x}{0.5} 
      \pgfmathsetmacro{\y}{0.5}
      \node[label={right=1em:$1234$}](0) at (0, 0) [circle, fill=black, inner sep=1.5pt] {}; 
      \node[label={left=1em:$1$}](1) at (-2*\x, 6*\y) [circle, fill=black, inner sep=1.5pt] {}; 
      \node[label={left=1em:$2$}](2) at (-\x, 7*\y) [circle, fill=black, inner sep=1.5pt] {}; 
      \node[label={right=1em:$3$}](3) at (0, 8*\y) [circle, fill=black, inner sep=1.5pt] {}; 
      \node[label={right=1em:$4$}](4) at (0, 9*\y) [circle, fill=black, inner sep=1.5pt] {}; 
      \node[label={left=1em:$12$}](12) at (-\x, 3*\y) [circle, fill=black, inner sep=1.5pt] {}; 
      \node[label={left=1em:$13$}](13) at (0, 4*\y) [circle, fill=black, inner sep=1.5pt] {}; 
      \node[label={left=1em:$14$}](14) at (-\x, 5*\y) [circle, fill=black, inner sep=1.5pt] {}; 
      \node[label={right=1em:$23$}](23) at (\x, 5*\y) [circle, fill=black, inner sep=1.5pt] {}; 
      \node[label={right=1em:$24$}](24) at (0, 6*\y) [circle, fill=black, inner sep=1.5pt] {}; 
      \node[label={right=1em:$34$}](34) at (\x, 7*\y) [circle, fill=black, inner sep=1.5pt] {}; 
      \node[label={right=1em:$123$}](123) at (0, \y) [circle, fill=black, inner sep=1.5pt] {}; 
      \node[label={right=1em:$124$}](124) at (0, 2*\y) [circle, fill=black, inner sep=1.5pt] {}; 
      \node[label={right=1em:$134$}](134) at (\x, 3*\y) [circle, fill=black, inner sep=1.5pt] {}; 
      \node[label={right=1em:$234$}](234) at (2*\x, 4*\y) [circle, fill=black, inner sep=1.5pt] {}; 
      \draw (123) -- (124);
      \draw (124) -- (134);
      \draw (124) -- (12);
      \draw (134) -- (234);
      \draw (12) -- (13);
      \draw (134) -- (13);
      \draw (13) -- (14);
      \draw (13) -- (23);
      \draw (234) -- (23);
      \draw (23) -- (24);
      \draw (24) -- (34);
      \draw (14) -- (24);
      \draw (14) -- (1);
      \draw (1) -- (2);
      \draw (24) -- (2);
      \draw (2) -- (3);
      \draw (34) -- (3);
      \draw (3) -- (4); 
      \draw (0) -- (123);
    \end{tikzpicture}
    \caption{$n=4$}
  \end{subfigure}
  \caption{Hasse diagrams for $\msfT_n$ for $n\in\{2,3,4\}$}
  \label{fig:T-poset}
\end{figure}

Let $P$ be a poset. The \define{order complex} of $P$, written $\Delta(P)$, is
the simplicial complex whose simplices are the flags of~$P$.

\begin{lemma}\label{lem:tab.flags}
  The posets $\Delta(\msfT_n)$ and $\rSSYT_n$ are isomorphic.
\end{lemma}

\begin{proof}
  The columns of a tableau $T = (C_1,\dots,C_\ell)\in \rSSYT_n$ form,
  by definition of $\msfT_n$, a flag $C_1 \sqsubset C_2 \sqsubset
  \dots \sqsubset C_\ell$. Conversely, every such flag yields a
  reduced tableau $(C_1,\dots,C_\ell)$. The bijection is clearly
  order-preserving.
\end{proof}

\begin{remark}\label{rem:PhiT.triv}
  We leave it to the reader to verify that, given $T\in\rSSYT$, we
  have $\Phi_T=1$ if and only if the columns $(C_1,\dots,C_\ell)$ of
  $T$ form a flag $C_1 \supseteq C_2 \supseteq\dots\supseteq C_\ell$.
\end{remark}

Let $(P,<)$ be a poset and $x, y\in P$. We say that $x$
\define{covers} $y$ if $x<y$ and $x\leq z < y$ implies $z=x$. We call
such a $y$ an \define{upper cover} for $x$. If $C,C'\in \msfT_n$ such
that $C'$ covers $C$, then we write $C\sqsubsetdot C'$. We
characterize all upper covers in $\msfT_n$ in~\Cref{lem:all-covers}.

Let $C\in\msfT_n$ and let $1\leq a \leq b\leq n$ such that $C$
contains the interval $[a,b]=\{a,a+1,\dots,b\}$. The latter is
\define{isolated} in $C$ if both $a-1$ and $b+1$ are not contained
in~$C$. For example, $C=\{1,2,3,5\}\in\msfT_5$ has exactly the two
isolated intervals $[1,3]$ and~$[5,5]$. Assume that $[a,b]$ is an
isolated interval in $C \in \msfT_n \setminus \{\{n\}\}$.

Note that every $C\in\msfT_n$ allows a unique decomposition
\begin{equation}\label{equ:C.iso}
  C =
  \bigsqcup_{i=1}^k [a_i, b_i]
\end{equation}
as a disjoint sum of isolated intervals, for uniquely determined
$k\in\N$ and $ 1 \leq a_1 \leq b_1 < b_1 + 1 < a_2 \leq b_2 < \cdots <
a_k \leq b_k \leq n $. The \define{elevation of $C$ at $[a,b]$} is
\begin{align*}
  \widehat{C}_{ab} &= \left((C \setminus\{b\}) \cup \{b+1\}\right) \cap [n] \in \msfT_n . 
\end{align*}
For $C = \{1,2,3,5\}\in \msfT_5$, the two elevations are
$\widehat{C}_{13}=\{1,2,4,5\}$ and~$\widehat{C}_{55}=\{1,2,3\}$.

\begin{lem}\label{lem:int-covers}
  Let $C= [a,b]\in \msfT_n\setminus\{\{n\}\}$. The unique upper cover
  for $C$ is $\widehat{C}_{ab}$.
\end{lem}

\begin{proof}
  Observe that $C \sqsubset \widehat{C}_{ab}$. Let $C'\subseteq [n]$ such that
  $C\sqsubseteq C' \sqsubset \widehat{C}_{ab}$. We have two cases depending on
  whether $b=n$ or not. 
  
  First we assume that $b\neq n$. Then $\#C = \#\widehat{C}_{ab}$, so $\#C =
  \#C'$. Let $k = \# C$, so $C(i) \leq C'(i)$ for all $i\in [k]$ Assume via
  contradiction that $C\neq C'$, so there exists some $i\in [k]$ such that $C(i)
  < C'(i)$. It follows that 
  $
    C(k) = b < b+ 1 \leq C'(k)
  $. 
  Therefore, $\widehat{C}_{ab} \sqsubseteq C'$, which is a contradiction. Thus,
  $C=C'$ in this case, so that $\widehat{C}_{ab}$ is the unique cover for $C$.

  Now we assume that $b=n$. It follows that $\#C - \#C' = 1$. Since $\min(C) = a$, 
  $
    [a, \#C' + a - 2] \sqsubseteq C'
  $,
  but $a = n - \#C + 1$, implying that $\widehat{C}_{ab}\sqsubseteq C'$. This is
  a contradiction, so $\widehat{C}_{ab}$ is the unique cover. 
\end{proof}

\begin{prop}\label{lem:all-covers}
  Let $C\in\msfT_n\setminus\{\{n\}\}$ be as in \eqref{equ:C.iso}. For each
  $i\in [k]$, we have
  \[ 
    C \sqsubsetdot \left(\widehat{C}_{a_i,b_i} \sqcup \bigsqcup_{j\neq
      i} [a_j, b_j], \right) \] and all upper covers of $C$ in
    $\msfT_n$ are of this form.
\end{prop}

\begin{proof}
  This is a simple induction on $k$, the base case being
  \Cref{lem:int-covers}.
\end{proof}

\subsection{Bruhat order on $B_n^{[n-1]}$}
\label{subsec:bruhat}

Let $(W, S)$ be a Coxeter system, so that $W$ is a Coxeter group with
simple reflections $S$. Let $\ell=\ell_W$ be the (Coxeter) length
function. Let $w,w'\in W$, where $w'$ has reduced expression
$s_1\cdots s_a$ for elements $s_i\in S$, so in
particular~$\ell(w')=a$.  We write $w\leq w'$ if there exists a
reduced expression $w=s_{i_1}\cdots s_{i_b}$ with
$\{i_1,\dots,i_b\}_<\subseteq[a]$. The relation $\leq$ on $(W, S)$ is
the \define{Bruhat order}; see, for
instance,~\cite[Sec.~2.3]{BW:Bruhat} (``subword property'').

The \define{hyperoctahedral group} $B_n$ is a Coxeter group generated
by simple reflections $S = \{s_0, \dots, s_{n-1}\}$ satisfying the
relations
\[ 
  (s_0s_1)^4 = (s_is_{i+1})^3 = (s_js_k)^2 = s_j^2 = 1
\] 
for all $i\in [n-2]$ and $j,k\in [n-1]_0$ with $|j-k|\geq 2$. We call relations of
the form $(s_js_k)^2=1$ ``commuting relations''. The group $B_n$ is isomorphic
to the group of signed $n\times n$-permutation matrices: indeed, for
$i\in[n-1]$, we may think of $s_i$ as the matrix transposing $i$ and $i+1$;
the reflection $s_0$ may be represented by the diagonal
matrix~$\diag(-1,1,1,\dots, 1)$. In~\Cref{lem:B_n^{[n-1]}} we describe the
\emph{parabolic quotient}
\[
  B_n^{[n-1]} = \{w\in B_n \mid \forall i \in [n-1],\ \ell(w) < \ell(ws_i)\};
\]
see \cite[Lem.~2.4.3]{BjoernerBrenti/05}. In \cite{StasinskiVoll/13},
elements of $w\in B_n^{[n-1]}$ are called \emph{ascending matrices} by
dint of their defining property $w(1) < \dots < w(n)$.

To this end we define elements $w_1,\dots,w_n\in B_n$ by setting $w_1=s_0$ and
$w_{k+1} = s_kw_k$ for~$k\in[n-1]$. For $I = \{i_1,\dots,i_\ell\}_{<}\subseteq
[n]$, set $ w_I = w_{i_1} \cdots\ w_{i_\ell}$.

\begin{lem}\label{lem:reduced-exp}
  For $I\subseteq [n]$, the word $w_I$ is a reduced expression.
\end{lem}

\begin{proof}
  Since each simple reflection appears in the word $w_i$ at most once,
  $w_I$ is a reduced expression for all $I\subseteq [n]$ with $\#I\leq
  1$.
  
  Let $1\leq i < j \leq n$. We show that we cannot apply the relation
  $s_k^2=1$, for $k\in [n]$, without increasing the length of the word
  $w=w_iw_j$. Since $w_i$ is a reduced expression, by just applying
  commuting relations we have
  \begin{equation}\label{eqn:word-i-j}
    \begin{split}
      w &= s_{i-1} \cdots s_0 s_{j-1}\cdots s_0 = s_{i-1} \cdots s_1 s_{j-1}\cdots s_2 s_0 s_1 s_0 \\
      &= s_{i-1} \cdots s_2 s_{j-1}\cdots s_3 s_1 s_2 s_0 s_1 s_0  = \hdots\\
      &= s_{j-1} \cdots s_{i+1} (s_{i-1}s_{i})\cdots (s_{k-1}s_k) \cdots
      (s_0s_1) s_0.
    \end{split}
  \end{equation}
  For all $k\in \{3, \dots, i\}$, we have reduced expressions of the form
  $s_{k-1}s_{k}s_{k-2}s_{k-1}$ and $s_0s_1s_0$ in~\eqref{eqn:word-i-j}. Hence,
  for all $I\subseteq [n]$, we cannot apply the relation $s_k^2=1$ without
  increasing the length of the expression for $w_I$.

  Let $1\leq i<j<k\leq n$. We show that we cannot apply $(s_rs_{r+1})^3=1$
  without increasing the length of $w=w_iw_jw_k$. By using the commuting
  relations and~\eqref{eqn:word-i-j}, we have 
  \begin{align*}
    w &= w_i s_{k-1} \cdots s_{j+1} (s_{j-1}s_{j}) \cdots (s_0s_1) s_0 \\
    &= s_{k-1} \cdots s_{j+1} (s_{j-1}s_{j})\cdots (s_{i+1}s_{i+2}) w_i (s_{i}s_{i+1})\cdots (s_0s_1) s_0 \\
    &= s_{k-1} \cdots s_{j+1} (s_{j-1}s_{j})\cdots (s_{i+1}s_{i+2}) (s_{i-1}s_is_{i+1}) \cdots (s_{0}s_1s_{2}) (s_0s_1)s_0 .
  \end{align*}
  For each $r\in \{1,\dots i -1\}$, we have expressions of the form 
  \begin{align*}
    u_r &:= (s_rs_{r+1}s_{r+2})(s_{r-1}s_{r}s_{r+1})(s_{r-2}s_{r-1}s_{r}),
  \end{align*}
  where $s_{-1}=1$. We cannot apply either $(s_rs_{r+1})^3=1$ or
  $(s_{r-1}s_{r})^3=1$ to the expression $u_r$ without increasing its length.
  Hence, $u_r$ is a reduced expression, so we cannot apply $(s_rs_{r+1})^3=1$
  relation to $w_I$ for all $I\subseteq [n]$ without increasing its length. The
  argument concerning $(s_0s_1)^4 = 1$ is similar.
\end{proof}

\begin{lem}\label{lem:B_n^{[n-1]}}
  We have
  \[ 
    B_n^{[n-1]} = \left\{ w_I ~\middle|~ I \subseteq [n]\right\}. 
  \]
\end{lem}

\begin{proof}
  Let $k\in [n]$. Since each $w_k$ ends with $s_0$, it follows that
  $\ell(w_ks_1) > \ell(w_k)$. For $i\in [n-1]$ and $i\geq k + 1$, the
  reflection $s_i$ commutes with all $s_j$ for $j\in [k-1]_0$, so
  $\ell(w_ks_i) > \ell(w_k)$. Lastly for all $i\in [2,n-1]\cap [k]$,
  the reflection $s_i$ commutes with all but at most two letters in
  the word $w_k = s_{k-1}\cdots s_1s_0$, namely $s_{i-1}$
  and~$s_{i+1}$. Thus, $\ell(w_ks_i) > \ell(w_k)$ since
  \begin{align*}
    w_k s_i = s_{k-1} \cdots s_{i+1}\underline{s_is_{i-1}s_i}s_{i-2}\cdots s_1s_0 &= s_{k-1} \cdots s_{i+1}\underline{s_{i-1}s_{i}s_{i-1}}s_{i-2}\cdots s_1s_0 \\ 
    &= s_{i-1}w_k.
  \end{align*}
  By \Cref{lem:reduced-exp}, all the $w_I$ are reduced
  expressions. The lemma follows.
\end{proof}

The restriction of the Bruhat order on $B_n$ to $B_n^{[n-1]}$ defines a partial
order. In particular, $w_1<\cdots < w_n$. In the next proposition we relate
$\msfT_n$ with $B_n^{[n-1]}$ by means of the set involution $g : 2^{[n]} \to
2^{[n]}, \, I\mapsto \{n - j + 1 \mid j \in [n]\setminus I \}$. This is a
specialization of a more general result relating the Gale order on subsets of a
finite set and Bruhat orders of parabolic quotients by
Vince~\cite[Thm.~1]{Vince/00}.

\begin{prop}\label{prop:Bruhat-iso}
  The map $\alpha : \msfT_n\cup\{\varnothing\} \to B_n^{[n-1]}$ given by $I
  \mapsto w_{g(I)}$ is an isomorphism of posets.
\end{prop}

\begin{proof}
  Since $g$ is a bijection and by \Cref{lem:B_n^{[n-1]}}, the map $\alpha$ is a
  bijection of sets. It remains to show that $\alpha$ is order-preserving.
  Suppose $C\sqsubseteq D$, so $[n]\setminus D\sqsubseteq [n]\setminus C$. Then
  there is an embedding $\iota : g(C) \hookrightarrow g(D)$ such that $\iota(x)
  \geq x$ for all $x\in g(C)$. Hence $w_{g(C)}\leq w_{g(D)}$, so $\alpha$ is an
  order-preserving isomorphism.
\end{proof}

\subsection{Combinatorial and topological properties of $\msfT_n$}
\label{subsec:top-props} 

\Cref{thm:CM} asserts that $|\Delta(\msfT_n)|$ is Cohen--Macaulay over
$\Z$ and homeomorphic to an $\left(\binom{n+1}{2}-1\right)$-ball. We
prove it in~\Cref{subsubsec:CM}. In \Cref{sec:HLS.coarse} we use it to
describe properties of specializations of the bivariate
coarsening~$\HLS_n(Y,(X)_C)$ at special values of~$Y$.

A finite poset $P$ is \define{graded} if it has a unique top and
bottom element and all maximal chains have the same cardinality. The
\define{rank} of a finite graded poset is one less than the
cardinality of a maximal chain linking bottom and top elements.

\begin{prop}\label{cor:graded-rank}
  The poset $\msfT_n$ is graded of rank $\binom{n+1}{2}-1$. 
\end{prop}

\begin{proof}
  By~\cite[Chain Property~2.6]{BW:Bruhat} and \Cref{prop:Bruhat-iso}, $\msfT_n$ is graded. Observe
  that 
  \[ 
    1 < w_1 < \cdots < w_n < w_1w_n < \cdots < w_{n-1}w_n < \cdots < w_1\cdots w_n,
  \] 
  where $1\in B_n$ is the identity, is a maximal chain in $B_n^{[n-1]}$ with
  cardinality $\binom{n+1}{2}+1$. By removing the top element from $\msfT_n\cup
  \{\varnothing\}$, the statement follows.
\end{proof}

A tableau $T\in\rSSYT_n$ is \define{maximal} if it corresponds to a maximal
flag under the isomorphism in~\Cref{lem:tab.flags}. Recall the flag of
partitions $\lambda^\bullet(T) = \left( \lambda^{(i)}(T)\right)_{i\in [n]}$
associated with~$T$; see~\eqref{def:flop.tab}. The following result asserts that
$T$ is maximal if and only if every integer between $1$ and $\binom{n+1}{2}$
is a part of some member of this flag.

\begin{lem}\label{prop:max-tab}
  Let $T\in\rSSYT_n$ and $r=\binom{n+1}{2}$. Then $T$ is maximal if
  and only if
  \begin{align}\label{eqn:partitions-set}
    \left\{ \lambda_j^{(i)}(T) ~\middle|~ i\in [n], j\in [i]\right\} = [r].
  \end{align} 
\end{lem}

\begin{proof}
  Suppose \eqref{eqn:partitions-set} holds. Note that $\lambda_1^{(n)}(T)$ is
  maximal among the parts~$\lambda_j^{(i)}(T)$. Since $T$ is reduced and
  $\lambda_1^{(n)}(T)=r$, it follows that $T$ is maximal by
  \Cref{cor:graded-rank}.

  Suppose $T = (C_1,\dots, C_r)$ is maximal, and assume, for a contradiction,
  that some of the parts of $\lambda=\lambda^{(a)}(T)$ and
  $\mu=\lambda^{(b)}(T)$ coincide for some $a,b\in [n]$ with $a>b$. First
  suppose $\lambda_k = \mu_k=\ell\in [r]$ for some $k\in \N$. Then
  $T_{k\ell} \leq b$. If $T$ has a $(k, \ell+1)$-cell, then
  \[ 
    T_{k\ell} = C_\ell(k) \leq b < a < C_{\ell+1}(k) = T_{k(\ell+1)},
  \]
  so $T_{k\ell} + 2\leq T_{k(\ell+1)}$. By \Cref{lem:all-covers}, $C_{\ell+1}$
  does not cover $C_{\ell}$. If $T$ does not have a $(k,\ell+1)$-cell, then
  $\#C_{\ell} > \#C_{\ell+1}$, allowing for $C_{\ell+1} = \varnothing$ when
  $\ell=r$. Since $T$ is maximal, by \Cref{lem:all-covers} we have
  $C_{\ell}\setminus C_{\ell+1} = \{n\}$. This implies $n\leq b < a$, which
  cannot happen. Hence, $\lambda_k = \mu_k > 0$ is impossible for all $k\in \N$.

  Assume now that $\lambda_j = \mu_i = \ell\in [r]$ for $i < j$. Thus both
  $T_{i\ell}\leq b$ and $T_{j\ell}\leq a$. From the previous case, we may assume
  that $b < T_{j\ell}\leq a$. If $T$ has a $(j, \ell+1)$-cell, then we have
  \begin{align*}
    T_{i\ell} &= C_{\ell}(i) \leq b < C_{\ell+1}(i) = T_{i(\ell+1)}, \\ 
    T_{j\ell} &= C_{\ell}(j) \leq a < C_{\ell+1}(j) = T_{j(\ell+1)}.
  \end{align*}
  Therefore, by \Cref{lem:all-covers}, $C_{\ell+1}$ does not cover $C_{\ell}$.
  If $T$ has an $(i, \ell+1)$-cell but not a $(j, \ell+1)$-cell, by a similar
  argument as in the previous case, we deduce that $n\leq b < a$. Lastly, if $T$
  does not have an $(i, \ell+1)$-cell, then $\#C_{\ell}\geq \#C_{\ell+1} + 2$.
  By \Cref{lem:all-covers}, $C_{\ell+1}$ does not cover $C_{\ell}$. Hence all
  cases lead to a contradiction. 
\end{proof}

To count maximal flags in $\Delta(\msfT_n)$, we consider {Gelfand--Tsetlin}
patterns, which are known to be in bijection with tableaux. For our purposes, a
\define{Gelfand--Tsetlin pattern of degree $n$} is a lower-triangular matrix
$A=(a_{ij})\in\Mat_n(\N_0)$, satisfying $a_{ij}\leq a_{(i+1)j} \leq
a_{(i+1)(j+1)}$ for all relevant values $j\leq i$.  We write $\GT_n$ for the set
of all Gelfand--Tsetlin patterns of degree~$n$.

The first part of the following result is well-known. It asserts that
a Gelfand--Tsetlin pattern records, in its $(n-i)$th off-diagonal, the
parts of the $i$th member the flag of partitions of a unique tableau,
and all tableaux arise in this way.

\begin{prop}\label{prop:GT} The map
  \begin{align}
    \Gamma: \SSYT_n &\longrightarrow \GT_n, \quad T \longmapsto
    \left(\lambda_{n+1-r}^{(n-r+s)}(T)\right)_{r\in[n],
      s\in[r]}\label{eqn:map-SSYT-GT}
  \end{align}
  is a bijection and maps reduced maximal tableaux to Gelfand--Tsetlin patterns
  whose set of entries $\{a_{ij}\mid 1\leq j \leq i \leq n\}$ is
  $\{1,\dots,\binom{n+1}{2}\}$. 
   The number of reduced maximal tableaux is 
  \begin{equation}\label{quant.thrall}
    \dfrac{\binom{n+1}{2}!\cdot\prod_{a=1}^{n-1}(a!)}{\prod_{b=1}^n ((2b-1)!)} .
  \end{equation}
\end{prop}

\begin{proof} 
  That $\Gamma$ is a bijection follows from~\cite[Sec.~7.10]{Stanley:Vol2}. The
  second claim follows from \Cref{prop:max-tab}, and \cite[Thm.~1]{Thrall/52}
  yields the final statement.
\end{proof}

The sequence defined by \eqref{quant.thrall} is
\href{https://oeis.org/A003121}{OEIS-sequence~A003121}
\cite{OEIS:A003121}. \Cref{fig:GT} exemplifies the bijection in \Cref{prop:GT}.

\begin{figure}[htb!]
  \centering
 $$\begin{array}{c} \begin{tikzpicture}
  \pgfmathsetmacro{\w}{0.35}
  \foreach \i [count = \k] in {8,4,2}{
          \foreach \j in {1,...,\i}{
                  \draw (\j*\w - \w, -\k*\w) -- ++(-\w,0) -- ++(0,\w) -- ++(\w,0) -- cycle;
          }
  }
          \node at (-0.5*\w, -0.5*\w) {{\footnotesize $1$}};
  \node at (0.5*\w, -0.5*\w) {{\footnotesize $1$}};
  \node at (1.5*\w, -0.5*\w) {{\footnotesize $1$}};
  \node at (2.5*\w, -0.5*\w) {{\footnotesize $1$}};
  \node at (3.5*\w, -0.5*\w) {{\footnotesize $2$}};
  \node at (4.5*\w, -0.5*\w) {{\footnotesize $2$}};
  \node at (5.5*\w, -0.5*\w) {{\footnotesize $3$}};
  \node at (6.5*\w, -0.5*\w) {{\footnotesize $3$}};
  \node at (-0.5*\w, -1.5*\w) {{\footnotesize $2$}};
  \node at (0.5*\w, -1.5*\w) {{\footnotesize $2$}};
  \node at (1.5*\w, -1.5*\w) {{\footnotesize $2$}};
  \node at (2.5*\w, -1.5*\w) {{\footnotesize $3$}};
  \node at (-0.5*\w, -2.5*\w) {{\footnotesize $3$}};
  \node at (0.5*\w, -2.5*\w) {{\footnotesize $3$}};
\end{tikzpicture}
 \end{array}
 \quad \longleftrightarrow \quad
 \begin{array}{c}
    \left(\begin{matrix}2 & & \\ 3 & 4 & \\4 & 6 & 8 \end{matrix}\right) 
 \end{array}$$
 \caption{A tableau in $\SSYT_3$ and its corresponding Gelfand--Tsetlin
   pattern in $\GT_3$.}
  \label{fig:GT}
\end{figure}

For $A=(a_{ij})\in \GT_n$, we define the polynomial
\begin{align*}
  \Psi_A(Y) &= \prod_{k=1}^n(1 - Y^k)^{d_k},
\end{align*}
where $d_k$ is defined as the number of pairs $(i,a)\in [n]\times \N_0$ such
that $a$ occurs $k$ times in the $(i-1)$th off-diagonal of $A$ and $k-1$ times
in the $(i-2)$th off-diagonal of~$A$. The polynomial $\Psi_A(Y)$ is defined in
\cite[(3)]{FM/16} and written as $p_A$. Observe that the set of pairs $(i,a)$
for a fixed $k\in [n]$ are in bijection with $\left\{(i,j) \in\mathscr{L}_T \mid
\#\Leg_T^+(i,j) = k\right\}$. This proves the following lemma.

\begin{lem}\label{lem:FM}
  With $\Gamma : \SSYT_n \to \GT_n$ as in \eqref{eqn:map-SSYT-GT}, for all $T\in
  \SSYT_n$,
  \[
    \Phi_T(Y) = \Psi_{\Gamma(T)}(Y).
  \]  
\end{lem}

Feigin--Makhlin establish a formula for the Hall--Littlewood polynomial
$P_\lambda(\bfx;t)$ associated with a partition $\lambda\in\mathcal{P}_n$ in
terms of the polynomials~$\Psi_A$. We write $\GT_\lambda$ for the set of
Gelfand--Tsetlin patterns corresponding to tableaux of shape $\lambda$. For
$A\in \GT_{\lambda}$ we write $\wt(A) = \wt(\Gamma^{-1}(A)) =
(\omega_1,\dots,\omega_n)$ for the weight of the corresponding tableau and set
$\bfx^{\wt(A)} = x_1^{\omega_1}\cdots x^{\omega_n}$. We have
(\cite[Thm.~1.1]{FM/16})

\begin{equation}\label{equ:FM}
P_\lambda(\bfx;t) = \sum_{A\in \GT_\lambda} \Psi_A(t) \bfx^{\wt(A)}.
\end{equation}

\subsubsection{Proof of \Cref{thm:CM}}\label{subsubsec:CM}

  For $n=1$, the statement follows since $\msfT_1=\{ \{1\}\}$, so we assume
  $n\geq 2$. By \Cref{prop:Bruhat-iso}, $\msfT_n\cup \{\varnothing\}\cong
  B_n^{[n-1]}$, where $J = \{s_1,\dots, s_{n-1}\}\subset S$. By
  \cite[Thm.~5.5]{BW:Bruhat}, the Stanley--Reisner ring of $\Delta(\msfT_n)$
  over $\Z$ is Cohen--Macaulay; hence $|\Delta(\msfT_n)|$ is Cohen--Macaulay
  over $\Z$.
  
  For $w, w'\in B_n^{[n-1]}$, we write 
  \begin{align*} 
    [w, w']^J &= \left\{z\in B_n^{[n-1]} \mid w\leq z \leq
  w'\right\}, &
    (w,w')^J &= [w,w']^J \setminus\{w,w'\}. 
  \end{align*}
  By \Cref{prop:Bruhat-iso}, $\msfT_n\setminus\{[n], \{n\}\}$ is isomorphic
  to $(1,\ w_2w_3\cdots w_n)^J$. This open interval is not \emph{full} in the
  sense of \cite[Sec.~4]{BW:Bruhat} because $s_2\in (1,\ w_2w_3\cdots w_n)$.
  Moreover
  \[ 
    \ell(w_2w_3\cdots w_n) - \ell(1) = \binom{n+1}{2} - 1. 
  \]
  By \cite[Thm.~5.4]{BW:Bruhat}, $|\Delta(\msfT_n\setminus\{[n],\{n\}\})|$ is
  homeomorphic to an $\left(\binom{n+1}{2}-3\right)$-ball. Thus,
  $|\Delta(\msfT_n)|$ is homeomorphic to an $\left(\binom{n+1}{2}-1\right)$-ball
  since it is obtained from the join of
  $|\Delta(\msfT_n\setminus\{[n],\{n\}\})|$ by the $1$-simplex.

  \Cref{prop:GT} now completes the proof of~\Cref{thm:CM}. \qed

\section{Symplectic Hecke series}\label{sec:hecke}

In \Cref{subsec:hecke.proof} we prove \Cref{thmabc:HLS.hecke}, showing
that the Hecke series $\Hecke_{n,\lri}$ are instantiations of the
Hall--Littlewood--Schubert series~$\HLS_n$. In \Cref{subsec:hecke.at.1} we apply a
result of Macdonald to obtain a surprisingly simple product expression
for these specific substitutions.

\subsection{Proof of \Cref{thmabc:HLS.hecke}}\label{subsec:hecke.proof} Write $q^{-\bfs}$ for
$(q^{-s_1},\dots,q^{-s_n})$. Then we have
\begin{align*}
  \lefteqn{ \Hecke_{n,\lri}(q^{-\bfs}, q^{N-s_0}Z)
    =\hat{\tau}(s_0,s_1,\dots, s_n,Z)} &&&(\text{\Cref{def:hecke}})\\ &=
  \sum_{\lambda\in \mathcal{P}_n} \sum_{m\geq
    \lambda_1}P_\lambda\left(q^{-\bfs};q^{-1}\right)\left(q^{N-s_0}Z\right)^m
  & & (\text{\cite[V.5.~(5.2)]{Macdonald/95}})\\ &=
  \frac{1}{1-q^{N-s_0}Z}\sum_{\lambda\in \mathcal{P}_n}P_\lambda\left(q^{-\bfs};q^{-1}\right)\left(q^{N-s_0}Z\right)^{\lambda_1}
  & &\\ &= \frac{1}{1-q^{N-s_0}Z}\sum_{\lambda\in \mathcal{P}_n}\sum_{A \in
    \GT_\lambda}\Psi_A(q^{-1})q^{-\wt(A)\bfs}\left(q^{N-s_0}Z\right)^{\lambda_1}
  & & (\text{\Cref{equ:FM}})\\ &= \frac{1}{1-q^{N-s_0}Z}\sum_{T\in \SSYT_n}
  \Phi_T(q^{-1}) q^{-\wt(T)\bfs}\left(q^{N-s_0}Z\right)^{\lambda_1^{(n)}(T)}
  & & (\text{\Cref{lem:FM}})\\ &= \frac{1}{1-q^{N-s_0}Z}\sum_{T\in
    \rSSYT_n}\Phi_T(q^{-1})\prod_{C\in T}\frac{q^{N-s_0-\sum_{i\in
        C}s_i}Z}{1-q^{N-s_0-\sum_{i\in C}s_i}Z} & & \\ &=
  \frac{1}{1-q^{N-s_0}Z}\HLS_n\left(q^{-1},\left(q^{N-s_0}Z\prod_{i\in
    C}q^{-s_i}\right)\right). & & (\text{\Cref{def:HLS}})
\end{align*}
Substituting $X=q^{N-s_0}Z$ and $x_i = q^{-s_i}$ for $i\in[n]$ finishes
the proof. \qed

\subsection{Symplectic Hecke series at $X=1$}\label{subsec:hecke.at.1}

The rational function $\Hecke_{n,\lri}(\bfx,X)$ is, for general $n$, very far
from a product of ``simple'' factors; see, for instance, \Cref{exa:hecke}
for~$n=3$. The next result, essentially due to Macdonald, states that setting
$X=1$ yields a neat factorization. Recall that we set
$\bfx_C = \prod_{i\in C} x_i$ for $C\subseteq [n]$. 

\begin{prop}[Hecke series at $X=1$]\label{prop:X=1.hecke} 
  Set~$x_0=1$. For all cDVR $\lri$ with residue field cardinality~$q$ we have
  \[
    \left.\left( \Hecke_{n,\lri}\left(\bfx,X\right)(1-X)\right)\right|_{X=1} = \HLS_n
    \left(q^{-1},\left( \bfx_C\right)_C\right) = \frac{\prod_{1\leq i<j \leq
      n}(1-q^{-1}x_ix_j)}{\prod_{0\leq i<j \leq n}(1-x_ix_j)}.
  \]
\end{prop}

\begin{proof}
  For a partition $\lambda$, let $P_{\lambda}(\bm{x}; t)$ be the
  Hall--Littlewood polynomial. Then
  \begin{align*}
    \HLS_n(q^{-1}, (\bm{x}_C)_C) &= \left.\left( \Hecke_{n,\lri}\left(\bfx,X\right)(1-X)\right)\right|_{X=1} & & (\text{\Cref{thmabc:HLS.hecke}})\\ 
    &= \sum_{\lambda\in\mathcal{P}_n} P_{\lambda}(\bm{x}; q^{-1}) & & (\text{\cite[V.5.~(5.2)]{Macdonald/95}}) \\
    &= \frac{\prod_{1\leq i<j \leq
    n}(1-q^{-1}x_ix_j)}{\prod_{0\leq i<j \leq n}(1-x_ix_j)} & & (\text{\cite[III.5.~Ex.~1]{Macdonald/95}}).
  \end{align*}
  Since the equalities hold for infinitely many prime powers $q$, the result
  follows.
\end{proof}

\begin{remark}
  The formula given in \Cref{prop:X=1.hecke} bears a striking
  resemblance with the Gindikin--Karpelevich formula
  \cite[(1.1)]{BumpNakasuji/10} for a $\mfp$-adic integral over
  unipotent groups. It is of great interest to reveal a
  conceptual explanation for this similarity.
\end{remark}

\Cref{prop:X=1.hecke} may be seen as a $q^{-1}$-analog of a
classical formula known as Littlewood identity;
see~\cite[p.~76]{Macdonald/95}. We write $s_\lambda(x_1,\dots,x_n)$
for the Schur polynomial associated with the
partition~$\lambda\in\mathcal{P}_n$. Taking the limit $q \rightarrow
\infty$ in \Cref{prop:X=1.hecke}, we obtain the following result due
to Schur.

\begin{cor}\label{obs:mac} We have
  \begin{equation*}\label{equ:HLS.lit}
    \sum_{\lambda\in\mathcal{P}_n} s_\lambda(\bfx) = \HLS_n\left(
    0, \left(\bfx_C\right)_C\right) = \sum_{T\in\SSYT_n}
    \bfx^{\wt(T)} = \prod_{0\leq i < j \leq n} \frac{1}{1-x_ix_j}.
  \end{equation*}
\end{cor}
In a similar vein, \Cref{prop:X=1.hecke} generalizes a number of
identities for generating functions enumerating tableaux $T$ by
statistics that factor over their weight~$\wt(T)$.

\section{Coarsening Hall--Littlewood--Schubert series}\label{sec:HLS.coarse}

We write $\HLS_n(Y,X)=\HLS_n(Y,(X)_C)$ for the bivariate rational function
obtained by coarsening the variables $X_C$ to a single variable~$X$ for all
non-empty $C\subseteq [n]$. In this section we focus on (fine and coarse)
instances of Hall--Littlewood--Schubert series at the special values $Y=0$, $Y=1$, and~$Y=-1$.

\subsection{$\HLS_n$ at $Y=0$}
\label{subsec:HLS.Y=0}

The starting point of this section is the observation that the Hall--Littlewood--Schubert series
specializes to a fine Hilbert series of a Stanley--Reisner. Let $\SR_n =
\SR(\Delta(\mathsf{T}_n))$ be the Stanley--Reisner ring over $\Z$ of
$\Delta(\mathsf{T}_n)$ and $\Hilb\left(\SR_n, (X_C)_C\right)$ its fine Hilbert
series. A general reference to Stanley--Reisner rings and their Hilbert series
is \cite[Ch.~II]{Stanley:commutative}.

\begin{lemma}\label{lem:HLS.HSR}
 We have
$$ 
\Hilb\left(\SR_n,(X_C)_C\right) = \HLS_n\left( 0 ,
  \left(X_C\right)_C\right).$$
\end{lemma}

\begin{proof}
  Annihilating $Y$ in $\HLS_n(Y,\bfX)$ has the effect of effacing the leg
  polynomials $\Phi_T(Y)$ in the sum defining $\HLS_n$:
  $$
    \HLS_n\left(0,(X_C)_C\right) = \sum_{T \in \rSSYT_n}~\prod_{C\in T}
    \frac{X_C}{1-X_C}.
  $$ 
  The claim follows from \Cref{lem:tab.flags}.
\end{proof}

We now consider $\HLS_n(0,X)$, the coarse Hilbert series of $\SR_n$.

\begin{prop}\label{cor:depth}
  Write $r = \binom{n+1}{2}$. There exist $h_{n,0},\dots, h_{n,r}\in
  \N_0$ such that
  \begin{align}
    \HLS_n(0,X) &=  \dfrac{\sum_{i=0}^rh_{n,i}X^i}{(1 - X)^r}, & 
    \sum_{i=0}^r h_{n,i} &= \dfrac{\binom{n+1}{2}!\cdot\prod_{a=1}^{n-1}(a!)}{\prod_{b=1}^n
    ((2b-1)!)}.\label{equ:h-vec} 
  \end{align}
Moreover, $h_{n,0}=1$ and $h_{n,1} = 2^{n} - 1 - \binom{n+1}{2}$.  
\end{prop}

\begin{proof}
  By \Cref{thm:CM}, $\SR_n$ is Cohen--Macaulay over $\Z$, so by
  \cite[Cor.~II.3.2]{Stanley:commutative} there exist $h_{n,0},\dots,
  h_{n,r}\in\N_0$ such that
  \begin{align*}
    \HLS_n(0,X) &= \Hilb\left(\SR_n,X\right) =
    \dfrac{\sum_{i=0}^rh_{n,i}X^i}{(1 - X)^r} .
  \end{align*}
  For the second equation, note that $h_{n,0}+\cdots + h_{n,r}$ is the
  number of maximal flags in $\Delta(\msfT_n)$; see for
  example~\cite[p.~58]{Stanley:commutative}. \Cref{thm:CM} yields
  their count. The last two statements follow from trivial
  general statements about $h$-vectors of simplicial
  complexes: $\msfT_n$ has cardinality $2^n-1$ and rank
  $\binom{n+1}{2}-1$; see~\Cref{cor:graded-rank}.
\end{proof}

We expect that $\HLS_n(0,X)$ has further interesting features
reflecting algebraic properties of $\SR_n$ and topological properties
of $\Delta(\msfT_n)$:

\begin{conj}\label{conj:depth}
  Use the notation of \Cref{cor:depth} and set $k = \binom{n-1}{2}$. Then
  \begin{enumerate}[label=$(\roman*)$]
    \item\label{parti} $h_{n,i} > 0$ for all $i\in [k]_0$ and $h_{n,k+1}=\cdots = h_{n,r} =
    0$;
    \item\label{partii} $h_{n,i} = h_{n,k-i}$ for all $i\in [k]_0$.
  \end{enumerate}
\end{conj}

\Cref{conj:depth}\ref{parti} suggests that the essential topological
information about $\Delta(\msfT_n)$ is captured in a lower-dimensional
simplicial complex than shown in
\Cref{thm:CM}. \Cref{conj:depth}\ref{partii} hints at a Poincar\'e
duality of sorts: this lower-dimensional simplicial complex is
homeomorphic to a sphere and its associated Stanley--Reisner ring is
Gorenstein. The following data provide evidence in favor
of~\Cref{conj:depth}. Like all other computations discussed here they
were performed using \verb=SageMath=~\cite{sagemath}.

\begin{ex}
  The following is the evidence we have for \Cref{conj:depth}:
  \begin{align*}
    \HLS_1(0,X)(1-X)^{1\phantom{1}} &= 1,\\
    \HLS_2(0,X)(1-X)^{3\phantom{1}} &= 1,\\
    \HLS_3(0,X)(1-X)^{6\phantom{1}} &= 1+X,\\
    \HLS_4(0,X)(1-X)^{10} &= 1+5X+5X^2+X^3,\\
    \HLS_5(0,X)(1-X)^{15} &= 1+16X+70X^2+112X^3+70X^4+16X^5+X^6,\\
    \HLS_6(0,X)(1-X)^{21} &= 1+42X + 539X^2 + 2948X^3 + 7854X^4 + 10824X^5 \\ 
    &\quad + 7854X^6 + 2948X^7 + 539X^8+42X^9+X^{10}.
  \end{align*}

  \vspace{-1.5em}
  \exqed
\end{ex}

\subsection{$\HLS_n$ at $Y=1$}
\label{subsec:HLS.Y=1}

In \Cref{subsec:HLS.Y=0}, we substituted $0$ for $Y$, effectively effacing
$\Phi_T$ for all~$T\in\SSYT_n$. By substituting $1$ for $Y$ instead, we turn
$\Phi_T$ into an indicator function, deciding whether or not the columns of $T$
form a flag under set-containment. These flags are also known as \emph{weak
orders}. They are in bijection with the faces of the barycentric subdivision of
an $n$-simplex.

More precisely, let $\Delta_n$ be the simplex with vertex set $[n]$ and
$\mathrm{sd}(\Delta_n)$ its barycentric subdivision. The \define{weak order zeta
function} (cf.~\cite[Def.~2.9]{SV1/15}) is
\begin{equation}\label{def:Iwo}
  \begin{split}
    I^{\textup{wo}}_n((X_C)_{\varnothing\neq C\subseteq [n]}) &=
    \Hilb(\SR(\mathrm{sd}(\Delta_n)), (X_C)_C) \\ &= \sum_{\varnothing \neq
      C_1 \subset \dots \subset C_\ell\subseteq [n]} ~\prod_{i=1}^\ell
    \frac{X_{C_i}}{1-X_{C_i}} \in \Q\left(\bfX\right).
  \end{split}
\end{equation}
For each $n\in\N$, the $n$th \define{Eulerian polynomial} is $ \mathrm{E}_n(X) =
\sum_{w\in S_n} X^{\mathrm{des}(w)} $, where $\mathrm{Des}(w) = \{i\in [n-1]
\mid w(i+1) < w(i) \}$ and $\mathrm{des}(w)=\#\mathrm{Des}(w)$ and $S_n$ is the
symmetric group on $[n]$. See~\cite[Ch.~1]{Petersen/15}.

\begin{prop}\label{obs:weak.order}
  We have
  \begin{align*}
    I^{\textup{wo}}_n((X_C)_C) &= \HLS_n\left(1, \left(X_C\right)_C\right).
  \end{align*}
  In particular, 
  \begin{equation*}\label{equ:HLS.Y=1.X.new}
    \dfrac{\mathrm{E}_n(X)}{(1 - X)^n} =  \HLS_n(1, X).            
  \end{equation*}
\end{prop}

\begin{proof}
  Let $T\in\mathrm{rSSYT}_n$. It suffices to observe that $\Phi_T(Y)$ is
  divisible by $1 - Y$ if and only if the columns of $T$ do not form a flag of
  subsets under set containment. Otherwise $\Phi_T(Y)=1$. This follows from
  combining \Cref{rem:PhiT.triv} with \Cref{prop:Phi=phi}. The second claim
  follows from \cite[Thm.~9.1]{Petersen/15}.
\end{proof}

\subsection{$\HLS_n$ at $Y=-1$}\label{subsec:HLS.Y=-1}
Perhaps the most surprising phenomena we
observe occurs when we substitute $Y$ with $-1$ in $\HLS_n(Y,X)$. Whereas
with the other two substitutions to $0$ and $1$ we could draw upon
knowledge of the simplicial complexes $\Delta(\msfT_n)$
and~$\mathrm{sd}(\Delta_n)$, there does not appear to be a
simplicial complex associated with the substitution to
$-1$; see~\Cref{ex:-1}. Nevertheless, $\HLS_n(-1,X)$ seems to have a number of
remarkable properties, which we record in~\Cref{conj:HLS.Y=-1.coarse}.

The special value $\Phi_T(-1)$ is, for any tableau~$T$, either zero or
a power of~$2$.

\begin{problem}
  Characterize the class of tableaux $T$ satisfying~$\Phi_T(-1)=0$.
\end{problem}

A solution to this problem might be analogous to the
partition-theoretic characterization of the vanishing of the
Hall--Littlewood $Q$ polynomials $Q_\lambda(\bm{x};-1)$ at $t=-1$ in terms
of the Hall--Littlewood polynomials $P_{\lambda}(\bm{x};-1)$;
see~\cite[III.8~(8.7)]{Macdonald/95}.

\begin{conj}\label{conj:HLS.Y=-1.coarse}
  Let $r=\binom{n+1}{2}$. There exist $h_{n,0}^-,\dots,
  h_{n,r-1}^-\in\N_0$ such that
  \begin{align*}
    \HLS_n(-1, X) &= \dfrac{\sum_{i=0}^{r-1}h_{n,i}^-X^i}{(1 - X)^r}, & 
    \sum_{i=0}^{r-1} h_{n,i}^- &= \dfrac{\binom{n}{2}!}{\prod_{i=1}^{n-1}(2i-1)^{n-i}},
  \end{align*}
  and $h_{n,1}^- = 2^{n} - 1 - \binom{n+1}{2}$. Moreover, $h_{n,i}^-=0$ if and
  only if $(n,i)=(2,1)$.
\end{conj}

\begin{remark}
  The sum of the $h_{n,i}^-$ in \Cref{conj:HLS.Y=-1.coarse} appears to coincide
  with the number of maximal chains in the poset of Dyck words of
  length~$2(n+1)$, ordered by inclusion.  Woodcock enumerated these paths in her
  PhD thesis and provided a bijection of these maximal paths to standard Young
  tableaux in staircase tableaux; see \cite[Sec.~4.3 \& 4.4]{Woodcock}
  and~\href{https://oeis.org/A005118}{OEIS-sequence~A005118}
  \cite{OEIS:A005118}.
\end{remark}

\begin{ex}\label{ex:-1}
  The following is the evidence we have for~\Cref{conj:HLS.Y=-1.coarse}.
  \begin{align*}
    \HLS_1(-1,X)(1-X)^{1\phantom{1}} &= 1,\\
    \HLS_2(-1,X)(1-X)^{3\phantom{1}} &= 1+X^2,\\
    \HLS_3(-1,X)(1-X)^{6\phantom{1}} &= 1+X+6X^2+6X^3+X^4+X^5,\\
    \HLS_4(-1,X)(1-X)^{10} &= 1+5X+32X^2+120X^3+226X^4+\dots+X^9,\\
    \HLS_5(-1,X)(1-X)^{15} &= 1+16X+179X^2+1568X^3+8545X^4 \\
    &\quad +30448X^5+63979X^6+83392X^7+\dots+X^{14}.
  \end{align*}
  By Macaulay's theorem~\cite[Thm.~II.2.2]{Stanley:commutative}, these
  polynomials are not $h$-polynomials of simplicial complexes, except in the
  case of $n=1$. \exqed
\end{ex}

\section{Hall--Littlewood--Schubert series as $\mfp$-adic integrals}
\label{sec:HLS.int}

In this section, we explore Hall--Littlewood--Schubert series from the
perspective of $\mfp$-adic integration. It allows us to connect these series,
in \Cref{subsec:sym.int}, to integrals over $\mfp$-adic symplectic groups and
to pro-isomorphic zeta functions of nilpotent groups. We use it to simplify
some of the delicate work in~\cite{BGS/22}. In \Cref{subsec:fun.eq} we use a
different expression of $\HLS_n$ in terms of $\mfp$-adic integrals associated
with Igusa's local zeta function to prove \Cref{thmabc:HLS.funeq}.

\subsection{Symplectic integrals and zeta functions of groups}\label{subsec:sym.int}
We show that Hall--Littlewood--Schubert series specialize to $\mfp$-adic
integrals associated with symplectic groups and thus to pro-isomorphic zeta
functions of nilpotent groups. We start with the former.

For a non-archimedean local field $K$, we denote by $\GSp_{2n}^+(K)$ the set of
integral invertible elements in the group of general symplectic similitudes over
the non-archimedean local field $K$. Let $\mu$ be the Haar measure on
$\GSp_{2n}^+(K)$ such that $\mu(\GSp_{2n}^+(\lri))=1$ for the ring of integers
$\lri$ of $K$, and let $|\cdot|_{\mfp}$ be the $\mfp$-adic norm. Recall $\maj(C)
= \sum_{i\in C}i$. The next result is essentially due to Macdonald.

\begin{thm}\label{thm:GSp.hls}
  Let $K$ be a non-archimedean local field, with ring of integers
  $\lri$ with residue field cardinality~$q$. Then
  \begin{align*}
    \int_{\GSp_{2n}^+(K)}|\det A|_{\mfp}^{s}\, \tud \mu &= \frac{1}{1-q^{-ns}}\HLS_n\left(q^{-1},\left(q^{\maj(C)-ns}\right)_C\right). 
  \end{align*}
\end{thm}

\begin{proof}
  The integral in the statement is equal to the zeta function
  $\zeta(s, \omega_{\mathrm{tr}})$ defined
  in~\cite[p.~303]{Macdonald/95}, where $\omega_{\mathrm{tr}}$ is the
  trivial spherical function given by $(s_0,s_1,\dots, s_n) = (N, -1,
  -2, \dots, -n)$; see Satake's calculations~\cite[App.~1,
    4]{Satake/63}. Note that~$\zeta(s,\omega)$, for an arbitrary
  spherical function $\omega$, is defined in terms of the $n$th root
  of the determinant. By Macdonald~\cite[p.~304]{Macdonald/95}
  and~\eqref{def:hecke}, we have
  \begin{align*}
    \int_{\GSp_{2n}^+(F)}|\det A|_{\mfp}^{s} \,\tud \mu &=
    \hat{\tau}(N,{-1},\dots,{-n},q^{-ns}) = \Hecke_{n,\lri}(q^1,\dots,q^n,q^{-ns}).
  \end{align*}
  The statement follows from~\Cref{thmabc:HLS.hecke}.
\end{proof}

Now we apply \Cref{thm:GSp.hls} to zeta functions of groups. Write
$\mcH_n$ for the $n$-fold centrally amalgamated product of the
Heisenberg group scheme~$\mcH$. Let $F$ be a number field of degree
$d$ with ring of integers $\Gri$. For a non-zero prime
ideal~$\mfp\in\mathrm{Spec}(\Gri)$, we write $\Gri_{\mfp}$ for the
completion of $\Gri$ at $\mfp$, and $F_\mfp$ for the field of
fractions of $\Gri_{\mfp}$. The $\Gri$-rational points of $\mcH_n$ can
be identified with the set
\[
  \left\{\begin{pmatrix}
  1 & u^{\mathrm{t}} & w \\ & \Id_{n} & v \\ & & 1 
  \end{pmatrix} ~\middle|~ u,v\in \Gri^{n}, w\in\Gri \right\} \subset \GL_{n+2}(\Gri).
\]
The abstract group $\mcH_n(\Gri)$ is a finitely generated nilpotent group.

Introduced in~\cite{GSS/88}, the \emph{pro-isomorphic zeta function}
$\zeta^{\wedge}_{G}(s)$ of the finitely generated nilpotent group $G$
enumerates subgroups $H\leq G$ such that $H$ and $G$ have the same
profinite completions. That the
pro-isomorphic zeta functions of the groups $\mcH_n(\Z)$ may be
expressed in terms of $\mfp$-adic integrals associated with symplectic
groups was already discussed in~\cite[Sec.~3.3]{duSL/96}. We use this
connection to link the Hall--Littlewood--Schubert seris $\HLS_n$ to
pro-isomorphic zeta functions of groups.

\begin{cor}
  Let $F$ a number field of degree $d$ with ring of
  integers~$\Gri$. Writing $q_{\mfp}$ for the cardinality of the
  residue field of $\Gri_{\mfp}$, we have
  \begin{align*}
    \zeta^{\wedge}_{\mcH_n(\Gri)}(s) &= \prod_{(0) \neq \mfp
    \in\Spec(\Gri)} \frac{1}{1-q_{\mfp}^{2dn-(n+1)s}}\cdot\HLS_n\left(q_{\mfp}^{-1},\left(q_{\mfp}^{\maj(C)+2dn-(n+1)s}\right)_C\right) . 
  \end{align*}
\end{cor}

\begin{proof}
  Let $\mu_{\mfp}$ be the Haar measure on $\GSp_{2n}^+(F_{\mfp})$ such that
  $\mu_{\mfp}(\GSp_{2n}^+(\Gri_{\mfp}))=1$ for non-zero
  $\mfp\in\mathrm{Spec}(\Gri)$. By \cite[Prop.~5.4]{BGS/22},
  \begin{align*}
    \zeta^{\wedge}_{\mcH_n(\Gri)}(s) &= \prod_{(0) \neq \mfp
      \in\Spec(\Gri) } \int_{\GSp_{2n}^+(F_{\mfp})}|\det
    A|_{\mfp}^{\left(1 + \frac{1}{n}\right)s -2d} \tud \mu_{\mfp}.
  \end{align*}
  The result follows by \Cref{thm:GSp.hls}.
\end{proof}

Berman, Glazer, and Schein show in \cite[(42)]{BGS/22} that the
following lemma holds using essentially \cite[Lem.~5.7]{BGS/22} whose
proof uses a delicate combinatorial argument. With our framework, we
can simplify their argument.

Recall that $S_n$ is the Coxeter group of type $A_{n-1}$, whose Coxeter length
function we denote by~$\ell$. Recall $\Des(w)\subseteq[n]$ from
\Cref{subsec:HLS.Y=1} for $w\in S_n$.

\begin{lem}\label{lem:BGS}
  Let $K$ be a non-archimedean local field, with ring of integers
  $\lri$ with residue field cardinality~$q$. For ${X}_i =
  q^{\binom{n+1}{2} - \binom{i+1}{2}-ns}$, we have
  \begin{align*}
    \int_{\GSp_{2n}^+(K)}|\det A|_{\mfp}^{s}\, \tud \mu &= \frac{\sum_{w\in S_n} q^{-\ell(w)} \prod_{i\in
    \Des(w)}{X}_i}{\prod_{i=0}^n(1-{X}_i)} . 
  \end{align*}
\end{lem}

\begin{proof}
  It is well-known that
  \begin{align*}
    \frac{\sum_{w\in S_n} q^{-\ell(w)} \prod_{i\in
      \Des(w)}{X}_i}{\prod_{i=0}^n(1-{X}_i)} &=
    \frac{1}{1-X_0}\igusa_n\left(q^{-1},\left(X_i\right)_{i\in[n]}\right);
  \end{align*}
  see, for instance, \cite[Rem.~3.12]{SV2}. Note that the Hecke series
  $\hat{\tau}(s_0,\dots, s_n, Z)$ is $B_n$-invariant (see
  \cite[p.~302]{Macdonald/95}). Therefore, by \eqref{def:hecke} for
  all $w\in S_n$ and $k\in [n]$,
  \begin{equation}\label{eqn:Bn-invariance}
    \begin{split}
      \Hecke_{n,\lri}(x_{w(1)},\dots, x_{w(n)}, X) &= \Hecke_{n,\lri}(\bm{x}, X), \\ 
      \Hecke_{n,\lri}(x_1,\dots, x_{k-1},x_k^{-1},x_{k+1},\dots, x_n,x_kX) &= \Hecke_{n,\lri}(\bm{x},X).
    \end{split}
  \end{equation}
  Hence, we have
  \begin{align*}
    \int_{\GSp_{2n}^+(K)}|\det A|_{\mfp}^{s}\, \tud \mu &= \dfrac{1}{1 - q^{-ns}}\HLS_n\left(q^{-1}, \left(q^{\maj(C) - ns}\right)_C\right) & & (\text{\Cref{thm:GSp.hls}}) \\ 
    &= \Hecke_{n,\lri}\left(q^{1},q^{2}, \dots, q^{n}, q^{-ns}\right) & & (\text{\Cref{thmabc:HLS.hecke}}) \\
    &= \Hecke_{n,\lri}\left(q^{-1},q^{-2}, \dots, q^{-n}, q^{\binom{n+1}{2}-ns}\right) & & (\text{\Cref{eqn:Bn-invariance}})\\
    &= \dfrac{1}{1 - X_0}\HLS_n\left(q^{-1}, \left(q^{\maj([n]\setminus C) - ns}\right)_C\right) & & (\text{\Cref{thmabc:HLS.hecke}}) \\
    &= \frac{1}{1-X_0}\igusa_n\left(q^{-1},\left(X_i\right)_{i\in[n]}\right).  & & (\text{\Cref{cor:igusa.hs.hls}}) \qedhere
  \end{align*}
\end{proof}

\subsection{Functional equations}\label{subsec:fun.eq}
As before, fix $n\in\N$ and a cDVR $\lri$. Let $\bm{s} =
(s_C)_{\varnothing\neq C\subseteq [n]}$ be complex variables.  For $T
\in \SSYT_n$ and $C\subseteq[n]$, we write $m_T(C)\in \N_0$ for the
multiplicity of $C$ as a column of~$T$.  In \Cref{lem:C-width} we
express $m_T(C)$ in terms of the parts of members of the flag of
partitions~$\lambda^\bullet(T)$ associated with~$T$;
see~\eqref{def:flop.tab}. These expressions, in turn, we formulate in
terms of rational functions that inform the integrand of the
$\mfp$-adic integral~\eqref{def:pad.int}.

Let $\bm{z} = (z_{ij})_{1\leq i \leq j \leq n}$ be indeterminates, and let
$\Xi\in\Tr_n(\Z[\bm{z}])$ be the upper-triangular $(n\times n)$-matrix with
entries $z_{ij}$ for all $1\leq i \leq j\leq n$. For $k\in [n]$, write
$\Xi^{(k)}$ for the lower-right $(k\times k)$-submatrix of $\Xi$. For each $\ell
\in [n]$, let $\rho_{\ell}^{(k)}\subset \Z [\bm{z}]$ be the set of $\ell\times
\ell$ minors of $\Xi^{(k)}$. Note that $\rho_{\ell}^{(k)}$ comprises homogeneous
polynomials of degree $\ell$ in the variables $\{z_{ij} \mid n-k+1\leq i \leq j
\leq n \}$.  Set $\rho^{(n+1)}_{\ell} = \rho_{\ell}^{(0)} = \{1\} =
\rho^{(k)}_{-m}$ for all $\ell\in\Z$, $k\in \N$, and $m\in \N_0$.

Let $\mathrm{d}\mu$ be the unique normalized Haar measure on $\lri^{\binom{n +
1}{2}}\cong \mathrm{Tr}_n(\lri)$ such that $\mu(\lri^{\binom{n + 1}{2}}) = 1$.
For a finite set $\mathcal{X} \subset \lri$, let 
\begin{align*} 
  v_{\mathfrak{p}}(\mathcal{X}) &= \min\{v_{\mathfrak{p}}(f) : f\in\mathcal{X}\}, & \| \mathcal{X} \| &= \max\{|f|_{\mathfrak{p}} : f\in \mathcal{X}\} = q^{-v_{\mathfrak{p}}(\mathcal{X})}, 
\end{align*}
where $v_{\mathfrak{p}}$ and $|\cdot |_{\mathfrak{p}}$ are the $\mathfrak{p}$-adic valuation and norm, respectively. For
sets $S,S'\subseteq \Z[\bm{z}]$, let $S \cdot S' = \{s\cdot s' \mid
s\in S, s'\in S'\}$.  We use $\prod$ to denote products of several
factors of sets.  Recall that, given $C\subseteq [n]$, we write $C(k)$ for the
$k$th-smallest element of $C$, and that $C(\#C+1) = n + 1$. Informally speaking,
we extend $C$ by the element~$n+1$.

We further define four sets of polynomials in $\Z [\bm{z}]$:
\begin{equation}\label{def:four.poly}
  \begin{array}{cc}
    \begin{aligned}
      R_{n, C}^{\num} &= \bigcup_{k=1}^{\#C} \left( \!\rho_{C(k) - k
      + 1}^{(C(k))} \!\prod_{\ell\in [\#C]\setminus\{k\}}\! \rho_{C(\ell) - \ell}^{(C(\ell))} \!\right), \\ 
      L_{n, C}^{\num} &= \!\bigcup_{k=1}^{\#C+1}\!
      \left( \!\rho_{C(k) - k - 1}^{(C(k) - 1)} \!\prod_{\ell\in [\#C +1 ]\setminus\{k\}}\!
      \rho_{C(\ell) - \ell}^{(C(\ell) - 1)}\!\right),
    \end{aligned} & 
    \begin{aligned}
      R_{n, C}^{\den} &= \prod_{k=1}^{\#C} \rho_{C(k)
    - k}^{(C(k))}, \\
      L_{n,
      C}^{\den} &= \prod_{k=1}^{\#C+1} \rho_{C(k) - k}^{(C(k) -
      1)}.
    \end{aligned}
  \end{array}
\end{equation}
For a matrix $M\in\Tr_n(\lri)$ and a set $S\subset \Z[\bm{z}]$, we
write $S(M)\subset \lri$ for the evaluation of the polynomials in $S$
evaluated at the entries of~$M$. We fix a basis for $\lri^n$, so that $\Lambda(M)\leq \lri^n$ is the lattice
generated by the rows of $M$ for $M\in\mathrm{Tr}_n(\lri)$. 

\begin{lem}\label{lem:C-width}
  Let $M\in \Tr_n(\lri)$ be non-singular and
  $T=T^{\bullet}(\Lambda(M))\in\SSYT_n$. For $C\in T$ we have
  \[ 
    m_{T}(C) = v_{\mathfrak{p}}(R_{n, C}^{\num}(M)) + v_{\mathfrak{p}}(L_{n, C}^{\num}(M)) - v_{\mathfrak{p}}(R_{n, C}^{\den}(M)) - v_{\mathfrak{p}}(L_{n, C}^{\den}(M)) .
  \] 
\end{lem}

\begin{proof}
 Let
  \begin{align*} 
    r &= \min\left\{ \lambda_k^{(C(k))}(T) ~\middle|~ k\in [\#C] \right\}, \\ 
    \ell &= \max \left\{\lambda_k^{(C(k) - 1)}(T) ~\middle|~ k\in [\#C+1] \right\}.
  \end{align*}
  With this terminology, the first occurrence of $C$ as a column of $T$ is in
  column $\ell+1$, the last in column~$r$. Hence $m_{T}(C) = r - \ell$. For
  $k\in [n]$ and $i\in [n+1]$ we have
  \begin{align*}
    \lambda_i^{(k)}(T) =
    v_{\mathfrak{p}}(\rho_{k-i+1}^{(k)}(M)) -
    v_{\mathfrak{p}}(\rho_{k-i}^{(k)}(M)).
  \end{align*}
  The lemma follows as
  \begin{align*}
    q^{-r} = \left\| \left\{ \dfrac{\rho_{C(k) - k + 1}^{(C(k))}(M)}{\rho_{C(k) - k}^{(C(k))}(M)} ~\middle|~ k\in [\#C] \right\}\right\| &= \dfrac{\left\| R_{n, C}^{\num}(M) \right\|}{\left\| R_{n, C}^{\den}(M) \right\|} , \\
    q^{\ell} = \left\| \left\{ \dfrac{\rho_{C(k) - k - 1}^{(C(k) -
            1)}(M)}{\rho_{C(k) - k}^{(C(k) - 1)}(M)}
        ~\middle|~ k \in [\#C+1] \right\} \right\| &= \dfrac{\left\| L_{n,
          C}^{\num}(M) \right\|}{\left\| L_{n, C}^{\den}(M)
      \right\|}.\qedhere
  \end{align*}
\end{proof}

For $T\in\SSYT_n$,
let
\begin{align*}
  \mathscr{M}_{n, T}(\lri) &= \left\{ M\in \Tr_n(\lri) ~\middle|~ \forall
  \varnothing \neq C\subseteq [n],\ \dfrac{\left\| L_{n, C}^{\num}(M)
    \right\|\left\| R_{n, C}^{\num}(M) \right\|}{\left\| L_{n,
      C}^{\den}(M) \right\| \left\| R_{n, C}^{\den}(M) \right\|}
  = q^{-m_T(C)}\right\}.
\end{align*}
By~\Cref{lem:C-width}, this is the set of matrices $M\in\Tr_n(\lri)$ such
that~$T^\bullet(\Lambda(M))=T$.  We obtain a disjoint union
\begin{equation}\label{equ:disjoint}
  \Tr_n(\lri) =
  \bigsqcup_{T\in\SSYT_n}\mathscr{M}_{n, T}(\lri).
\end{equation}
Recall the definition of the weight $\wt(T) = (\omega_1,\dots, \omega_n)$ in
\Cref{subsec:tab}.

\begin{prop}\label{prop:T-measure}
  For $T\in \SSYT_n$ we have
  \begin{align*}
    \mu(\mathscr{M}_{n, T}(\lri)) &= f_{n,T}^{\mathrm{in}}(\lri) \cdot (1 - q^{-1})^n \prod_{i=1}^n q^{-i\omega_{n-i+1}}.
  \end{align*}
\end{prop}

\begin{proof} We note that, for a lattice $\Lambda$ with Hermite composition
 $(\delta_1,\dots,\delta_n)$, we have
  \begin{equation*}
    \mu\left(\left\{ M \in \mathrm{Tr}_n(\lri) ~\middle|~ \Lambda(M) =
    \Lambda \right\}\right) = (1-q^{-1})^n \prod_{i=1}^n
    q^{-i\delta_i}.
  \end{equation*}
  If $\Lambda$ has intersection tableau $T = T^\bullet(\Lambda)$ of
  weight $(w_1,\dots,w_n)$ and $M\in\Tr_n(\lri)$ with
  $\Lambda(M)=\Lambda$ then $\delta_i = w_{n-i+1}$ for all $i\in[n]$.
  The claim follows as, by~\Cref{lem:C-width},
  \begin{align*}
    f_{n,T}^{\mathrm{in}}(\lri) &= \dfrac{\mu(\mathscr{M}_{n,
        T}(\lri))}{\mu\left(\left\{ M \in \mathrm{Tr}_n(\lri) ~\middle|~ \Lambda(M) =
    \Lambda \right\}\right) }.\qedhere
  \end{align*}
\end{proof}

We now express $\HLS_{n}(q^{-1}, (q^{\Schubdim{n}{C}-s_C})_C)$ in terms of a
$\mathfrak{p}$-adic integral, a key step towards our proof of
\Cref{thmabc:HLS.funeq}.  For variables $\bm{s}' = (s_1', \dots, s_n')$, we set
 \begin{align}\label{def:pad.int}
    I_{n, \lri}(\bm{s}, \bm{s}') &= \int_{\Tr_n(\lri)}
    \prod_{\varnothing \neq C \subseteq [n]} \left(\dfrac{\left\|
      R_{n, C}^{\num} \right\| \left\| L_{n,
        C}^{\num}\right\|}{\left\| R_{n, C}^{\den} \right\|\left\|
      L_{n, C}^{\den} \right\|}\right)^{s_C} \prod_{i=1}^n
    |z_{ii}|^{s_i'} \,\mathrm{d}\mu.
  \end{align}
 
\begin{prop}\label{prop:hls.integral} We have
  \begin{align*}\label{eqn:hls.integral}
    \HLS_{n}\left(q^{-1}, \left(q^{\Schubdim{n}{C}-s_C}\right)_C\right) & = (1 -
    q^{-1})^{-n}I_{n,\lri}(\bfs, -1,-2,\dots,-n).
  \end{align*}
\end{prop}

\begin{proof}
  Given $\Lambda \in \mcL(\lri^n)$, we write $T(\Lambda)\in\SSYT_n$
  for the tableau associated with~$\Lambda$. We have
  \begin{equation}\label{eqn:hls-SSYT} 
    \begin{split}
      \HLS_{n}\left(q^{-1}, \left(q^{\Schubdim{n}{C}-s_C}\right)_C\right) &= \sum_{\Lambda \in \mcL(\lri^n)} ~\prod_{C\in T(\Lambda)}
      q^{-s_C} \\
      &= \sum_{T\in \SSYT_n} f_{n,T}^{\mathrm{in}}(\lri)\cdot
      q^{-\sum_{\varnothing \neq C\subseteq [n]} m_T(C)s_C}.
    \end{split}
  \end{equation}

  By \Cref{prop:lattice-part-flag}, the $\mathfrak{p}$-adic valuation of the
  diagonal elements of matrices in $\mathscr{M}_{n,T}(\lri)$ are
  constant. Therefore, on each $\mathscr{M}_{n,T}(\lri)$ the integrand of
  $I_{n,\lri}(\bm{s})$ is constant, namely
  $$q^{\sum_{i=1}^n i\omega_{n-i+1}-\sum_{\varnothing \neq C\subseteq
      [n]}m_T(C)s_C}.$$ Using the disjoint union~\eqref{equ:disjoint}, \Cref{prop:T-measure}, and \eqref{eqn:hls-SSYT}  we obtain
  \begin{align*}
\lefteqn{(1-q^{-1})^{-n} I_{n,\lri}(\bfs, -1,-2,\dots,-n)}\\
    &= (1 - q^{-1})^{-n}\sum_{T\in \SSYT_n} \int_{\mathscr{M}_{n,T}(\lri)} \prod_{\varnothing \neq C \subseteq [n]} \left(\dfrac{\left\| R_{n, C}^{\num} \right\| \left\| L_{n, C}^{\num}\right\|}{\left\| R_{n, C}^{\den} \right\|\left\| L_{n, C}^{\den} \right\|}\right)^{s_C} \prod_{i=1}^n |y_{ii}|^{-i} \,\mathrm{d}\mu \\
    &= (1 - q^{-1})^{-n}\sum_{T\in \SSYT_n} q^{\sum_{i=1}^n i\omega_{n-i+1}-\sum_{\varnothing \neq C\subseteq [n]}m_T(C)s_C} \mu(\mathscr{M}_{n,T}(\lri)) \\
    &= \sum_{T\in \SSYT_n} f_{n, T}^{\mathrm{in}}(\lri)\cdot q^{-\sum_{\varnothing \neq C\subseteq [n]}m_T(C)s_C} = \HLS_{n}(q^{-1}, (q^{\Schubdim{n}{C}-s_C})_C). \qedhere
  \end{align*}
\end{proof}

\Cref{prop:hls.integral} presents $\HLS_{n}(q^{-1},
(q^{\Schubdim{n}{C}-s_C})_C)$ in terms of the $\mathfrak{p}$-adic
integral~$I_{n,\lri}(\bfs,\bfs')$, whose integrand is a product of
maximal $\mathfrak{p}$-adic norms of sets of homogeneous polynomials
(of the same degree for each set). For the proof of \Cref{prop:funeq},
we record the degrees of the polynomial functions involved.

\begin{lem}\label{lem:four.deg}
  We have
  \begin{equation*}\label{equ:four.deg}\deg R_{n, C}^{\num} + \deg L_{n, C}^{\num} - \deg R_{n,
      C}^{\den} - \deg L_{n, C}^{\den} = \begin{cases} 1 & \textup{ if }C=[n],\\
      0 & \textup{ otherwise.}\end{cases}
    \end{equation*}  
  \end{lem}

  \begin{proof}
  By inspection of \eqref{def:four.poly} we find that
  \begin{align*}
    \deg R_{n, C}^{\num} &= \Schubdim{n}{[n] \setminus C} + 1, & \deg R_{n, C}^{\den}
    &= \Schubdim{n}{[n] \setminus C}, \\ \deg L_{n, C}^{\num} &= \max\{0,
    \Schubdim{n}{[n] \setminus C}+ n - \#C-1\}, & \deg L_{n, C}^{\den} &= \Schubdim{n}{[n] \setminus C} + n - \#C.
  \end{align*}
  Regarding $\deg L^{\num}_{n,C}$, observe that $\Schubdim{n}{[n]
    \setminus C}+ n - \#C-1 < 0$ if and only if $\Schubdim{n}{[n]
    \setminus C}+ n - \#C-1 = -1$. The latter is equivalent to $C =
  [n]$. Hence,
  \begin{align*}
    \deg L^{\num}_{n,[n]} &= \deg L^{\den}_{n,[n]}.\qedhere
  \end{align*}
\end{proof}

\begin{prop}\label{prop:funeq}
  For all $n$, there exists a finite set $S=S(n)$ of primes such that
  \begin{align*}
    \left. I_{n, \lri}(\bm{s}, \bm{s}') \right|_{q \to q^{-1}} &=
    q^{-s_{[n]} - \sum_{i=1}^n s_i'} I_{n, \lri}(\bm{s}, \bm{s}'),
  \end{align*}
  for all cDVRs $\lri$ with residue characteristic not contained in~$S$.
\end{prop}

\begin{proof} 
  We use \cite[Thm.~3.1]{MV/21} together with \Cref{lem:four.deg}. The latter
  asserts that the degree of the rational expression associated with the
  variable $s_C$ is zero unless~$C=[n]$, in which case it is one.
\end{proof}

\subsection{Proof of \Cref{thmabc:HLS.funeq}}\label{subsec:HLS.funeq.proof}
By \Cref{prop:hls.integral} it follows that, for the rings to which
\Cref{prop:funeq} applies,
  \begin{align*}
    \left.\HLS_{n}\left(q^{-1}, \left(q^{\Schubdim{n}{C}-s_C}\right)_C\right)\right|_{q\to q^{-1}} &= (-1)^nq^{-n} \cdot \left(\left.I_{n, \lri}(\bm{s}, -1,-2,\dots,-n) \right|_{q \to q^{-1}}\right)\\ 
    &= (-1)^nq^{-s_{[n]} + \binom{n}{2}}\cdot \HLS_{n}\left(q^{-1}, \left(q^{\Schubdim{n}{C}-s_C}\right)_C\right).\label{eqn:hls-functional}
  \end{align*}
  As this equation holds for infinitely many~$q$, it follows that
  \begin{align*}
    \HLS_n\left(Y^{-1}, \bfX^{-1}\right) &= (-1)^n Y^{-\binom{n}{2}} X_{[n]} \cdot \HLS_n\left(Y, \bm{X}\right). 
  \end{align*}
  This concludes the proof of \Cref{thmabc:HLS.funeq}. \qed

\begin{acknowledgements}
  We thank Claudia Alfes for the special role she played in giving
  this work its legs. We are grateful to Angela Carnevale and Alex
  Fink for discussions on Bruhat order. We also thank Darij Grinberg,
  Peter Littelmann, and Tobias Rossmann for insightful discussions,
  and Tomer Bauer for his comments. Voll was funded by the Deutsche
  Forschungsgemeinschaft (DFG, German Research Foundation) --
  Project-ID 491392403 -- TRR~358.
\end{acknowledgements}

\bibliography{MV2} \bibliographystyle{abbrv}

\begin{thebibliography}{10}

\bibitem{AMV2/24}
C.~Alfes, J.~Maglione, and C.~Voll.
\newblock {Ehrhart polynomials, Hecke series, and affine buildings}.
\newblock {\em S\'em. Lothar. Combin.}, 91B:Art. 84, pp.~12, 2024.

\bibitem{BS/18}
M.~Beck and R.~Sanyal.
\newblock {\em Combinatorial reciprocity theorems}, volume 195 of {\em Graduate Studies in Mathematics}.
\newblock American Mathematical Society, Providence, RI, 2018.
\newblock An invitation to enumerative geometric combinatorics.

\bibitem{BGS/22}
M.~N. Berman, I.~Glazer, and M.~M. Schein.
\newblock Pro-isomorphic zeta functions of nilpotent groups and {L}ie rings under base extension.
\newblock {\em Trans. Amer. Math. Soc.}, 375(2):1051--1100, 2022.

\bibitem{BjoernerBrenti/05}
A.~Bj\"{o}rner and F.~Brenti.
\newblock {\em Combinatorics of {C}oxeter groups}, volume 231 of {\em Graduate Texts in Mathematics}.
\newblock Springer, New York, 2005.

\bibitem{BW:Bruhat}
A.~Bj\"orner and M.~Wachs.
\newblock Bruhat order of {C}oxeter groups and shellability.
\newblock {\em Adv. in Math.}, 43(1):87--100, 1982.

\bibitem{BumpNakasuji/10}
D.~Bump and M.~Nakasuji.
\newblock Integration on {$p$}-adic groups and crystal bases.
\newblock {\em Proc. Amer. Math. Soc.}, 138(5):1595--1605, 2010.

\bibitem{BushnellReiner/80}
C.~J. Bushnell and I.~Reiner.
\newblock Zeta functions of arithmetic orders and {S}olomon's conjectures.
\newblock {\em Math. Z.}, 173(2):135--161, 1980.

\bibitem{CSV/24}
A.~Carnevale, M.~M. Schein, and C.~Voll.
\newblock Generalized {I}gusa functions and ideal growth in nilpotent {L}ie rings.
\newblock {\em Algebra Number Theory}, 18(3):537--582, 2024.

\bibitem{duSL/96}
M.~P.~F. du~Marcus and A.~Lubotzky.
\newblock Functional equations and uniformity for local zeta functions of nilpotent groups.
\newblock {\em Amer. J. Math.}, 118(1):39--90, 1996.

\bibitem{duSG/00}
M.~P.~F. du~Sautoy and F.~Grunewald.
\newblock Analytic properties of zeta functions and subgroup growth.
\newblock {\em Ann. of Math. (2)}, 152(3):793--833, 2000.

\bibitem{FM/16}
B.~Feigin and I.~Makhlin.
\newblock A combinatorial formula for affine {H}all--{L}ittlewood functions via a weighted {B}rion theorem.
\newblock {\em Selecta Math. (N.S.)}, 22(3):1703--1747, 2016.

\bibitem{Fulton/97}
W.~Fulton.
\newblock {\em Young tableaux}, volume~35 of {\em London Mathematical Society Student Texts}.
\newblock Cambridge University Press, Cambridge, 1997.

\bibitem{Gale/68}
D.~Gale.
\newblock Optimal assignments in an ordered set: {A}n application of matroid theory.
\newblock {\em J. Combinatorial Theory}, 4:176--180, 1968.

\bibitem{GSS/88}
F.~J. Grunewald, D.~Segal, and G.~C. Smith.
\newblock Subgroups of finite index in nilpotent groups.
\newblock {\em Invent. Math.}, 93(1):185--223, 1988.

\bibitem{KL/72}
S.~L. Kleiman and D.~Laksov.
\newblock Schubert calculus.
\newblock {\em Amer. Math. Monthly}, 79:1061--1082, 1972.

\bibitem{LV/23}
S.~Lee and C.~Voll.
\newblock Zeta functions of integral nilpotent quiver representations.
\newblock {\em Int. Math. Res. Not. IMRN}, (4):3460--3515, 2023.

\bibitem{LubotzkySegal/03}
A.~Lubotzky and D.~Segal.
\newblock {\em Subgroup growth}, volume 212 of {\em Progress in Mathematics}.
\newblock Birkh\"{a}user Verlag, Basel, 2003.

\bibitem{Macdonald/95}
I.~G. Macdonald.
\newblock {\em Symmetric functions and {H}all polynomials}.
\newblock Oxford Mathematical Monographs. The Clarendon Press, Oxford University Press, New York, second edition, 1995.

\bibitem{MV/21}
J.~Maglione and C.~Voll.
\newblock {Flag Hilbert--Poincar\'e series of hyperplane arrangements and their Igusa zeta functions}, 2021.
\newblock To appear in Israel J.\ Math. \href{https://arxiv.org/abs/2103.03640}{\texttt{arXiv:2103.03640}}.

\bibitem{PanchishkinVankov/07}
A.~Panchishkin and K.~Vankov.
\newblock Explicit {S}himura's conjecture for {${\rm Sp}_3$} on a computer.
\newblock {\em Math. Res. Lett.}, 14(2):173--187, 2007.

\bibitem{Petersen/15}
T.~K. Petersen.
\newblock {\em Eulerian numbers}.
\newblock Birkh\"{a}user Advanced Texts: Basler Lehrb\"{u}cher. Birkh\"{a}user/Springer, New York, 2015.

\bibitem{Rossmann/15}
T.~Rossmann.
\newblock Computing topological zeta functions of groups, algebras, and modules, {I}.
\newblock {\em Proc. Lond. Math. Soc. (3)}, 110(5):1099--1134, 2015.

\bibitem{Rossmann/18}
T.~Rossmann.
\newblock Computing local zeta functions of groups, algebras, and modules.
\newblock {\em Trans. Amer. Math. Soc.}, 370(7):4841--4879, 2018.

\bibitem{Zeta}
T.~Rossmann.
\newblock {\em {\textsf{\textup{Zeta}}, version 0.4}}, 2019.
\newblock See \url{http://www.maths.nuigalway.ie/~rossmann/Zeta/}.

\bibitem{Satake/63}
I.~Satake.
\newblock Theory of spherical functions on reductive algebraic groups over {$p$}-adic fields.
\newblock {\em Inst. Hautes \'{E}tudes Sci. Publ. Math.}, (18):5--69, 1963.

\bibitem{SV1/15}
M.~M. Schein and C.~Voll.
\newblock Normal zeta functions of the {H}eisenberg groups over number rings~{I}: the unramified case.
\newblock {\em J. Lond. Math. Soc. (2)}, 91(1):19--46, 2015.

\bibitem{SV2}
M.~M. Schein and C.~Voll.
\newblock Normal zeta functions of the {H}eisenberg groups over number rings {II}---the non-split case.
\newblock {\em Israel J. Math.}, 211(1):171--195, 2016.

\bibitem{OEIS:A003121}
N.~J.~A. Sloane.
\newblock {OEIS Sequence A003121}, 2021.
\newblock Accessed: 2024-10-03.

\bibitem{OEIS:A005118}
N.~J.~A. Sloane.
\newblock {OEIS Sequence A005118}, 2021.
\newblock Accessed: 2024-10-03.

\bibitem{Solomon/77}
L.~Solomon.
\newblock Zeta functions and integral representation theory.
\newblock {\em Advances in Math.}, 26(3):306--326, 1977.

\bibitem{Stanley:commutative}
R.~P. Stanley.
\newblock {\em Combinatorics and commutative algebra}, volume~41 of {\em Progress in Mathematics}.
\newblock Birkh\"{a}user Boston, Inc., Boston, MA, second edition, 1996.

\bibitem{Stanley:Vol2}
R.~P. Stanley.
\newblock {\em Enumerative combinatorics. {V}ol. 2}, volume~62 of {\em Cambridge Studies in Advanced Mathematics}.
\newblock Cambridge University Press, Cambridge, 1999.

\bibitem{StasinskiVoll/13}
A.~Stasinski and C.~Voll.
\newblock A new statistic on the hyperoctahedral groups.
\newblock {\em Electron. J. Combin.}, 20(3):Paper 50, 23, 2013.

\bibitem{sagemath}
{The Sage Developers}.
\newblock {\em {S}ageMath, the {S}age {M}athematics {S}oftware {S}ystem ({V}ersion 9.2)}, 2020.
\newblock \texttt{\url{https://www.sagemath.org}}.

\bibitem{Thrall/52}
R.~M. Thrall.
\newblock A combinatorial problem.
\newblock {\em Michigan Math. J.}, 1:81--88, 1952.

\bibitem{Vankov/11}
K.~Vankov.
\newblock Explicit {H}ecke series for symplectic group of genus 4.
\newblock {\em J. Th\'{e}or. Nombres Bordeaux}, 23(1):279--298, 2011.

\bibitem{Vince/00}
A.~Vince.
\newblock The greedy algorithm and {C}oxeter matroids.
\newblock {\em J. Algebraic Combin.}, 11(2):155--178, 2000.

\bibitem{Voll/04}
C.~Voll.
\newblock Zeta functions of groups and enumeration in {B}ruhat-{T}its buildings.
\newblock {\em Amer. J. Math.}, 126(5):1005--1032, 2004.

\bibitem{Voll/10}
C.~Voll.
\newblock Functional equations for zeta functions of groups and rings.
\newblock {\em Ann. of Math. (2)}, 172(2):1181--1218, 2010.

\bibitem{Voll/11}
C.~Voll.
\newblock A newcomer's guide to zeta functions of groups and rings.
\newblock In {\em Lectures on profinite topics in group theory}, volume~77 of {\em London Math. Soc. Stud. Texts}, pages 99--144. Cambridge Univ. Press, Cambridge, 2011.

\bibitem{Woodcock}
J.~Woodcock.
\newblock Properties of the poset of {D}yck paths ordered by inclusion, 2008.
\newblock \href{https://arxiv.org/abs/1011.5008}{\texttt{arXiv:1011.5008}}.

\end{thebibliography}

\begin{appendices}

\section{Examples}

We exemplify some of the paper's rational functions---all of which can be
obtained as substitutions of Hall--Littlewood--Schubert series for~$n\leq
3$. Further data may be found at the Zenodo repository under the following
link:
\begin{center}
  \href{https://zenodo.org/uploads/13895162}{\texttt{https://zenodo.org/uploads/13895162}}
\end{center}

\subsection{Affine Schubert series}
Recall \Cref{def:affS.int} of the affine Schubert series
$\affSin_{n,\lri}(\bfZ)$ of intersection type.

\begin{ex}[Intersection type]\label{exa:affS.int}
  We have
  \begin{alignat*}{2}
    \affSin_{1, \lri}(Z_{11}) &= \dfrac{1}{1 - Z_{11}}, & \quad
    \affSin_{2, \lri}(\bm{Z}) &= \dfrac{1 - Z_{11} Z_{21}}{(1 - q Z_{11})(1 - Z_{21})(1 - Z_{11} Z_{22})}.
  \end{alignat*}
  For $n=3$ we write $\affSin_{3,\lri}(\bfZ) = \mathsf{N}_{3,\lri}^{\mathrm{in}}(\bfZ) /
  \mathsf{D}_{3,\lri}^{\mathrm{in}}(\bfZ)$, where
  {\small
  \begin{align*}
    \mathsf{N}_{3,\lri}^{\mathrm{in}}(\bfZ) &= 1 - Z_{21}Z_{31}^2 - Z_{11}Z_{22}Z_{31}Z_{32} - Z_{11}Z_{21}^2Z_{32}^2 - q Z_{11}Z_{21}Z_{31}^2 \\ 
    &\quad + Z_{11}Z_{21}Z_{22}Z_{31}Z_{32}^2 - q Z_{11}Z_{21}Z_{22}Z_{32}^2 + Z_{11}Z_{21}^2Z_{31}^2Z_{32} - q Z_{11}Z_{21}^2Z_{31}Z_{32} \\ 
    &\quad + q Z_{11}Z_{21}Z_{22}Z_{31}^2Z_{32} + q Z_{11}Z_{21}^2Z_{31}^2Z_{32} - q^2Z_{11}Z_{21}^2Z_{31}Z_{32} + q Z_{11}Z_{21}^2Z_{31}^3 \\ 
    &\quad - q^2Z_{11}Z_{21}^2Z_{31}^2 + q Z_{11}^2Z_{21}Z_{22}Z_{31}Z_{32}^2 + q Z_{11}Z_{21}^3Z_{31}Z_{32}^2 - q^2Z_{11}^2Z_{21}Z_{22}Z_{32}^2 \\ 
    &\quad + q^2Z_{11}Z_{21}^2Z_{31}^3 + q Z_{11}^2Z_{21}^2Z_{22}Z_{32}^3 - q Z_{11}Z_{21}^3Z_{31}^3Z_{32} + q^2Z_{11}^2Z_{21}Z_{22}Z_{31}^2Z_{32} \\ 
    &\quad + q^2Z_{11}Z_{21}^3Z_{31}^2Z_{32} - q Z_{11}^2Z_{21}^2Z_{22}Z_{31}Z_{32}^3 + q^2Z_{11}^2Z_{21}^2Z_{22}Z_{32}^3 \\ 
    &\quad - q Z_{11}^2Z_{21}^2Z_{22}Z_{31}^2Z_{32}^2 + q^2Z_{11}^2Z_{21}^2Z_{22}Z_{31}Z_{32}^2 + q^2Z_{11}^2Z_{21}^3Z_{31}Z_{32}^2 \\ 
    &\quad - q^2Z_{11}^2Z_{21}^2Z_{22}Z_{31}^2Z_{32}^2 + q^3Z_{11}^2Z_{21}^2Z_{22}Z_{31}Z_{32}^2 - q^2Z_{11}^2Z_{21}^3Z_{31}^3Z_{32} \\ 
    &\quad + q^3Z_{11}^2Z_{21}^3Z_{31}^2Z_{32} - q^2Z_{11}^2Z_{21}^3Z_{22}Z_{31}Z_{32}^3 - q^3Z_{11}^2Z_{21}^2Z_{22}Z_{31}^3Z_{32} \\ 
    &\quad - q^3Z_{11}^2Z_{21}^4Z_{31}^2Z_{32}^2 - q^3Z_{11}^3Z_{21}^3Z_{22}Z_{31}Z_{32}^3 + q^3Z_{11}^3Z_{21}^4Z_{22}Z_{31}^3Z_{32}^3, \\
    \mathsf{D}_{3,\lri}^{\mathrm{in}}(\bfZ) &= (1 - Z_{31})(1 - Z_{21}Z_{32})(1 - Z_{11}Z_{22}Z_{33})(1 - q Z_{21}Z_{31})(1 - q Z_{11}Z_{21}Z_{32}) \\ 
    &\quad \times (1 - q^2Z_{11}Z_{22}Z_{32})(1 - q^2Z_{11}Z_{21}Z_{31}). 
  \end{align*}}
\end{ex}

Recall \Cref{def:affS.proj} of the affine Schubert series
$\affSpr_{n,\lri}(\bfZ)$ of projection type.

\begin{ex}[Projection type]\label{exa:affS.proj}
  We have
  \begin{alignat*}{2}
    \affSpr_{1, \lri}(Z_{11}) &= \dfrac{1}{1 - Z_{11}}, \quad
    \affSpr_{2, \lri}(\bm{Z}) &= \dfrac{1 - Z_{11}Z_{21}}{(1 - Z_{11})(1 - qZ_{21})(1 - Z_{11}Z_{22})}.
  \end{alignat*}
  For $n=3$ we write
  $\affSpr_{n,\lri}(\bm{Z}) = \mathsf{N}_{3,\lri}^{\mathrm{pr}}(\bfZ) / \mathsf{D}_{3, \lri}^{\mathrm{pr}}(\bfZ)$, where {\small
  \begin{align*}
    \mathsf{N}_{3,\lri}^{\mathrm{pr}}(\bm{Z}) &= 1 - Z_{11}Z_{22}Z_{31}Z_{32} - Z_{11}Z_{21}^2Z_{31}^2 - q Z_{11}Z_{21}Z_{31}^2 - q^2Z_{21}Z_{31}^2 \\ 
    &\quad - Z_{11}^2Z_{21}Z_{22}Z_{32}^2 - q Z_{11}Z_{21}Z_{22}Z_{32}^2 - q Z_{11}Z_{21}^2Z_{31}Z_{32} - q^2Z_{11}Z_{21}^2Z_{32}^2 \\ 
    &\quad + Z_{11}^2Z_{21}Z_{22}Z_{31}^2Z_{32} + q Z_{11}Z_{21}Z_{22}Z_{31}^2Z_{32} - q^2Z_{11}Z_{21}^2Z_{31}Z_{32} + q Z_{11}Z_{21}^2Z_{31}^3 \\ 
    &\quad + Z_{11}^2Z_{21}^2Z_{22}Z_{31}Z_{32}^2 + q Z_{11}^2Z_{21}Z_{22}Z_{31}Z_{32}^2 + q^2Z_{11}Z_{21}Z_{22}Z_{31}Z_{32}^2 \\ 
    &\quad + q^2Z_{11}Z_{21}^2Z_{31}^2Z_{32} + q^2Z_{11}Z_{21}^2Z_{31}^3 + q Z_{11}^2Z_{21}^2Z_{22}Z_{32}^3 + q Z_{11}^2Z_{21}^2Z_{22}Z_{31}Z_{32}^2 \\ 
    &\quad + q Z_{11}^2Z_{21}^3Z_{31}^2Z_{32} + q^2Z_{11}Z_{21}^3Z_{31}^2Z_{32} + q^3Z_{11}Z_{21}^2Z_{31}^2Z_{32} + q^2Z_{11}^2Z_{21}^2Z_{22}Z_{32}^3 \\ 
    &\quad - q Z_{11}^2Z_{21}^2Z_{22}Z_{31}^2Z_{32}^2 + q^2Z_{11}^2Z_{21}^3Z_{31}Z_{32}^2 + q^3Z_{11}Z_{21}^3Z_{31}Z_{32}^2 \\ 
    &\quad - q Z_{11}^2Z_{21}^2Z_{22}Z_{31}^3Z_{32} - q^2Z_{11}^2Z_{21}^2Z_{22}Z_{31}^2Z_{32}^2 - q^2Z_{11}^2Z_{21}^3Z_{31}^3Z_{32} \\ 
    &\quad - q^3Z_{11}Z_{21}^3Z_{31}^3Z_{32} - q Z_{11}^3Z_{21}^3Z_{22}Z_{31}Z_{32}^3 - q^2Z_{11}^2Z_{21}^3Z_{22}Z_{31}Z_{32}^3 \\ 
    &\quad - q^3Z_{11}^2Z_{21}^2Z_{22}Z_{31}Z_{32}^3 - q^3Z_{11}^2Z_{21}^4Z_{31}^2Z_{32}^2 + q^3Z_{11}^3Z_{21}^4Z_{22}Z_{31}^3Z_{32}^3, \\
    \mathsf{D}_{3,\lri}^{\mathrm{pr}}(\bm{Z}) &= (1 - Z_{11}Z_{22}Z_{33})(1 - Z_{11}Z_{22}Z_{32})(1 - Z_{11}Z_{21}Z_{31})(1 - q Z_{21}Z_{31})(1 - q^2Z_{31}) \\
    &\quad \times (1 - q Z_{11}Z_{21}Z_{32})(1 - q^2Z_{21}Z_{32}).
  \end{align*}}
\end{ex}

\subsection{Symplectic Hecke series}\label{app:hecke}
Recall the polynomials $\Hecke_n^{\textup{num}}(Y,\bfx,X)$ yielding the
numerators of the Hecke series $\Hecke_{n,\lri}(\bfx,X)$; see
\eqref{equ:hecke}.

\begin{ex}\label{exa:hecke}
  {\small
  \begin{align*}
    \Hecke_1^{\mathrm{num}}(Y, \bm{x}, X) &= 1,
    \\ \Hecke_2^{\mathrm{num}}(Y, \bm{x}, X) &= 1 - Y x_1x_2X^2,
    \\ \Hecke_3^{\mathrm{num}}(Y, \bm{x}, X) &= 1 - x_1x_2x_3X^2 - Y x_2x_3X^2 - Y x_1x_3X^2 - Y x_1x_2X^2 - Y x_1x_2x_3X^2 \\ 
    &\quad + Y x_1x_2x_3X^3 - Y x_1x_2x_3^2X^2 - Y x_1x_2^2x_3X^2 - Y x_1^2x_2x_3X^2 \\ 
    &\quad - Y^2x_1x_2x_3X^2 + Y x_1x_2x_3^2X^3 + Y x_1x_2^2x_3X^3 + Y x_1^2x_2x_3X^3 \\ 
    &\quad + Y^2x_1x_2x_3X^3 + Y x_1x_2^2x_3^2X^3 + Y x_1^2x_2x_3^2X^3 + Y^2x_1x_2x_3^2X^3 \\ 
    &\quad + Y x_1^2x_2^2x_3X^3 + Y^2x_1x_2^2x_3X^3 + Y^2x_1^2x_2x_3X^3 + Y x_1^2x_2^2x_3^2X^3 \\ 
    &\quad + Y^2x_1x_2^2x_3^2X^3 + Y^2x_1^2x_2x_3^2X^3 + Y^2x_1^2x_2^2x_3X^3 - Y x_1^2x_2^2x_3^2X^4 \\ 
    &\quad - Y^2x_1x_2^2x_3^2X^4 - Y^2x_1^2x_2x_3^2X^4 - Y^2x_1^2x_2^2x_3X^4 + Y^2x_1^2x_2^2x_3^2X^3 \\ 
    &\quad - Y^2x_1^2x_2^2x_3^2X^4 - Y^2x_1^2x_2^2x_3^3X^4 - Y^2x_1^2x_2^3x_3^2X^4 - Y^2x_1^3x_2^2x_3^2X^4 \\ 
    &\quad - Y^3x_1^2x_2^2x_3^2X^4 + Y^3x_1^3x_2^3x_3^3X^6.
  \end{align*}}
\end{ex}

\subsection{Hermite--Smith series}\label{app:HS}

Recall \Cref{def:HS} of the the Hermite--Smith series $\HS_{n,\lri}(\bfx,\bfy)$.

\begin{ex}
  We have 
  \begin{align*}
    \HS_{1,\lri}(x_1,y_1) &= \dfrac{1}{1 - x_1y_1}, &
    \HS_{2,\lri}(\bm{x},\bm{y}) &= \dfrac{1 - x_1^2y_1y_2}{(1 - x_1y_1)(1 - x_2y_1y_2)(1 - qx_1y_2)}.
  \end{align*}
  For $n=3$ we write $\HS_{3,\lri}(\bm{x}, \bm{y}) =
  \HS_{3,\lri}^{\mathrm{num}}(\bm{x},\bm{y})/\HS_{3,\lri}^{\mathrm{den}}(\bm{x},\bm{y})$,
  where {\small
  \begin{align*}
    \HS_{3,\lri}^{\mathrm{num}}(\bm{x},\bm{y}) &= 1 - x_1^2y_1y_2 - x_1x_2y_1y_2y_3 - q x_1^2y_1y_3 - x_2^2y_1^2y_2y_3 - q x_1x_2y_1y_2y_3 \\ 
    &\quad - q^2x_1^2y_2y_3 - q x_2^2y_1y_2^2y_3 + x_1^2x_2y_1^2y_2y_3 - q^2x_1x_2y_1y_2y_3 + q x_1^3y_1y_2y_3 \\ 
    &\quad - q^2x_2^2y_1y_2y_3^2 + x_1x_2^2y_1^2y_2^2y_3 + q x_1^2x_2y_1y_2^2y_3 + q x_1^2x_2y_1^2y_2y_3 \\ 
    &\quad + q^2x_1^3y_1y_2y_3 + q x_1x_2^2y_1^2y_2y_3^2 + q^2x_1^2x_2y_1y_2y_3^2 + q x_1x_2^2y_1^2y_2^2y_3 \\ 
    &\quad + q^2x_1^2x_2y_1y_2^2y_3 + q x_2^3y_1^2y_2^2y_3^2 + q^2x_1x_2^2y_1y_2^2y_3^2 + q^2x_1x_2^2y_1^2y_2y_3^2 \\ 
    &\quad + q^3x_1^2x_2y_1y_2y_3^2 - q x_1^3x_2y_1^2y_2^2y_3 + q^2x_2^3y_1^2y_2^2y_3^2 - q x_1^2x_2^2y_1^2y_2^2y_3^2 \\ 
    &\quad + q^3x_1x_2^2y_1y_2^2y_3^2 - q^2x_1^3x_2y_1^2y_2y_3^2 - q x_1x_2^3y_1^3y_2^2y_3^2 - q^2x_1^2x_2^2y_1^2y_2^2y_3^2 \\
    &\quad - q^3x_1^3x_2y_1y_2^2y_3^2 - q^2x_1x_2^3y_1^2y_2^3y_3^2 - q^3x_1^2x_2^2y_1^2y_2^2y_3^2 - q^3x_1x_2^3y_1^2y_2^2y_3^3 \\
    &\quad + q^3x_1^3x_2^3y_1^3y_2^3y_3^3, \\
    \HS_{3,\lri}^{\mathrm{den}}(\bm{x},\bm{y}) &= (1 - x_1y_1)(1 - x_2y_1y_2)(1 - q x_1y_2)(1 - x_3y_1y_2y_3)(1 - q x_2y_1y_3)\\
    &\quad \times (1 - q^2x_1y_3)(1 - q^2x_2y_2y_3).
  \end{align*}}
\end{ex}

\subsection{Quiver representation zeta functions}

Recall the definition \eqref{def:rep.zeta} of the zeta function
$\zeta_{V_n(\lri)}(\bfs)$ of the $\lri$-representation $V_n(\lri)$ of
the dual star quiver $\mathsf{S}_n^*$. 

\begin{ex}
  Set $t_i = q^{-s_i}$ for $i\in[n]$. We have
  \begin{align*}
    \zeta_{V_1(\lri)}(s_1) &= \dfrac{1}{1 - t_1}, & 
    \zeta_{V_2(\lri)}(\bm{s}) &= \dfrac{1 - t_1t_2^3}{(1 - t_2)(1 - t_2^2)(1 - t_1t_2^2)(1 - q t_1t_2)}.
  \end{align*}
  For $n=3$ we write $\zeta_{V_3(\lri)}(\bm{s}) =
  Z_{V_3(\lri)}^{\mathrm{num}}(\bm{s})/Z_{V_3(\lri)}^{\mathrm{den}}(\bm{s})$,
  where
  \begin{align*}
    Z_{V_3(\lri)}^{\mathrm{num}}(\bm{s}) &= 1 - t_2^2t_3^5 - q t_1t_2t_3^4 - t_1t_2^2t_3^5 - q^2t_1t_2^3t_3^3 - t_1t_2^3t_3^6 - q t_1t_2^3t_3^5 - q^2t_1t_2^3t_3^4 \\
    &\quad + q t_1t_2^3t_3^6 - q t_1t_2^4t_3^5 + t_1t_2^3t_3^8 + q t_1t_2^3t_3^7 + q^2t_1t_2^3t_3^6 - q^2t_1^2t_2^3t_3^5 + t_1t_2^4t_3^8 \\
    &\quad + q t_1t_2^4t_3^7 + q^2t_1^2t_2^3t_3^6 + q t_1^2t_2^3t_3^8 + q^2t_1t_2^5t_3^6 + q t_1t_2^5t_3^8 + q^2t_1^2t_2^4t_3^7 + q^3t_1^2t_2^4t_3^6 \\
    &\quad - q t_1t_2^5t_3^9 + q t_1^2t_2^5t_3^8 + q^2t_1^2t_2^5t_3^7 + q^3t_1^2t_2^5t_3^6 - q^2t_1^2t_2^4t_3^9 + q^2t_1^2t_2^5t_3^8 \\
    &\quad - q t_1^2t_2^5t_3^{10} - q^2t_1^2t_2^5t_3^9 - q^3t_1^2t_2^5t_3^8 - q t_1^2t_2^5t_3^{11} - q^3t_1^2t_2^6t_3^9 - q^2t_1^2t_2^7t_3^{10} \\
    &\quad - q^3t_1^3t_2^6t_3^9 + q^3t_1^3t_2^8t_3^{14}, \\ 
    Z_{V_3(\lri)}^{\mathrm{den}}(\bm{s}) &= (1 - t_2)(1 - t_3)(1 - q t_3)(1 - t_3^3)(1 - t_2^2t_3^3)(1 - q t_2^2t_3^2)(1 - q^2t_1t_2t_3) \\
    &\quad \times (1 - t_1t_2^2t_3^3)(1 - q t_1t_2t_3^3)(1 - q^2t_1t_2^2t_3^2)
  \end{align*}
\end{ex}

\end{appendices}

\end{document}